\documentclass{amsart}

\usepackage{times}
\usepackage{amsthm}
\usepackage{amsmath}
\usepackage{amssymb}
\usepackage{mathrsfs}
\usepackage{pb-diagram}
\usepackage{color}
\usepackage{extarrows}

\usepackage{hyperref}

\renewcommand\eqref[1]{(\ref{#1})} 

\newcommand\Ind{\operatorname{Ind}}
\newcommand{\tto}{\rightrightarrows}

\def\A{{\mathcal A}}
\def\o{\omega}

\def\th{\theta}

\def\T{\mathbb{T}}
\def\Z{\mathbb{Z}}

\def\H{\mathcal H}

\def\R{\mathbb{R}}

\def\S{\mathcal S}
\def\L{\sf L}

\def\e{{\sf e}}
\def\g{{\mathfrak g}}
\def\m{{\sf m}}

\def\r{{\rm r}}
\def\d{{\rm d}}

\def\({\left(}
\def\[{\left[}
\def\){\right)}
\def\]{\right]}

\def\si{\sigma}
\def\Si{\Sigma}

\def\G{{\sf G}}

\def\<{\langle}
\def\>{\rangle}

\def\NAME{totally intransitive}

\def\supp{\mathop{\mathrm{supp}}\nolimits}

\newtheorem{Theorem}{Theorem}[section]
\newtheorem{Remark}[Theorem]{Remark}
\newtheorem{Lemma}[Theorem]{Lemma}
\newtheorem{Assumption}[Theorem]{Assumption}
\newtheorem{Corollary}[Theorem]{Corollary}
\newtheorem{Proposition}[Theorem]{Proposition}
\newtheorem{Definition}[Theorem]{Definition}
\newtheorem{Example}[Theorem]{Example}

\newtheorem{Hypothesis}[Theorem]{Hypothesis}

\numberwithin{equation}{section}

\begin{document}


\title{Spectral Theory in a Twisted Groupoid Setting: Spectral Decompositions, Localization and 
Fredholmness}


\author[M. Mantoiu]{Marius M\u antoiu} \address{Departamento de Matem\'aticas, Facultad de Ciencias, Universidad de Chile, Santiago,
  Chile.}  \email{mantoiu@uchile.cl}

\author[V. Nistor]{Victor Nistor} \address{Universit\'{e} de Lorraine, UFR MIM, Ile du Saulcy, CS 50128, 57045 METZ, France and
  Inst. Math. Romanian Acad. PO BOX 1-764, 014700 Bucharest Romania}
\email{victor.nistor@univ-lorraine.fr}

\thanks{M. M. has been supported by the Fondecyt Project 1160359. V.N. has been partially supported by ANR-14-CE25-0012-01.\\
Manuscripts available from
\textbf{http:{\scriptsize//}iecl.univ-lorraine.fr{\scriptsize/}$\tilde{}$Victor.Nistor{\scriptsize
    /}}
}





\begin{abstract}
We study bounded operators defined in terms of the regular representations of the $C^*$-algebra of an amenable, Hausdorff, second countable locally compact groupoid endowed with a continuous
$2$-cocycle. We concentrate on spectral quantities associated to natural quotients of this twisted algebra, such as the essential spectrum, the essential numerical range, and Fredholm properties. We
obtain decompositions for the regular representations associated to units of the groupoid belonging to a free locally closed orbit, in terms of spectral quantities attached to points (or orbits) in the
boundary of this main orbit. As examples, we discuss various classes of magnetic pseudo-differential operators on nilpotent groups. We also prove localization and non-propagation properties associated to suitable parts of the essential spectrum. These are applied to twisted groupoids having a \NAME\ groupoid restriction at the boundary.
\end{abstract}

\date{\today}
\maketitle
\tableofcontents

\section{Introduction}\label{introduction}

We are setting in this paper some of the foundations of the study of magnetic Hamiltonians on non-flat spaces and their generalizations. The structure of these
Hamiltonians is rather complicated and their study thus requires a more sophisticated machinery, in particular, it requires a non-trivial use
of operator algebras. Our approach is to model these magnetic Hamiltonians (and others) using {\em twisted groupoids.} On a technical level, this
leads us to relate spectral properties of operators canonically associated to a groupoid $\Xi$ with units $X$ and endowed with a
$2$-cocycle $\omega$ to the structure of the groupoid and of the $2$-cocycle. This is technically difficult and innovative, but
important in applications. The pair $(\Xi,\omega)$ will be called a {\it twisted groupoid}, for short.

\smallskip
With a few exceptions, including \cite{BBdN, MN, MPR2}, twisted groupoids have not yet been used before in spectral theory. On the
other hand, regular (untwisted) groupoids have recently been used in relation to spectral theory (mainly in connection with Fredholm and index properties) by Androulidakis and Skandalis \cite{AS},
Debord and Skandalis \cite{DS1, DS2, DS3}, Debord, Lescure, and Rochon \cite{DLR}, Monthubert \cite{Mo1, Mo2}, van Erp and Yuncken
\cite{vEY1, vEY2}, the second named author (with collaborators) \cite{CNQ, LN, NP, NWX}, and by others. The groupoids arising from
crossed product $C^*$-algebras have been used even for a longer time. Probably the first one to notice the relevance of crossed-product $C^*$-algebras in spectral theory is Georgescu
\cite{dMG} in relation to the $N$-body problem. He has then developped a comprehensive approach to spectral theory using crossed-product
$C^*$-algebras, see, for instance, \cite{GI, GI1, Ge} and the references therein. A comprehensive related work (mostly in the
framework of magnetic pseudodifferential operators) is due to the first named author and collaborators, see \cite{LMR, Ma, Ma2} and the references
therein. Many other researchers have worked on similar problems, and it is a daunting task to provide a comprehensive overview of the
literature on the subject, so we content ourselves here to mention just a few of the most relevant references \cite{BLLS1, CCQ, HM, LR, LS, LauterMoroianu1, LS1, Li, Ra, RRS}.  We apologize for the many missing references. To the best of our knowledge, our results are new
in the stated generality; in particular, they are not contained in \cite{MN}. Some of them are new even in the untwisted case.

\smallskip
As we have mentioned above, our main potential applications are to Hamiltonians on non-flat spaces.  Non-flat spaces are becoming more
and more important in view of their role in Quantum Field theory on curved space-times \cite{BW, GW, JS}. To deal with the complications
that arise on non-flat space-times, we appeal to the approaches in \cite{CNQ, Come, Ma2}.

\smallskip
Let us now describe the structure of the article. We decided to give here in the introduction a leisurely account of the contents of
the paper, while keeping the technical terms and formulae to a minimum. The reader could return to this presentation as needed.

\smallskip
{\bf Section \ref{sec.two}} is dedicated to reviewing some basic constructions, starting with that of a locally compact groupoid
endowed with a continuos $2$-cocycle. We also prove some technical results that will be useful subsequently. We first introduce the class
of {\it tractable groupoids}; they are defined as amenable, Hausdorff, second countable, locally compact groupoids with a fixed Haar
system. For simplicity, we will keep these assumptions throughout the paper, even though they are not always needed. Some basic
constructions \cite{Re} are then briefly recalled, including that of {\it the twisted groupoid $C^*$-algebra}, the {\it regular
  representations} and the important {\it vector representation}, available as soon as the tractable groupoid is {\it standard}, i.e. it has an open dense orbit with trivial isotropy.

\smallskip
In the second subsection we recall the construction of {\it the groupoid extension} and summarize some results of Brown and an Huef
\cite{BaH} that will be essential in the sequel. They show that twisted groupoid $C^*$-algebras may be seen as (closed, two-sided,
self-adjoint) ideals and direct summands in the $C^*$-algebras of the corresponding extensions \cite{BaH}. This allows us to reduce some
issues concerning twisted groupoid $C^*$-algebras to the untwisted case.

\smallskip
Many of our spectral results will rely on intermediate results on regular representations of twisted groupoid $C^*$-algebras. In
Subsection \ref{caloriferit}, we deduce these intermediate results from the untwisted case, using the Brown-an Huef connection of the
previous subsection and the commutativity of the diagram \eqref{coliclor}. In particular, in the case of a standard groupoid
(see above Definition \ref{lista2} for the definition of standard groupoids) we check that the vector representation is faithful. An
important role is played by {\em Exel's property} \cite{CNQ} (see also \cite{Ex, Ro}), asking the reduced $C^*$-norm of an element to be
attained for some regular representation. In Corollary \ref{feminism}, we show that Exel's property is fulfilled for tractable twisted groupoids.

\smallskip
In Subsection \ref{guacamole}, we deal with the problem of {\it exactness} of the short sequence attached to the choice of a closed
invariant set of units, a topic very well understood in the untwisted case. Such a result including $2$-cocycles appeared recently in
\cite{BBdN}, but in a form that is not general enough for our purposes. So we indicate a proof fit to our setting, based of the
precise connection \cite{BaH} between twisted groupoid algebras and the algebras of the groupoid extension.

\smallskip
{\bf Section \ref{sec.three}} contains the abstract forms of our spectral results. The main novelty of these results stems from the
{\it twisted} groupoid setting, but, in certain cases, the result seems to be new even for trivial cocycles--at least in the stated generality.

\smallskip
The results of Subsection \ref{meni} explain how to use the units of a groupoid and the associated regular representations to study the
spectra of elements in the twisted groupoid algebra. An important role here is played by {\em quasi-orbits}, which are defined as closures of
orbits of the groupoid $\Xi$ acting on its units $X$. In fact, instead of considering all the units of the groupoid, one could just use a
covering of the set of units by quasi-orbits and the regular representations associated to one generic unit in each of the chosen
quasi-orbits. The proof relies on showing that our groupoid satisfies the Exel property, a fact proved in Subsection \ref{caloriferit}. This
property was introduced formally in \cite{CNQ}, but it had been implicitely used before in spectral theory in references such as
\cite{Ex,NP,Ro}, via the notion of {\it strictly norming family of representations}.

\smallskip
The next subsection contains decompositions of the spectra and (via Atkinson's Theorem) applications to Fredholm conditions. They are
obtained by considering the operators $H_z := \Pi_z(F)$ obtained from an element $F$ of the twisted groupoid algebra $C^*(\Xi, \omega)$ by
applying to it the regular representation $\Pi_z$, where $z \in X$. It is assumed that $z$ is a {\it regular unit}, meaning that its isotropy
group is trivial and its orbit is locally closed. The reduction to the quasi-orbit $\mathcal Q_x$ generated by $x$ is then a standard
groupoid. The closure of such an orbit contains additional units $x$, called ``marginal,'' each contributing a term ${\sf sp}(H_x)$ to the
decomposition of the essential spectrum of $H_z$. There could be one more contribution coming from the addition of the unit in the
non-unital case. In fact, it is enough to include a subfamily of marginal points if they generate a collection of quasi-orbits covering
the boundary of the closure. We obtain that the unions appearing in the decompositions of the essential spectrum are already closed and
the invertibility conditions that characterize Fredholmness are automatically uniform. See \cite{LiS,NP} for similar ideas and results.

\smallskip
In Subsection \ref{caloriforit}, we obtain results similar to those of the previous two subsections, but replacing ``spectrum'' with ``{\it
 numerical range}'' and ``essential spectrum'' with ``{\it essential numerical range}.'' This last notion is a numerical range computed
in a quotient by the ideal of compact operators. It turns out that the resulting decomposition of the numerical range no longer consists of
merely a union, but rather of the convex hull of the union. To the best of our knowledge, such results have not yet been considered in
the literature, not even for simpler (in particular, untwisted) cases. Once again, it is enough to use coverings by quasi-orbits of
the marginal (i.e. non-generic) part of a big quasi-orbit generated by a regular point.

\smallskip
In Subsection \ref{gomitte}, in the \'etale case, one shows the absence of the discrete spectrum (and the equality between the
numerical range and the essential numerical range) for the operator $H_x$ attached to an unit $x$ generating a minimal quasi-orbit and
satisfying an extra condition. This topic deserves extra investigation.

\smallskip
{\bf Section \ref{sec.four}} contains some examples that have as common feature the appearence of a {\it variable magnetic field} on a
connected, simply connected nilpotent Lie group. On a such a group $\G$, there is a pseudo-differential formalism
\cite{Gl1,Gl2,Man1,Man2,MR,Me} in terms of scalar symbols defined on $\G\times\mathfrak g^\sharp$, where $\mathfrak g$ is the Lie algebra
of $\G$ and $\mathfrak g^\sharp$ is its dual. In addition, there are explicit classes of $2$-cocycles attached to magnetic fields (smooth
closed $2$-forms on $\G$)\,. Note that in \cite[Sect.\,5]{MN}, another type of twisted pdeudo-differential calculus has been treated. It
works for rather general classes of groups but, instead of the dual of the Lie algebra, one uses the irreducible representation theory of the
group. We refer to \cite{MR,MR1} for connections between the two calculi.

\smallskip
In Subsection \ref{gerfomoreni} we are briefly recalling the method from \cite{BM} to construct a {\it twisted pseudo-differential
  theory}; it can be seen as a quantization of the cotangent bundle $T^*\G$\,. Then we provide an interpretation of this calculus in a
groupoid setting that allows us to apply spectral results from the previous section. The relevant groupoid is a transformation groupoid
twisted by a magnetic $2$-cocycle, which essentially recovers the twisted crossed product $C^*$-algebras used in \cite{PR1}. The images
of elements through regular groupoid representations become {\it magnetic pseudo-differential operators} after a composition with a
partial Fourier transform and a unitary equivalence involving gauge covariant choices of vector potentials.

\smallskip
In Subsections \ref{gerfomoni} and \ref{grafoni}, we are studying a related situation in which, instead of the entire nilpotent group, one
considers suitable subsets. In both cases one gets {\it compressions} of operators corresponding to the entire group, which in principle
makes the spectral analysis harder. The setup is that of {\it systems with partial} (but not global) {\it symmetry}.

\smallskip
In Subsection \ref{gerfomoni}, the magnetic pseudo-differential operators of \ref{gerfomoreni} are compressed to complements of
relatively compact sets of the group. It is shown that the essential spectral data (essential spectra, essential numerical ranges, and
Fredholmness) are the same for the compressed and the initial operators, although they act in different Hilbert spaces and there
seems to be no simple relative compactness argument to be used. Consequently, in the compressed case, we again obtain
decomposition formulae for the essential spectrum. For the proof one uses magnetically twisted versions of groupoids corresponding to {\it
  a partial action} \cite{Ab,Ex1,Ex2}; they can also be seen as {\it non-invariant} reductions of the twisted transformation groupoid of
the previous subsection. The results seem to be interesting even for null magnetic field, which corresponds to the Abelian group $\G=\mathbb R^n$ (without any cocycle).

\smallskip
In Subsection \ref{grafoni}, we use the positive semigroup $\sf H$ of the three-dimensional Heisenberg group $\G$ to define the
compression. If no cocycle is used (i.e. $\omega = 1$), one gets the Wiener-Hopf-type operators and $C^*$-algebras studied by A. Nica in
\cite{Nic2}; see also \cite{MuR,Nic1} for a wider context. We add a $2$-cocycle, defined by a variable magnetic field. To reduce the
spectral analysis of the resulting {\it magnetic Wiener-Hopf operators} to a direct applications of the results in Section
\ref{sec.three}, we borrow from \cite{Nic2} the groupoid model of the Wiener-Hopf $C^*$-algebras; implementing the magnetic $2$-cocycle is a
simple matter. New features are now present. First, in \ref{gerfomoreni} and \ref{gerfomoni}, the unit space was at our disposal; it modeled the behavior at infinity of the magnetic field
and of the ``coefficients'' of the pseudo-differential symbols we decided to study. Since now neither $\sf H$ nor $\G\!\setminus\!\sf H$
are compact, the unit space will be a well-chosen compactification $X$ of $\sf H$\,, the geometry of the subset $\sf H$ imposing rigid
requirements on this choice; then the magnetic field has to adapt itself to $X$. As a consequence, the quasi-orbit structure of the unit
space is now explicit: there are six quasi-orbits, disposed in layers, but just two of them are enough to cover the boundary
$X\!\setminus\!\sf H$\,. Thus, besides decomposition formulae using the (adapted) regular representations of all the points of this
boundary, one also has simpler decompositions in terms of two units, generating the two quasi-orbits. A certain higher-dimensional
Heisenberg group is also computable from this point of view \cite[Sect.\,5]{Nic2}. The quasi-orbit structure being quite complicated, we decided not to include it in this paper.

\smallskip
{\bf Section \ref{sec.five}} is dedicated to {\it localization and non-propagation properties}. Let $H$ be a bounded normal operator in
$\H:=L^2(M;\mu)$\,. For any continuous real function $\kappa$ we denote by $\kappa(H)$ the normal operator in $\H$ constructed via the
functional calculus. If $\Psi:M\to\R$ is a bounded measurable function, we use the same symbol for the operator of multiplication by
$\Psi$ in $\H$\,. There is a single obvious inequality
\begin{equation}\label{apriori}
\|\Psi\kappa(H)\|_{\mathbb B(\H)}\,\le\,\sup_{x\in M}|\Psi(x)|\sup_{\lambda\in{\sf sp}(H)}\!|\kappa(\lambda)|
\end{equation}
that holds without extra assumptions. We treat situations in which the left hand side is small
\begin{equation}\label{aposteriori}
\|\Psi\kappa(H)\|_{\mathbb B(\H)}\,\le\,\epsilon
\end{equation}
for some $\epsilon>0$ given in advance, although in the right hand side of \eqref{apriori} the two factors are (say) equal to $1$\,.
Such results have been obtained in \cite{DS} (a very particular case) and in \cite{AMP,MPR2}, for operators $H$ deduced from a dynamical
system defined by the action of an Abelian locally compact group $\G$ on compactifications of $\G$\,. Here we investigate the problem in the
framework of the much more general twisted groupoid $C^*$-algebras.

\smallskip
The main abstract result is proven in subsection \ref{homeny}. The objects of our investigation are normal elements (or multipliers) $F$
of the standard twisted groupoid algebra and operators $H_0\!:=\Pi_0(F)$ acting in $L^2(M;\mu)$ via the vector representation. An important role is played by "the region at
infinity" $X_\infty:=X\!\setminus\!M$. As in previous sections, to every quasi-orbit $\mathcal Q$ contained in $X_\infty$ one associates
an element $F_\mathcal Q$ in the twisted $C^*$-algebra of the reduced groupoid $\Xi_\mathcal Q$\,, with spectrum contained in the essential
spectrum of the operator $H_0$\,. The function $\kappa$ in \eqref{aposteriori} has to be supported away from ${\sf sp}(F_\mathcal Q)$\,.  
Then the traces on $M$ of small neighborhoods $W$ of $\mathcal Q$ in the unit space $X$ define the functions $\Psi$ admitted in
\eqref{aposteriori}. In terms of the evolution group attached to $H_0$\,, informally, one may say that "at energies not belonging to
the spectrum of the asymptotic observable $F_\mathcal Q$\,, propagation towards the quasi-orbit $\mathcal Q$ is very improbable".

\smallskip
In the final subsection we outline a class of examples, not deriving from group actions, leading finally to a situation in which both the
nature and spectrum of the asymptotic observable $F_\mathcal Q$ and the neighborhood of the corresponding quasi-orbit are transparent
enough. The unit space $X$ is built as a compactification of an a priori given $M$, defined by a continuous surjection from the
complement of a compact subset of $M$ to a compact space $X_\infty$\,. It is also assumed that the restriction of the groupoid
to $X_\infty$ is a {\it \NAME} groupoid. Then the observable corresponding to the orbits $\mathcal Q\subset X_\infty$ are group
convolution operators, converted by a Fourier transform in multiplication operators in the Abelian case.

\smallskip
We intend to continue this investigation in a future publication. Besides aiming at other types of spectral properties, one
would like to include unbounded operators and new classes of examples. To have interesting examples in unbounded cases, one
essentially has to show that the resolvent family of a an interesting operator belongs to some twisted groupoid $C^*$-algebra. Besides some
particular situations, this is a difficult problem. The great achievement would be the inclusion of {\it twisted pseudo-differential
  operators on Lie groupoids}; we refer to \cite{ALN,DS,NWX,LMN} and references therein for the case $\o=1$\,. The twisted case will need
further theoretical investigation. We also intend to study some discrete systems leading to \'etale groupoids endowed with cocycles.

\medskip
\textbf{Acknowledgements:} This work initiated during a pleasant and fruitful visit of M. M. to Universit\'e de Lorraine in Metz.

\section{Twisted groupoid $C^*$-algebras}\label{sec.two}

\subsection{Groupoids endowed with cocycles and their $C^*$-algebras}\label{goamol}

See \cite{CNQ, Re, MN} for background and basic definitions concerning groupoids. In particular, we shall write $\Xi \tto X$ for a groupoid
$\Xi$ with units $X$. Recall from \cite{MN} that an {\em admissible groupoid} is a Hausdorff, locally compact groupoid $\Xi$ with unit
space $\Xi^{(0)}\!\equiv X$, with family $\Xi^{(2)}\!\subset\Xi\times\Xi$ of composable pairs, and with a fixed right Haar system $\lambda := \{\lambda_x\!\mid\!x\in X\}$. The
associated left Haar system $\{\lambda^x\!\mid\!x\in X\}$ is the one obtained from $\lambda$ by composing with the inversion
$\xi\to\iota(\xi)\equiv \xi^{-1}$\,. We will use this notation throughout the paper. We are also going to assume usually that $\Xi$
is second countable and amenable. For general concepts pertaining to groupoids, in particular, for groupoid amenability, we refer to
\cite{ADR,Re}. We agree to identify the units $X$ to the correspoinding subset of $\Xi$: that is, $X \subset \Xi$.

\smallskip
We shall write $x\approx y$ if $x=\r(\xi)$ and $y={\rm d}(\xi)$ for some $\xi\in\Xi$\,.  Then $\approx$ is an equivalence relation on
$X$. The equivalence classes of this relation are called {\it orbits (of $\,\Xi$ on $X$)}, and a subset which is a union of orbits is called $\Xi$-{\it invariant} or {\it saturated}.

\begin{Definition}\label{lista1}
Let $(\Xi, \lambda, \omega)$ be a locally compact twisted groupoid with Haar system $\lambda$.
\begin{enumerate}
\item[(i)] The groupoid $\Xi$ is called {\rm transitive} if its unit space $X$ has only one orbit, i.e. $x\approx y\,$ for any $x,y\in X$.
\item[(ii)] The groupoid $\Xi$ is called {\rm topologically transitive} if $X$ has a dense open orbit $M$.
\end{enumerate}
\end{Definition}

It is easy to see that if $\Xi \tto X$ is a topologically transitive groupoid, then it has a unique dense orbit, called the {\em main
  orbit.} Moreover, the main orbit is the unique open $\Xi$-invariant subset of $X$.

\smallskip
It will be convenient to use the following concept.

\begin{Definition}\label{terminology}
A {\em tractable groupoid} is an amenable, second countable, Hausdorff, locally compact groupoid $\Xi$ with a fixed right Haar system $\lambda=\{\lambda_x\!\mid\!x\in X\}$\,.
\end{Definition}

Denoting by ${\rm d,r}:\Xi\to \Xi^{(0)}$ {\it the source and range maps}, one defines {\it the ${\rm r}$-fibre} $\Xi^x$, {\it the $\d$-fibre} $\Xi_x$, and {\it the isotropy group}
$\Xi_x^x:=\Xi_x\cap\Xi^x$ of a unit $x\in X$. More generally, for $A,B\subset X$ one sets
\begin{equation*}
\Xi_A := {\rm d}^{-1}(A)\,,\quad\Xi^B\! := {\rm r}^{-1}(B)\,, \quad \Xi_A^B := \Xi_A\cap\Xi^B.
\end{equation*}

\begin{Definition}\label{asteluza}
{\rm A $2$-cocycle} \cite{Re} is a continuous function $\,\o:\Xi^{(2)}\to\T:=\{z\in\mathbb C\mid |z|=1\}$ satisfying
\begin{equation}\label{doicocic}
\o(\xi,\eta)\o(\xi\eta,\zeta) = \o(\eta,\zeta)\o(\xi,\eta\zeta)\,,\quad\ \forall\,(\xi,\eta) \in \Xi^{(2)},\ (\eta,\zeta)\in\Xi^{(2)},
\end{equation}
\begin{equation}\label{normaliz}
\o(x,\eta) = 1 =\o(\xi,x)\,,\quad\ \forall\,\xi,\eta\in\Xi\,,\,x\in X,\;\r(\eta)= x =\d(\xi)\,.
\end{equation}
\end{Definition}

Following \cite{Re}, if $(\Xi,\lambda,\o)$ is given, a $^*$-algebra structure is defined on the vector space $C_c(\Xi)$ of all continuous,
compactly supported functions $f:\Xi\to\mathbb C$ with product
\begin{align}
(f\star_\o\!g)(\xi) & :=\int_{\Xi}f(\eta)g(\eta^{-1}\xi)\,\o(\eta,\eta^{-1}\xi)\,d\lambda^{\r(\xi)}(\eta)\\
& = \int_{\Xi}f(\xi\eta)g(\eta^{-1})\,\o(\xi\eta,\eta^{-1})\,d\lambda^{\d(\xi)}(\eta)
\end{align}
and the involution
\begin{equation*}
f^{\star_\o}(\xi) := \overline{\o(\xi,\xi^{-1})}\,\overline{f(\xi^{-1})}\,.
\end{equation*}
Then, by the usual completion procedure based on {\em all} bounded $*$-representations of $\Xi$ \cite{Re}, one gets {\it the (full) twisted groupoid $C^*$-algebra} ${\sf C}^*(\Xi,\o)$\,.

\smallskip 
Given a measure $\nu$ on $X$, one defines the measure $\Lambda_\nu:=\int_X\!\lambda_x d\nu(x)$ on $\Xi$ by
\begin{equation*}
\int_\Xi f(\xi) d\Lambda_\nu(\xi) := \int_X\!\Big[\int_{\Xi_x}\!f(\xi)d\lambda_x(\xi)\Big]d\nu(x)\,,\quad\forall\,f\in C_{\rm c}(\Xi)\,.
\end{equation*}
We say that the measure $\nu$ on $X$ is {\em quasi-invariant} if the measures $\Lambda_\nu$ and $\Lambda^{-1}_\nu\!:=\Lambda_\nu\circ\iota$
are equivalent.  Given a quasi-invariant measure $\nu$ on $X$, we associate to it {\it the induced representation}
\begin{equation*}
\Ind_\nu:{\sf C}^*(\Xi,\o) \to \mathbb B\big[L^2(\Xi;\Lambda_\nu)\big]\,,\quad{\rm Ind}_\nu(f)u:=f\star_\o\!u\,.
\end{equation*}
Let $x\in X$ be a unit and $\nu=\delta_x$ the Dirac measure concentrated at $x$. Then, the associated induced representation is
the {\it regular representation} $\Pi_x := \Ind_{\delta_x} :{\sf C}^*(\Xi,\o) \to \mathbb B \big[L^2(\Xi_x;\lambda_x)\big]$. It is thus defined by
\begin{equation*}
\Pi_x(f)u := f\ast_\o\!u\,,\quad\forall\,f\in C_{\rm c}(\Xi)\,,\ u\in L^2(\Xi_x;\lambda_x) =: \H_x\,.
\end{equation*}
The regular representations serve to define the reduced norm
\begin{equation}\label{redunorm}
\|\cdot\|_{\rm r}\,:C_{\rm c}(\Xi)\to\R_+\,,\quad\|f\|_{\rm r}\,:=\sup_{x\in X}\|\Pi_x(f)\|_{\mathbb B(\H_x)}\,.
\end{equation}

\begin{Remark}\label{convention}
\normalfont{Recall that we assume our groupoids to be amenable. It is then known \cite[Sect.\,4]{BaH} that the canonical surjection ${\sf C}^*(\Xi,\o) \to {\sf C}_{\rm r}^*(\Xi,\o)$ from the full
  $C^*$-algebra to the reduced $C^*$-algebra associated to $(\Xi,\lambda, \omega)$ is injective, and hence an isomorphism. {\it
We shall thus identify the two algebras in what follows and use only the notation ${\sf C}^*(\Xi,\o)$\,. The reduced norm, which
    thus coincides the full one, will be denoted by $\|\cdot\|_{{\sf C}^*(\Xi,\o)}$ or simply by $\|\cdot\|$\,.}}
\end{Remark}

\begin{Definition}\label{aciush}
We shall say that a $2$-cocycle $\omega$ is a {\it coboundary}, if it is of the form
\begin{equation}\label{coboundary}
  \o(\xi,\eta) = [\delta^1(\si)](\xi,\eta)
  :=\si(\xi)\si(\eta)\si(\xi\eta)^{-1}
\end{equation}
for some continuous map $\si:\Xi\to\T$\,. Two 2-cocycles $\o_1$ and
$\o_2$ are called {\em cohomologous} if $\o_2=\delta^1(\si)\o_1$.
\end{Definition}

\begin{Remark}\label{acihush}
\normalfont{It is easy to see (and well-known) that to cohomologous $2$-cocycles $\o_1$ and $\o_2$, $\o_2=\delta^1(\si)\o_1$, there correspond isomorphic $C^*$-algebras ${\sf C}^*(\Xi,\o_2)\cong{\sf C}^*(\Xi,\o_1)$\,. If $\o(\xi,\eta)=1$ for any $\xi,\eta\in\Xi$\,, we shall simply write ${\sf C}^*(\Xi):={\sf C}^*(\Xi,\mathbf 1)$\,.}
\end{Remark}

\begin{Remark}\label{jiek}
\normalfont{It follows easily that, for every $\xi \in \Xi$, one has the unitary equivalence $\,\Pi_{\r(\xi)}\approx\Pi_{\d(\xi)}$\,, so
  the regular representations along an orbit are all unitarily equivalent:
\begin{equation*}
   x\approx y\,\Rightarrow\,\Pi_x\approx\Pi_y\,.
\end{equation*}}
\end{Remark}

We shall need the following classes of groupoids. Recall that all our groupoids are assumed to be tractable, see Definition \ref{terminology}.

\begin{Definition}\label{lista2}
Let $(\Xi, \lambda, \omega)$ be a tractable twisted groupoid. If $\,\Xi$ is topologically transitive and the isotropy of the main orbit
is trivial (i.e. $\,\Xi^z_z=\{z\}\,$ for one, equivalently for all $z$ in the main orbit $M$)\,, we shall say that $\Xi$ is {\rm standard}.
\end{Definition}

\begin{Example}\label{cuartet}
\normalfont{The standard transitive groupoids $\Xi \tto X$ are precisely the pair groupoids $\Xi = X \times X$. By \cite[Remark 2.2]{BaH}, on such groupoids all $2$-cocycles are trivial
  (coboundaries). The (twisted) groupoid algebra of a pair groupoid is elementary, i.e. isomorphic to the $C^*$-algebra of all the compact operators on a Hilbert space.}
\end{Example}

If $A$ is invariant, $\Xi_A^A=\Xi_A=\Xi^A$ is a subgroupoid, called {\it the restriction of $\,\Xi$ to $A$}\,. If $A$ is also locally
closed (the intersection between an open and a closed set), then $A$ is locally compact and $\Xi_A$ {\it will also be a tractable
  groupoid}, on which one automatically considers the restriction of the Haar system.

\begin{Remark}\label{idem}
\normalfont{Assume that the groupoid $\Xi$ is standard with main orbit $M$. By Example \ref{cuartet}, the restriction $\o_M := \o\vert_{M}$
  of a $2$-cocycle $\o$ on $\Xi$ to $M$ must be a coboundary. In particular ${\sf C}^*_{\rm r}(\Xi_M,\o_M)={\sf C}^*(\Xi_M,\o_M)\cong{\sf C}^*(M\times M)$ is an elementary
  $C^*$-algebra. In addition, for each $z\in M$, the restriction
\begin{equation*}
\r_z\!:=\r|_{\Xi_z}\!:\Xi_z=\Xi(X)_z\to M
\end{equation*} 
is surjective (since $M$ is an orbit) and injective (since the isotropy is trivial). Thus one transports the measure $\lambda_z$ to a
(full Radon) measure $\mu$ on $M$ (independent of $z$\,, by the invariance of the Haar system) and gets a representation $\Pi_0:{\sf
  C}^*(\Xi,\o)\to\mathbb B\big[L^2(M,\mu)\big]$\,, called {\it the vector representation} \cite{LN}. Let
\begin{equation*}
  {\sf R}_z:L^2(M;\mu)\to L^2\big(\Xi_z;\lambda_z\big)\,,\quad {\sf R}_z(v):=v\circ r_z\,.
\end{equation*}
Then $\Pi_0$ turns out to be unitarily equivalent to $\Pi_z$:
\begin{equation}\label{ibidem}
\Pi_0(f)={\sf R}_z^{-1}\Pi_z(f){\sf R}_z\,,\quad\Pi_0(f)v=\big[f\star_\o\!(v\circ \r_z)\big]\circ\r_z^{-1}.
\end{equation}}
\end{Remark}

\subsection{Groupoid extensions}\label{coforit}

We shall now proceed to related twisted groupoids and their $C^*$-algebras to untwisted groupoids, following a well-known reduction idea, see \cite{BaH, MW, Re} and the references therein.

\begin{Definition}\label{grextension}
Let $\o$ be a $2$-cocycle on the groupoid $\Xi$\,. {\rm The $\o$-extension of $\,\Xi$ by $\T$} will be denoted by $\Xi^\o$ or by
$\T\times^\o\Xi$\,.  As a topological space it is $\T\times\Xi$ with the product topology. The structural maps are
\begin{equation}\label{structu}
\r^\o(s,\xi) := (1,\r(\xi))\,,\quad \d^\o(s,\xi):=(1,\d(\xi))\,,
\end{equation}
\begin{equation*}\label{ral}
  (s,\xi)(t,\eta) :=(st\,\o(\xi,\eta),\xi\eta)\,,\quad(s,\xi)^{-1}\!:=\big(s^{-1}\o(\xi,\xi^{-1})^{-1},\xi^{-1}\big)\,.
\end{equation*}
\end{Definition}

We refer to \cite{MW} for connections with the general groupoid extension theory.

\begin{Remark}\label{adaugita}
\normalfont{For $\{1\}\times B\equiv B\subset X=\Xi^{(0)}\equiv\big(\Xi^\o\big)^{(0)}$ one has
\begin{equation}\label{umflat}
(\Xi^\o)_B = \T \times \Xi_B\,, \quad (\Xi^\o)^B = \T \times \Xi^B,\quad(\Xi^\o)_B^B = \T \times \Xi_B^B\,.
\end{equation}}
\end{Remark}

\begin{Lemma}\label{abercorn}
Let $(\Xi,\o)$ be a tractable groupoid endowed with a continuous $2$-cocycle.
\begin{enumerate}
\item[(i)] The extension $\Xi^\o$ is also tractable.
\item[(ii)] If $\,\Xi$ is topologically transitive, then $\Xi^\o$ is also topologically transitive.
\item[(iii)] If $\,\Xi$ is standard, then $\Xi^\o$ is topologically
  transitive and the isotropy groups of any point of the main orbit may be identified with the torus $\T$\,.
\end{enumerate}
\end{Lemma}

\begin{proof}
(i) Clearly $\Xi^\o$ is also locally compact, Hausdorff and second countable. On $\Xi^\o\!$ we can consider the right Haar system
  $\{\lambda_x^\o:=dt\times\lambda_x\mid x\in X\}$\,, where $dt$ is the normalized Haar measure on the torus. By
  \cite[Prop.\,5.1.2]{ADR}, if $\Xi$ is amenable the extension $\Xi^\o$ will also be amenable.

\smallskip(ii) The map $\,x\mapsto(1,x)$ identifies $X$ with the unit space of the $\o$-extension. By \eqref{structu}, the orbits are the
same as before, i.e. $x\overset{\Xi}{\approx}y$ if and only if $x\!\overset{\;\Xi^\o}{\approx}\!y$\,. In particular, the main orbits are the same.

\smallskip
(iii) From \eqref{structu} or \eqref{umflat} it follows that the isotropy groups of the extension have all the form $\T\times\Xi^x_x$
for some $x\in X$. So, if $\Xi$ is standard and $z\in M$, then $(\Xi^\o)_z^z=\T\times\{z\}$\,.
\end{proof}

Let us recall here the constructions of Brown and an Huef \cite{BaH}, which is a main technical tool in what follows. For every $n\in\Z$, let us define the $n$th homogeneous component
\begin{equation*}
  C_{\rm c}(\Xi^\o|n) := \big \{ \Phi \in C_{\rm c}(\Xi^\o) \mid\Phi(ts,\xi) = t^{-n} \Phi(s,\xi) \,,\, \forall\,s,t \in \T\,,\,\xi\in\Xi\big\}\,,
\end{equation*}
which is a $^*$-subalgebra of $C_{\rm c}(\Xi^\o)$ with the convolution product. The usual $C^*$-completion leads by definition to ${\sf C}^*(\Xi^\o|n)$\,. The map
\begin{equation*}
  \kappa^n:C_{\rm c}(\Xi^\o|n)\to C_{\rm c}(\Xi,\o^n)\,,\quad\big[\kappa^n(\Phi)\big](\xi):=\Phi(1,\xi)
\end{equation*}
with inverse
\begin{equation}\label{inviersa}
  (\kappa^n)^{-1}:C_{\rm c}(\Xi,\o^n)\to C_{\rm c}(\Xi^\o|n)\,,\quad\big[(\kappa^n)^{-1}(f)\big](t,\xi):=t^{-n}f(\xi)
\end{equation}
is a $^*$-algebra isomorphism that extends to the associated $C^*$-algebras ${\sf C}^*(\Xi^\o|n)\cong{\sf C}^*(\Xi,\o^n)$\,.

\smallskip
In \cite{BaH} it is also proved that the map
\begin{equation*}
  \chi^n:C_{\rm c}(\Xi^\o)\to C_{\rm c}(\Xi^\o|n)\,,\quad[\chi^n(\Phi)](s,\xi):=\int_\T \Phi(ts,\xi)t^n dt
\end{equation*}
extends to an epimorphism $\,\chi^n\!:{\sf C}^*(\Xi^\o)\to{\sf C}^*(\Xi^\o|n)$ that is the identity on $C_{\rm c}(\Xi^\o|n)$\,.
Composing the isomorphism $\kappa^1$ to the left with the epimorphism $\chi^1$, one already sees that ${\sf C}^*(\Xi,\o)$ is a quotient of ${\sf C}^*\big(\Xi^\o\big)$\,.

\smallskip
Although the two $C^*$-norms in the next equation are computed via different sets of representations, the result of \cite{BaH}[Lemma 3.3] gives that
\begin{equation*}
  \|\Phi\|_{{\sf C}^*(\Xi^\o|n)}\, = \,\|\Phi\|_{{\sf C}^*(\Xi^\o)}\,,\quad\forall\,\Phi\in C_{\rm c}(\Xi^\o|n)\,.
\end{equation*}
The closure ${\sf J}^n$ of $C_{\rm c}(\Xi^\o|n)$ in ${\sf C}^*(\Xi^\o)$ is a (closed bi-sided self-adjoint) ideal in ${\sf C}^*(\Xi^\o)$ and $\chi^n$ can be seen as an isomorphism between
${\sf J}^n$ and ${\sf C}^*(\Xi^\o|n)$\,. It is also shown that these ideals provide a direct sum decomposition
\begin{equation*}\label{mimi}
  {\sf C}^*(\Xi^\o)=\bigoplus_{n\in\Z}{\sf J}^n\cong\bigoplus_{n\in\Z}{\sf C}^*(\Xi^\o|n)\cong\bigoplus_{n\in\Z}{\sf C}^*(\Xi,\o^n)\,.
\end{equation*}
One of the conclusions is that {\it the twisted groupoid $C^*$-algebra embeds as a direct summand in the $C^*$-algebra of the extended
  groupoid: ${\sf C}^*(\Xi;\o)\hookrightarrow{\sf C}^*(\Xi^\o)$} (as the homogeneous component of order $1$).


\begin{Remark}\label{rem.psdo}
\normalfont{Assume that $\Xi \tto X$ is a Lie groupoid \cite{Mac} or, more generally, a longitudinally smooth groupoid \cite{LMN}. Then
  $\Xi^\o$ is also a Lie groupoid (respectively, a longitudinally smooth groupoid). This allows us to consider the pseudo-differential calculus (algebra) $\Psi^{\infty}(\Xi^\o)$ \cite{Mo1, Mo2, NWX,
    LMN}. Using the formalism above, one should be able to introduce and study a twisted pseudo-differential algebra
  $\Psi^{\infty}(\Xi;\o)$\,, such that $\Psi^{0}(\Xi;\o)\subset{\sf C}^*(\Xi,\o)$\,. The action of $\,\T$ extends to an action on
  $\Psi^{\infty}(\Xi^\o)$ and we define $\Psi^{\infty}(\Xi; \o) \simeq p_1 \Psi^{\infty}(\Xi^\o) = \Psi^{\infty}(\Xi^\o) p_1$.}
\end{Remark}

\subsection{Properties of representations}\label{caloriferit}

We have to deal with non-unital algebras. If $\mathscr C$ is a $C^*$-algebra, generic {\it unitizations} of $\mathscr C$ are denoted
by $\mathscr C^{\bf U}$; they are unital $C^*$algebras containing $\mathscr C$ as an essential ideal.  Any non-degenerate representation
$\Pi:\mathscr C\to\mathbb B(\H)$ extends uniquely to a representation $\Pi^{\bf U}\!:\mathscr C^{\bf U}\to\mathbb B(\H)$\,; if $\Pi$ is
injective, $\Pi^{\bf U}$ is also injective. In the non-unital case there is a smallest {\it minimal} unitization $\mathscr C^{\bf m}
\equiv \mathscr C \oplus \mathbb C$ (so the quotient $\mathscr C^{\bf m}/\mathscr C$ has dimension one), and a largest one, namely {\it
  the multiplier algebra} $\mathscr C^{\bf M}$.  If $\mathscr C$ is unital, one has $\mathscr C=\mathscr C^{\bf m}=\mathscr C^{\bf M}$.

\smallskip
Let us fix some $x\in X$, the set of units of $\Xi$, and consider the following diagram of $C^*$-morphisms
\begin{equation}\label{coliclor}
\begin{diagram}
  \node{{\sf C}^*(\Xi,\o)}\arrow{s,r}{\Pi_x}\arrow{e,t}{\delta}\node{C^*(\Xi^\o)}\arrow{s,r}{\Pi^\o_x}\\
  \node{\mathbb  B\big[L^2(\Xi_x;\lambda_x)\big]}\arrow{e,t}{\Delta}\node{\mathbb B\big[L^2(\Xi^\o_x;\lambda^\o_x)\big]}
\end{diagram}
\end{equation}
that we now describe. The regular representation $\Pi_x$ has been introduced in subsection \ref{goamol}. Since $x\in
X\equiv(\Xi^\o)^{(0)}$, there is also the regular representation $\Pi^\o_x$ of the (non-twisted) groupoid $C^*$-algebra
$C^*(\Xi^\o)$\,. The morphism $\delta$ is given by the Brown-an Huef theory, as explained above: on $C_{\rm c}(\Xi,\o)$ it is defined as
the composition of $(\kappa^1)^{-1}$, cf. \eqref{inviersa}, with the canonical injection $C_{\rm c}(\Xi^\o\!\mid\!1)\hookrightarrow C_{\rm c}(\Xi,\o)$\,. This means
\begin{equation*}
  [\delta(f)](t,\xi)=t^{-1}f(\xi)\,,\quad\forall\, (t,\xi)\in\mathbb T\times\Xi\,,\ f\in C_{\rm c}(\Xi,\o)\,.
\end{equation*}
Finally note that, by Remark \ref{adaugita}, 
\begin{equation*}
  L^2\big(\Xi^\o_x;\lambda^\o_x\big) = L^2\big(\mathbb T\times\Xi_x;dt\times\lambda_x\big)\cong L^2(\mathbb T;dt)\otimes L^2(\Xi_x;\lambda_x)\,.
\end{equation*}
Let us set
\begin{equation*}
  S : L^2(\mathbb T;dt)\to L^2(\mathbb T;dt)\,,\quad[S(\phi)](t) :=\int_\T r\phi(r)dr\,t^{-1}.
\end{equation*}
Clearly $S$ is the orthogonal projection on $\mathbb C\psi_{-1}$\,, where $\psi_{-1}(t):=t^{-1}$ has norm one. Then we define the
$C^*$-algebraic morphism $\Delta$ by
\begin{equation*}
  \Delta(T):=S\otimes T,\quad\forall\,T\in\mathbb B\big[L^2(\Xi_x;\lambda_x)\big]\,.
\end{equation*}

\begin{Proposition}\label{sacomute}
The diagram \eqref{coliclor} commutes.
\end{Proposition}

\begin{proof} 
To check the commutativity of the diagram, it is enough to compute for $f\in C_{\rm c}(\Xi,\o)$\,, $\phi\in L^2(\T;dt)$\,, $u\in
L^2(\Xi_x;\lambda_x)$ and $(t,\xi)\in\mathbb T\times\Xi$ (by $\ast$ we denote the groupoid convolution associated to $\Xi^\o$)\,:
\begin{equation*}
\begin{aligned}
  \big[ \big( (\Pi_x^\o\circ\delta) f \big)(\phi \otimes u)\big](t,\xi) &= \big[\delta(f)\ast(\phi \otimes u)\big](t,\xi)\\
  & = \int_\T\int_\Xi\,[\delta(f)](s,\eta)(\phi \otimes u) \big((s^{-1}\o (\eta,\eta^{-1})^{-1} \!, \eta^{-1})(t, \xi) \big) dsd\lambda_x(\eta)\\
  & = \int_\T\int_\Xi s^{-1} f(\eta) \phi \big( s^{-1}t \, \o(\eta,\eta^{-1})^{-1} \o (\eta^{-1}, \xi) \big) u \big( \eta^{-1}\xi \big)dsd \lambda_x( \eta)\\
  & = \int_\T \int_\Xi s^{-1} f(\eta) \phi \big(s^{-1}t\, \o(\eta,\eta^{-1} \xi)^{-1} \big) u( \eta^{-1} \xi) dsd \lambda_x(\eta)\\
  & = t^{-1} \! \int_\T r \phi( r ) dr \int_\Xi f( \eta) u(\eta^{-1}\xi ) \o ( \eta, \eta^{-1} \xi ) d \lambda_x ( \eta ) \\
  & = [ S ( \phi ) ] ( t )(f\star_\o\!u)(\xi)\\
  & = \big[ \big (S \otimes \Pi_x(f) \big)( \phi \otimes u) \big] (t,\xi)\\
  & = \big[ \big( ( \Delta \circ \Pi_x ) f \big ) ( \phi \otimes u)\big] ( t, \xi )\,.
\end{aligned}
\end{equation*}
The forth equality relies on the $2$-cocycle identity \eqref{doicocic} and the normalization \eqref{normaliz}, in which the unit is our
$\eta\eta^{-1}$. The fifth equality is obtained by a change of variable.
\end{proof}

The next result is a well-known result of Koshkam and Skandalis \cite[Cor.\,2.4]{KS} the usual (untwisted) groupoid $C^*$-algebras
(see also \cite[Prop.\,2.7]{BB}). The point (i), for the twisted case, has been shown in \cite[Prop.2.5]{MN} even without the second
countability assumption. We include it here since the method of proof might be interesting.

\begin{Corollary}\label{construint}
Let $\o$ be a continuous $2$-cocycle on the standard groupoid $\Xi\tto X$ with main orbit $M$.
\begin{enumerate}
\item[(i)] For every $x\in M$ the representation $\Pi_x$ is  faithful. If $\,{\sf C}^*(\Xi,\o)$ is not unital, for any
  unitization, the extension $\Pi_x^{\bf U}:{\sf C}^*(\Xi,\o)^{\bf U}\to\mathbb B\big[L^2(\Xi_x;\lambda_x)\big]$ is also faithful.
\item[(ii)] The vector representation $\Pi_0$ and its unitizations are faithful.
\end{enumerate}
\end{Corollary}

\begin{proof}
(i) From Proposition \ref{sacomute}, one has $\Pi_x^\o\circ\delta=\Delta\circ\Pi_x$\,. The morphism $\delta$ is
  injective since $\Xi^\o$ is Hausdorff \cite[Cor.\,2.4]{KS}. See also
  \cite[Prop.\,2.7]{BB}. It follows then that $\Pi_x$ is also injective. In the non-unital case, faithfulness is preserved for the
  the extension, since ${\sf C}^*(\Xi,\o)$ is an essential ideal of ${\sf C}^*(\Xi,\o)^{\bf U}$.

\smallskip
(ii) is a consequence of the unitary equivalence described in Remark \ref{idem} and of the part (i).
\end{proof}

The next notion \cite{Ex,Ro} will play an important role later.

\begin{Definition}\label{excel}
We say that the twisted groupoid $(\Xi,\o)$ {\rm has Exel's property} if for every $f \in {\sf C}^*(\Xi,\o)$, there exists $x\in X$ (depending on $f$) such that $\,\|f\|\,=\,\|\Pi_x(f)\|_{\mathbb
  B[L^2(\Xi_x;\lambda_x)]}$\,.
\end{Definition}

This means that in \eqref{redunorm} the supremum of the function $X\ni x \to \|\Pi_x(f)\|_{\mathbb B[L^2(\Xi_x;\lambda_x)]}$ is attained for each element $f \in {\sf C}^*(\Xi,\o)$. If there is no cocycle, we speak simply of {\it groupoids having Exel's property}; in particular, this can be applied to $\Xi^\o$.

\begin{Corollary}\label{decolocolo}
If the groupoid extension $\Xi^\o$ has Exel's property, then the twisted groupoid $(\Xi,\o)$ has Exel's property.
\end{Corollary}

\begin{proof}
We are going to use once again the commutative diagram \eqref{coliclor}. Let $f\in{\sf C}^*(\Xi,\o)$ and assume that the
point for which the norm of $\Phi:=\delta(f)\in{\sf C}^*(\Xi^\o)$ is attained is $x$\,; set $\H_x:=L^2(\Xi_x;\lambda_x)$\,. Then, since $\delta$ is injective and $\|S\|_{\mathbb B[L^2(\T)]}\,=1$
\begin{equation*}
\begin{aligned}
  \|f\|_{{\sf C}^*(\Xi,\o)}\,&=\,\|\delta(f)\|_{{\sf C}^*(\Xi^\o)}\\
  &=\,\|\Pi^\o_x[\delta(f)]\|_{\mathbb B[L^2(\T)\otimes\H_x]}\\
  &=\,\|\Delta[\Pi_x(f)]\|_{\mathbb B[L^2(\T)\otimes\H_x]}\\
  &=\,\|S\otimes\Pi_x(f)\|_{\mathbb B[L^2(\T)\otimes\H_x]}\\
  &=\,\|\Pi_x(f)\|_{\mathbb B(\H_x)}\,,
\end{aligned}
\end{equation*}
showing that the norm of $f$ is also attained in $x$\,.
\end{proof}

\begin{Corollary}\label{feminism}
Let $\Xi$ be a tractable groupoid and $\o$ a $2$-cocycle. Then the twisted groupoid $(\Xi,\o)$ has Exel's property.
\end{Corollary}

\begin{proof}
The starting point is the fact that a tractable groupoid has Exel's property. This is Theorem 3.18 from \cite{NP}, extending a result of
Exel \cite{Ex} from the \'etale case and relying on a deep result of Ionescu and Williams \cite{IW}. See also \cite{CNQ}.  By Lemma
\ref{abercorn}, the extension $\Xi^\o$ is also tractable. Thus $\Xi^\o$ has Exel's property. Finally, Corollary \ref{decolocolo}
shows that Exel's property is inherited from the extension by the twisted groupoid $(\Xi,\o)$\,.
\end{proof}

\subsection{Restrictions}\label{guacamole}

Let $M$ be an open $\Xi$-invariant subset of $X$. Then $A\!:=M^{\rm c}$ is closed and $\Xi$-invariant, which yields the restrictions
$\Xi^M_M=\Xi^M\!=\Xi_M$ and $\Xi^{A}_{A}=\Xi^{A}\!=\Xi_{A}$ and \cite[Lemma 2.10]{MRW3} a short exact sequence of groupoid $C^*$-algebras
\begin{equation}\label{ses}
0 \longrightarrow {\sf C}^*(\Xi_M) \longrightarrow {\sf C}^*(\Xi)\overset{\rho_A}{\longrightarrow} {\sf C}^*(\Xi_A) \longrightarrow 0\,.
\end{equation}
On continuous compactly supported functions the monomorphism consists of extending by the value $0$ and the epimorphism $\rho_A$ acts as a
restriction. We recall that one works on $\Xi_{M}$ and $\Xi_{A}$, respectively, with the corresponding restrictions of the fixed Haar system $\lambda$\,.

\begin{Proposition}\label{cociclizare}
Let $\o$ be a $2$-cocycle on the tractable groupoid $\,\Xi$ and let $\,\Xi^{(0)}\!=:\!X=M\sqcup A$ for some open invariant subset of units
$M$. Using restricted $2$-cocycles $\,\o_B\!:(\Xi_B)^{(2)}\to \T$\,, one has the short exact sequence
\begin{equation}\label{sas}
  0 \longrightarrow{\sf C}^*(\Xi_M,\o_M) \longrightarrow {\sf C}^*(\Xi,\o) \overset{\rho_A}{\longrightarrow} {\sf C}^*(\Xi_A,\o_A) \longrightarrow 0\,.
\end{equation}
\end{Proposition}

\begin{proof}
Recall the groupoid extension $\Xi^\o$ and the identifications $(1,x)\equiv x$ and $\{1\}\times B\equiv B$ for any $B\subset
X$. First note that if $B$ is locally closed and $\Xi$-invariant and $\o_B$ is the restriction of $\o$ to $(\Xi_B)^{(2)}=(\Xi_B\times\Xi_B)\cap\Xi^{(2)}$, then the extension
$\Xi_B^{\o_B}$ and the $C^*$-algebras ${\sf C}^*(\Xi_B,\o_B)$ and ${\sf C}^*\big(\Xi_B^{\o_B}\big)$ make sense. We recall that the
$C^*$-norms involve in each case all the $^*$-representations that are continuous with respect to the inductive limit topology on $C_{\rm c}$-functions and the weak topology on operators.

\smallskip
If $M\subset X$ is open and $\Xi$-invariant, then $\{1\}\times M\equiv M$ is open and $\Xi^\o$-invariant. Thus we have for free the short exact sequence
\begin{equation}\label{sus}
  0 \longrightarrow {\sf C}^*(\Xi^\o_M) \longrightarrow {\sf C}^*(\Xi^\o) \overset{\rho_A^\o}{\longrightarrow}{\sf C}^*(\Xi^\o_A) \longrightarrow 0\,,
\end{equation}
where on $C_{\rm c}(\Xi^\o)$ the restriction morphism $\rho_A^\o$ only acts at the level of the second variable.

\smallskip 
On the other hand, as explained above, twisted groupoid $C^*$-algebras may be seen as (closed two-sided self-adjoint)ideals and direct summands in the $C^*$-algebras of the corresponding extensions.  The rough idea of the proof is that {\it we could write the sequence of $\,C^*$-algebra isomorphisms: 
\begin{equation*}\label{sequence} 
{\sf C}^*(\Xi_A,\o_A)\cong{\sf C}^*\big(\Xi_A^{\o_A}|1\big)\cong{\sf C}^*(\Xi^\o| 1)/{\sf C}^*\big(\Xi_M^{\o_M} |1\big)\cong{\sf
      C}^*(\Xi,\o)/{\sf C}^*(\Xi_M,\o_M) \end{equation*} if one would justify the middle isomorphism (that still involves $C^*$-algebras
  that do not have the form ${\sf C}^*(\H)$ for some groupoid $\H$)}\,.
 
 \smallskip 
We now fix $n=1$ and note that the results of \cite{BaH} hold for all the pairs $(\Xi,\o)$\,, $(\Xi_M,\o_M)$ and $(\Xi_A,\o_A)$\,. For the
relevant subsets $B=A,M$, use notations as $\kappa^1_B\,,\chi^1_B,{\sf J}^1_B$\,. Both the operations of extending by the value $0$ in the
$\Xi$-variable and restricting in the $\Xi$-variable preserve homogeous components and commute (when suitably defined) with the
$^*$-morphisms $\chi^1_B$ and $\kappa^1_B$\,. It is enough to check this at the level of continuous compactly supported functions, which
is trivial. Thus from \eqref{sus} we deduce the short exact sequence \begin{equation*}\label{sos} 0\longrightarrow {\sf  J}^1_M\longrightarrow {\sf J}^1\longrightarrow {\sf
    J}_A^1\longrightarrow 0\,, \end{equation*} where the two non-trivial arrows are restrictions to homogeneous components of the
corresponding non-trivial arrows from \eqref{sus}. This one is isomorphic to the short sequence 
\begin{equation*}\label{sras}
  0\longrightarrow{\sf C}^*(\Xi_M,\o_M)\longrightarrow{\sf C}^*(\Xi,\o)\overset{\rho_A}{\longrightarrow}{\sf C}^*(\Xi_A,\o_A)\longrightarrow 0\,.  
 \end{equation*} 
 by using the vertical arrows $\kappa^1_B$\,, and then to \eqref{sas}, by composing with vertical arrows $\chi^1_B$\,. As we said, commutativity of the
square diagrams are checked easily on continuously compactly supported functions by using explicit formulas for $\chi^1_B$ and $\kappa^1_B$
(they "commute with restrictions or $0$-extensions in the variable $\xi\,$")\,, which is enough. Exactness of \eqref{sas} now follows and the proof is finished.
\end{proof}

Let us denote by $\mathbf{inv}(X)$ the family of all closed, $\Xi$-invariant subsets of $X$. For $A,B\in\mathbf{inv}(X)$ with $B\supset A$, one has the canonical inclusion
$i_{BA}:\Xi_{A}\to\Xi_{B}$ and the restriction morphism
\begin{equation*}
  \rho_{BA}:C_{\rm c}(\Xi_{B})\to C_{\rm c}(\Xi_{A})\,,\quad\rho_{BA}(f):=f|_{\Xi_A}=f\circ i_{BA}\,,
\end{equation*}
which extends to an epimorphism $\rho_{BA}:{\sf C}^*(\Xi_{B},\o_B)\to{\sf C}^*(\Xi_{A},\o_A)$\,.  Clearly
$\rho_{A}=\rho_{XA}$\,. One gets in this way an inductive system $\big\{{\sf C}^*(\Xi_{A},\o_A),\rho_{BA}\big\}_{\mathbf{inv}(\Xi)}$ of $C^*$-algebras.

\begin{Corollary}\label{impec}
In the setting above, if $\;B\supset A$ there is a short exact
sequence
\begin{equation*}
  0 \longrightarrow {\sf C}^*(\Xi_{B\setminus A},\o_{B\setminus A})\longrightarrow {\sf
    C}^*(\Xi_B,\o_B)\overset{\rho_{BA}}{\longrightarrow} {\sf C}^*(\Xi_{A},\o_A)\longrightarrow 0\,,
\end{equation*}
allowing the identification
\begin{equation}\label{scurt}
{\sf C}^*\big(\Xi_{A},\o_A\big) \cong {\sf C}^*\big(\Xi_B,\o_B\big)/{\sf C}^*\big(\Xi_{B\setminus A},\o_{B\setminus A}\big)\,.
\end{equation}
\end{Corollary}

\begin{proof}
The statement follows from our Proposition \ref{cociclizare}, since $\Xi_B$ is also tractable.
\end{proof}

\begin{Remark}\label{exasperare}
\normalfont{Assume that the $C^*$-algebras $\mathscr C$ and $\mathscr D$ are not unital and denote by $\mathscr C^{\bf m}$ and $\mathscr
  D^{\bf m}$, respectively, their minimal unitizations. If $\rho:\mathscr C\to\mathscr D$ is a $^*$-epimorphism, then
\begin{equation*}
  \rho^{\bf m} : \mathscr C^{\bf m} \to \mathscr D^{\bf m},\quad\rho^{\bf m}(c+\lambda 1) := \rho(c)+\lambda 1
\end{equation*}
is a unital $^*$-epimorphism and $\ker(\rho^{\bf m})=\ker(\rho)$\,. Applying this to the situation in the Corollary, besides \eqref{scurt} one gets
\begin{equation*}
  {\sf C}^*(\Xi_{A},\o_A) \cong {\sf C}^*(\Xi_B,\o_B)^{\bf m}/{\sf C}^*(\Xi_{B\setminus A},\o_{B\setminus A})^{\bf m}\,.
\end{equation*}}
\end{Remark}

\begin{Remark}\label{coerenta}
\normalfont{Let $x\in A \in \mathbf{inv}(X)$\,; one has the regular representations $\Pi_x\!:{\sf C}^*(\Xi,\o) \to \mathbb B\big[L^2(\Xi_x; \lambda_x) \big]$ and $\Pi_{A,x}\!:{\sf
    C}^*\big(\Xi_A, \o_A\big)\to \mathbb B \big[L^2(\Xi_x;\lambda_x)\big]$ (the fibers $\Xi_x$ and $\Xi_{A,x}$ are seen to
  coincide). It is easy to check, and will be used below, that $\Pi_{x}=\Pi_{A,x} \circ \rho_A$\,. }
\end{Remark}

We close this section with one more remark about restrictions. In some situations, $C_{\rm c}(\Xi)$ is too small to contain all the elements
of interest. On the other hand, the nature of the elements of the completion ${\sf C}^*(\Xi,\o)$ is not easy to grasp; for instance it
is not clear the concrete meaning of a restriction applied to them. In search of a good compromise, we recall \cite[pag.\,50]{Re} the embeddings
\begin{equation*}
  C_{\rm c}(\Xi)\subset L^{\infty,1}(\Xi)\subset {\sf C}^*(\Xi,\o)\,,
\end{equation*} 
where $L^{\infty,1}(\Xi)$ is a Banach $^*$-algebra, the completion of $C_{\rm c}(\Xi)$ in {\it the Hahn norm}
\begin{equation*}
  \|f\|_{\Xi}^{\infty,1} := \max\Big\{\sup_{x\in X}\int_{\Xi_x}\!|f(\xi)| d\lambda_x(\xi)\,,\sup_{x\in X}\int_{\Xi_x}\!|f\big(\xi^{-1}\big)|d\lambda_x(\xi)\Big\}\,.
\end{equation*}
Actually, ${\sf C}^*(\Xi,\o)$ is the enveloping $C^*$-algebra of $L^{\infty,1}(\Xi)$\,. The following result is obvious:

\begin{Lemma}\label{folou} 
If $A\subset X$ is closed and invariant, the restriction map $\rho_A$ extends to a $^*$-morphism $\rho_A:L^{\infty,1}(\Xi)\to
L^{\infty,1}\big(\Xi_A\big)$ that is contractive with respect to the Hahn norms.
\end{Lemma}

\begin{Remark}\label{doomsday}
\normalfont{It is still unclear how this restrictions act on a general element $f\in L^{\infty,1}(\Xi)$\,. However, if $f$ is a continuous function on $\Xi$ with finite Hahn norm (write $f\in
  L^{\infty,1}_{\rm cont}(\Xi)$)\,, the restriction acts in the usual way and one gets an element $\rho_A(f)$ of $L^{\infty,1}_{\rm cont}(\Xi_A)$\,.}
\end{Remark}

\section{Spectral results}\label{sec.three}

\subsection{Quasi-orbit decomposition of spectra}\label{meni}

In this subsection we continue to assume that {\it the groupoid $\Xi$\,, with unit space $X$, is tractable and that $\o$ is a
  (continuous) $2$-cocycle on $\,\Xi$}\,. We are going to need some terminology.
\begin{itemize}
\item
{\it A quasi-orbit} is the closure of an orbit and $\,\mathbf{qo}(X)\subset\mathbf{inv}(X)$ denotes the family of all quasi-orbits. {\it The orbit of a point} $x$ will be denoted by
$\mathcal O_x$\,. Of course, $\mathcal O_x = {\rm r}(\Xi_x)$\,. Let us set $\mathcal Q_x := \overline{\mathcal O_x}$ for {\it the quasi-orbit generated by $x$}\,.
\item
More generally, if $Y\in\mathbf{inv}(X)$, the families $\,\mathbf{qo}(Y)\subset\mathbf{inv}(Y)$ are defined similarly, in
terms of the restriction groupoid $\Xi_Y$\,; clearly $\mathbf{qo}(Y)$ can be seen as the family of all the $\Xi$-quasi-orbits that are contained in $Y$\,.
\item
Let $x\in\mathcal Q\in\mathbf{qo}(X)$\,. We say that $x$ is {\it generic} (with respect to $\mathcal Q$) and write $x\in\mathcal
Q^{\rm g}$ if $x$ generates $\mathcal Q$\,, i.e. $\mathcal Q_x=\mathcal Q$. Then the subset $\mathcal Q^{\rm g}$ is invariant and
dense in $\mathcal Q$. Otherwise, that is, if $x \in \mathcal Q$\,, but $\mathcal Q_x \neq \mathcal Q$\,, we say that $x$ is {\it
  non-generic} in $\mathcal Q$ and write $x\in\mathcal Q^{\rm n}$; this means that $\mathcal Q_x\subsetneqq\mathcal Q$\,.
\end{itemize}

Recall that a surjective $^*$-morphism $\Phi:\mathscr C\to\mathscr D$ extends uniquely to a $^*$-morphism $\Phi^{\bf M}:\mathscr C^{\bf
  M}\to \mathscr D^{\bf M}$ between the multiplier algebras; if $\mathscr C$ and $\mathscr D$ are $\si$-unital (separable, in
particular), then $\Phi^{\bf M}$ is also surjective. For each locally closed invariant subset $Y$ of $X$ one sets $\mathscr C_Y := {\sf
  C}^*(\Xi_Y,\o_Y)$ and $\mathscr C_Y^{\bf U} := {\sf C}^*(\Xi_Y,\o_Y)^{\bf U}$ for ${\bf U} \in \{{\bf m,M}\}$; the index
$Y = X$ will be omitted. It is also convenient to abbreviate $\H_y:=L^2(\Xi_y;\lambda_y)$ for every $y\in X$.

\smallskip
For an element $F$ belonging to $\mathscr C^{\bf U}$, for every $Y\in\mathbf{inv}(X)$ and for every $x\in X$, one sets
\begin{equation}\label{poartafesul}
  F_Y := \rho_Y^{\bf U}(F) \in \mathscr C_Y^{\bf U} \quad\mbox{and}\quad H_x := \Pi_x^{\bf U}(F) \in \mathbb B(\H_x)\,.
\end{equation} 
Below, the letter ${\bf U}$ indicates one of the unitizations $\mathbf{m}$ or ${\mathbf M}$\,; Remark \ref{camsuy} refers to these choices.  The next result is interesting in itself and complements
some of the results from \cite[Sect.3]{MN}. It will have a counterpart in terms of the essential spectrum; see Theorem \ref{carefuly}.
Occasionally, although this is not really necessary, we will indicate the unital $C^*$-algebra in which the spectrum is computed.

\begin{Theorem}\label{garefuly}
Let $\Xi \tto X$ be an tractable groupoid endowed with a $2$-cocycle $\o$ and let $Y\subset X$ a closed invariant set of units. Consider an
element $F\in\mathscr C^{\bf U}\!:={\sf C}^*(\Xi,\o)^{\bf U}$\,. One denotes by ${\sf S}_Y(F)$ the spectrum of the image $\beta_Y(F_Y)$ of
$F_Y\in\mathscr C^{\bf U}_Y$ in the quotient $\mathscr C_Y^{\bf U}/\mathscr C_Y$. Let $\,\{\mathcal Q_i\}_{i\in I}$ be a family of
quasi-orbits such that $\,Y=\bigcup_{i\in I}\mathcal Q_i$\,, and for each $i\in I$ let $x_i\in\mathcal Q_i^{\sf g}$ (i.e.\! $x_i$ generates the quasi-orbit $\mathcal Q_i$)\,. Then
\begin{equation}\label{hardiancess}
 {\sf sp}(F_Y) \, = \, {\sf S}_Y(F) \cup \bigcup_{x\in Y}{\sf sp}\big(H_x\big) \, = \, {\sf S}_Y(F) \cup\, \bigcup_{i\in I}{\sf
   sp} \big(H_{x_i} \big) \, = \, {\sf S}_Y(F) \cup\, \bigcup_{i\in I}{\sf sp}\big(F_{\mathcal Q_i} \vert \mathscr C^{\bf U}_{\mathcal Q_i} \big) \,.
\end{equation}
\end{Theorem}

\begin{proof}
The equation \eqref{hardiancess} will be abbreviated to ${\sf sp}(F_Y) = \Si_1 = \Si_2 = \Si_3$\,. Let us prove first that ${\sf
  sp}(F_Y)=\Si_1$\,. Using notations as in Remark \ref{coerenta}, let us consider the family of $^*$-morphisms
\begin{equation*}
  \mathfrak F := \big\{\beta_Y : \mathscr C^{\bf U}_Y \to \mathscr C^{\bf U}_Y/\mathscr C_Y\big\} \cup \big \{\Pi^{\bf U}_{Y,x} :
  \mathscr C^{\bf U}_Y \to \mathbb B(\H_x)\, \big\vert\, x\in Y\big\}\,.
\end{equation*}
Note that, in connection with the explicit form of $\mathfrak F$ and of $\Si_1$\,, we have that
\begin{equation*}
  {\sf S}_Y(F) = {\sf sp} \big[\beta_Y(F_Y)\big] \quad \mbox{and}\quad H_x = \Pi_x^{\bf U}(F) = \Pi^{\bf U}_{Y,x}(F_Y)\,.
\end{equation*} 
By \cite[Th.\,3.6\,,Th.\,3.4]{NP}\,, to show that $\,{\sf sp}(F_Y)$ coincides with the union $\Si_1$ of spectra, it is enough to prove
that the family $\mathfrak F$ is {\it strictly norming}, meaning that for every $G_Y\in\mathscr C^{\bf U}_Y$ there is an element $\pi$ of
the family such that $\|G_Y\|_{\mathscr C^{\bf U}_Y}\,=\,\|\pi(G_Y)\|$\,. In Corollary \ref{feminism} it is already proven that the twisted groupoid algebra $\mathscr C_Y\!:={\sf
  C}^*(\Xi_Y,\o_Y)$ has Exel's property; in other words, by Definition \ref{excel}, the family
\begin{equation*}
  \mathfrak F_0 := \big\{\Pi_{Y,x} : \mathscr C_Y \to \mathbb B(\H_x)\, \big \vert\, x\in Y\big\}
\end{equation*} 
is strictly norming for the ideal $\mathscr C_Y$. Then $\mathfrak F$ will be strictly norming for $\mathscr C^{\bf U}_Y$, by
\cite[Prop.\,3.15]{NP}. For the reader's convenience, we indicate the remaining argument using our notation.

\smallskip
So let $G_Y\in\mathscr C^{\bf U}_Y$\,; by replacing it with $G_Y^{\star_\o}\star_\o G_Y$, one can assume it is positive. If
$\|G_Y\|_{\mathscr C^{\bf U}_Y}\,=\,\|\beta_Y(G_Y)\|_{\mathscr C^{\bf U}_Y/\mathscr C_Y}$, we are done. Otherwise, let $\varphi$ be a
continuous function on ${\sf sp}(F_Y)$ that is zero on ${\sf sp}\big[\beta_Y(F_Y)\big]$\,, satisfying
\begin{equation*}
  \varphi\big( \|G_Y\|_{\mathscr C^{\bf U}_Y} \big) = \,\|G_Y\|_{\mathscr C^{\bf U}_Y} \quad \mbox{and} \quad \varphi(t) \le t\ \ {\rm if}\ \ t\ge 0\,.
\end{equation*} 
Then $\varphi(G_Y)\in\mathscr C_Y$ and $\,\|\varphi(G_Y)\|_{\mathscr C^{\bf U}_Y} \, = \, \|G_Y\|_{\mathscr C^{\bf U}_Y}$. Since the
family $\mathfrak F_0$ is strictly norming, there exists $x\in Y$ such that
\begin{equation*}
  \|G_Y\|_{\mathscr C^{\bf U}_Y} \, =\, \|\varphi(G_Y)\|_{\mathscr C^{\bf U}_Y} \, = \,\|\Pi_{Y,x}\big[\varphi(G_Y)\big]\|_{\mathbb
    B(\H_x)} \,=\, \|\varphi\big[\Pi_{Y,x}(G_Y)\big]\|_{\mathbb B(\H_x)}\, \le \,\|\Pi_{Y,x}(G_Y)\|_{\mathbb B(\H_x)}.
\end{equation*}
This shows that the family $\mathfrak F$ is strictly norming and the proof of ${\sf sp}(F_Y)=\Si_1$ is completed.

\smallskip
The equalities of the various $\Sigma$s follows from some general $C^*$-arguments. Indeed, recall that, by Remark \ref{coerenta}, one can write
\begin{equation*}
  H_{x_i} \, = \, \Pi^{\bf U}_{x_i}(F) \, = \, \Pi^{\bf U}_{\mathcal Q_i,x_i} \big[\rho^{\bf U}_{\mathcal Q_i}(F)\big] \, = \, \Pi^{\bf U}_{\mathcal Q_i,x_i} \big(F_{\mathcal Q_i}\big)\,.
\end{equation*} 
The regular representation $\Pi^{\bf U}_{\mathcal Q_i,x_i} : \mathscr C^{\bf U}_{\mathcal Q_i} \to \mathbb B(\H_{x_i})$ is faithful, by Corollary \ref{construint}, the orbit of $x_i$ being dense in
$\mathcal Q_i$\,. Thus it preserves the spectrum and one gets $\Si_1=\Si_2$\,. Then, clearly $\Si_2\subset\Si_1$\,. Since the family
$\,\{\mathcal Q_i\}_{i\in I}$ covers $Y$, for each $x\in Y$ there is some $i\in I$ such that $x\in\mathcal Q_i$\,. Then $\mathcal
Q_x\subset\mathcal Q_i$ and thus ${\sf sp}\big(H_x\big)\subset{\sf sp}\big(H_{x_i}\big)$\,, by Theorem 3.3 in \cite{MN}. This shows that $\Si_2=\Si_3$ as well. This completes the proof.
\end{proof}


\begin{Remark}\label{camsuy}
\normalfont{By choosing a large unitizations ${\bf U}$, we obtain a larger class of elements to which the results apply, but understanding the set ${\sf S}_Y(F)$ becomes more difficult.
\begin{enumerate}
\item[(a)] Let us assume that $\mathscr C$ is not unital and take ${\bf U=m}$ (if $F$ permits it). Then $\mathscr C^{\bf m}\ni F=G+{\sf s}1$ for some unique $G\in\mathscr C$ and ${\sf
    s}\in\mathbb C$\,. In this case ${\sf sp}(F_Y)=\{\sf s\}$ for every $Y$. Clearly ${\sf s}=0$ if and only if $F\in\mathscr C$\,.
\item[(b)] On the other hand, if one takes ${\bf U}={\bf M}$\,, allowing much larger classes of elements $F$, in most situations the
  component ${\sf S}_Y(F)$ (a spectrum in a corona algebra) is difficult to compute explicitly.
\item[(b)] Actually, one can also work with arbitrary unital algebras (or ``unitizations'') ${\sf C}^*(\Xi,\o)^{\bf U}$ containing ${\sf C}^*(\Xi,\o)$ as an essential ideal. We expect that the choice ${\sf
  C}^*(\Xi,\o)^{\bf U}\!=$ the completion of $\Psi^0(\Xi; \o)$ of Remark \ref{rem.psdo} will be useful in applications. This is the
  case if $\omega = 1$ and $Y=X$ (see \cite{CNQ} and the references therein); the presence of a principal symbol map leads to a
  commutative $C^*$-algebra $\overline{\Psi^0(\Xi)}/{\sf C}^*(\Xi)$ and now the set ${\sf S}_X(F)$ is given simply by the values of this
  principal symbol; see \cite[Corollary 4.16]{CNQ}.  In the presence of a cocycle, these properties remain to be checked.
\end{enumerate}}
\end{Remark}

\subsection{Essential spectra and the Fredholm property}\label{ameny}

In this subsection we are mainly interested in the essential spectra of the operators arising from regular representations of the twisted  groupoid $C^*$-algebra.

\begin{Definition}\label{discresential}
{\rm The discrete spectrum} ${\sf sp}_{\rm dis}(H)$ of a bounded operator $H$ in a Hilbert space $\mathscr H$ is the set of all its
isolated eigenvalues of finite multiplicity. The complement of the discrete spectrum of $H$ is called the {\rm the essential spectrum} of
$H$ and it is denoted by ${\sf sp}_{\rm ess}(H) := {\sf sp}_{\rm dis}(H)^c = {\sf sp}(H)\!\setminus\!{\sf sp}_{\rm dis}(H)$\,.
\end{Definition}

We also have that ${\sf sp}_{\rm ess}(H) = \{\lambda \in \mathbb C \,\big\vert\, H - \lambda \mbox{ is not Fredholm} \}$.

\smallskip
In the following, the framework and the notation are the same as in the previous subsection. {\it We decided to restrict our discussion to
  that of the minimal unitization} ${\bf U}={\bf m}$\,, mainly in order to make the results more concrete (cf. Remark \ref{camsuy}).

\smallskip
The next lemma describes the generic part $\mathcal Q^{\rm g}$ of a $\Xi$-quasi-orbit $\mathcal Q$ in certain cases.

\begin{Lemma}\label{nuispatiu}
Suppose that $\mathcal O$ is a locally closed orbit and set $\mathcal Q:=\overline{\mathcal O}\in\mathbf{qo}(X)$\,. Then $\mathcal O$ is
open in $\mathcal Q$\,, one has $\mathcal O=\mathcal Q^{\rm g}$ and no other orbit contained in $\mathcal Q$ is dense in $\mathcal Q$\,. If,
in addition, the orbit $\mathcal O$ is free (i.e. its isotropy is trivial), then the reduced groupoid $\Xi_\mathcal Q$ is standard, $\mathcal O$ being the main orbit.
\end{Lemma}

\begin{proof}
It is well-known that the subset $\mathcal O$ is locally closed if, and only if, it is open in its closure $\mathcal Q$\,. All the points
of $\mathcal O$ are generic for $\mathcal Q$\,, by construction. If $\mathcal O'\subset\mathcal Q$ is another orbit, then
$\overline{\mathcal O'}\cap\mathcal O=\emptyset$ (the closures of $\mathcal O'$ in $\mathcal Q$ and in $X$ are the same) and thus
$\mathcal O'$ cannot be dense in $\mathcal Q$\,. So, if $x \notin \mathcal O$\,, then $\mathcal Q_x:=\overline{\mathcal O_x}$ is
strictly contained in $\mathcal Q$ and thus $x$ is not generic.  The last assertion of the lemma follows from the definition of a standard groupoid.
\end{proof}

\begin{Definition}\label{regular}
A point $z \in X := \Xi^{(0)}$ is called {\em regular} if its orbit $\mathcal O_z$ is locally closed and free.
\end{Definition}

The next Theorem provides Fredholmness criteria and information about the essential spectrum of $H_z:=\Pi^{\bf m}_z(F)$\,, for such a
regular unit $z \in X$\,, in terms of the invertibility and of the usual spectra of various other elements or operators.

\begin{Theorem}\label{carefuly}
Let $\Xi \tto X$ be a tractable groupoid endowed with a continuous $2$-cocycle $\o$ and consider an element $F = G+{\sf s}1 \in \mathscr
C^{\bf m} := {\sf C}^*(\Xi,\o)^{\bf m}$\,, with $G \in \mathscr C:={\sf C}^*(\Xi,\o)$ and ${\sf s} \in \mathbb C$\,. Let $z$ be a regular element of the unit space $X = \Xi^{(0)}$.
\begin{enumerate}
\item[(i)]  One has
\begin{equation}\label{jardance}
 {\sf sp}_{\rm ess}(H_z) = {\sf sp} \big(F_{\mathcal Q_z^{\rm n}}\big \vert \mathscr C^{\bf m}_{\mathcal Q_z^{\rm n}}\big)\,.
\end{equation}
\item[(ii)] Let $\,\{\mathcal Q_j\}_{j\in J}\subset\mathbf{qo}\big(\mathcal Q_z\big)$ be a set of quasi-orbits such that $\mathcal Q_z^{\rm n}=\bigcup_{j\in J}\mathcal Q_j$
  (i.e. it is a covering of $\mathcal Q_z^{\rm n}=\mathcal Q_z\!\setminus\!\mathcal O_z$ with quasi-orbits), and for each $j\in
  J$ let $x_j\in\mathcal Q_j^{\sf g}$ (i.e.\! $x_j$ generates the quasi-orbit $\mathcal Q_{j}$)\,.  Then
\begin{equation}\label{jardiancess}
  {\sf sp}_{\rm ess}(H_z) = \{{\sf s}\}\cup\bigcup_{j\in J}{\sf sp}\big(F_{\mathcal Q_j} \!\mid\! \mathscr C_{\mathcal Q_j}^{\bf
    m}\big) = \{{\sf s}\}\cup\bigcup_{j\in J}{\sf sp}\big(H_{x_j}\big) = \{{\sf s}\}\cup\!\bigcup_{x \in \mathcal Q_z^{\rm n}}{\sf sp}\big(H_x\big) \,.
\end{equation}
\item[(iii)] Keeping all the previous assumptions and notations, the following assertions are equivalent:
\begin{itemize}
\item the operator $H_z$ is Fredholm,
\item $F_{\mathcal Q_z^{\rm n}}$ is invertible in $\mathscr C^{\bf m}_{\mathcal Q_z^{\rm n}}$\,,
\item $F_{\mathcal Q_z}\!\notin\mathscr C_{\mathcal Q_z}$ and $F_{\mathcal Q_j}$ is invertible in $\mathscr C_{\mathcal Q_j}^{\bf m}$ for every $j\in J$,
\item $F_{\mathcal Q_z}\!\notin\mathscr C_{\mathcal Q_z}$ and each $H_{x_j}$ is invertible,
\item $F_{\mathcal Q_z}\!\notin\mathscr C_{\mathcal Q_z}$ and $H_x$ is invertible for each $x\in\mathcal Q_z^{\rm n}$\,.
\end{itemize}
\end{enumerate}
\end{Theorem}

\begin{proof}
The idea of the proof is as follows. Proving \eqref{jardance} relies on the fact that the essential spectrum of
\begin{equation*}
  H_z = \Pi^{\bf m}_{z}(F) = \Pi^{\bf m}_{\mathcal Q_{z},z}\big[\rho^{\bf m}_{\mathcal Q_{z}}(F) \big] = \Pi^{\bf m}_{\mathcal Q_{z},z}(F_{\mathcal Q_z}) \in \mathbb B(\H_z)
\end{equation*} 
coincides with the spectrum of its image in the Calkin algebra $\mathbb B(\H_z)/\mathbb K(\H_z)$ and on writing
\begin{equation}\label{bazilica}
\mathscr C^{\bf m}_{\mathcal Q_z^{\rm n}} \cong \mathscr C^{\bf m}_{\mathcal Q_z}/\mathscr C_{\mathcal O_z} \hookrightarrow\mathbb B(\H_z)/\mathbb K(\H_z)\,.
\end{equation}
The isomorphism in \eqref{bazilica} follows from Corollary \ref{impec} and Remark \ref{exasperare} with $A\!:=\mathcal Q_z^{\rm
  n}\subset\mathcal Q_z=:\!B$ and $B\!\setminus\!A=\mathcal O_z$\,. The $C^*$-algebraic embeding $\hookrightarrow\,$ reflects the
fact that the (unitization of the) regular representation $\Pi^{\bf m}_{\mathcal Q_z,z}$ sends injectively $\mathscr C^{\bf m}_{\mathcal
  Q_z}$ into $\mathbb B(\H_z)$\,, while its restriction sends isomorphically $\mathscr C_{\mathcal O_z}$ onto $\mathbb K(\H_z)$ because $z$ is regular.

\smallskip
Let us give now more details. Denoting by $\gamma_z\!:\mathscr C^{\bf m}_{\mathcal Q_z}\to\mathscr C^{\bf m}_{\mathcal Q_z}/\mathscr
C_{\mathcal O_z}$ and $\Gamma_z\!:\mathbb B(\mathscr H_z)\to\mathbb B(\H_z)/\mathbb K(\H_z)$ the canonical quotient maps, we have the diagram of $^*$-morphisms:
\begin{equation}\label{coliflor}
\begin{diagram}
  \node{\mathscr C^{\bf m}}\arrow{e,t}{\rho^{\bf m}_{\mathcal Q_z}}\node{\mathscr C^{\bf m}_{\mathcal
      Q_z}}\arrow{s,r}{\gamma_z}\arrow{e,t}{\Pi^{\bf m}_{\mathcal Q_z,z}}\node{\mathbb
    B(\H_z)}\arrow{s,r}{\Gamma_z}\\ \node{\mathscr C^{\bf m}_{\mathcal Q_z^{\rm n}}}\arrow{e,t}{\cong}\node{\mathscr C^{\bf
      m}_{\mathcal Q_z}/\mathscr C_{\mathcal  O_z}}\arrow{e,t}{\tilde\Pi^{\bf m}_{\mathcal Q_z,z}}\node{\mathbb B(\H_z)/\mathbb K(\H_z)}
\end{diagram}
\end{equation}
Then one writes
\begin{equation*}
\begin{aligned}
  {\sf sp}_{\rm ess}(H_z)& = {\sf sp}\big[\Gamma_z(H_z)\big]={\sf sp}\big[\Gamma_z\big(\Pi_z^{\bf m}(F)\big)\big]\\ &={\sf
    sp}\big[\Gamma_z\big(\Pi^{\bf m}_{\mathcal Q_z,z}[\rho^{\bf m}_{\mathcal Q_z}\!(F)]\big)\big]\\ &={\sf
    sp}\big[\tilde\Pi^{\bf m}_{\mathcal Q_z,z}\big(\gamma_z[\rho^{\bf m}_{\mathcal Q_z}\!(F)]\big)\big]\\ &={\sf sp}\big[\gamma_z(F_{\mathcal Q_z})\big]\,.
\end{aligned}
\end{equation*}
For the last equality, recall that $\tilde\Pi_{\mathcal Q_z,z}$ is injective.

\smallskip
Finally $\,{\sf sp}\Big(\gamma_z(F_{\mathcal Q_z})\!\mid\!\mathscr C^{\bf m}_{\mathcal Q_z}/\mathscr C_{\mathcal O_z}\Big)={\sf
  sp}\Big(F_{\mathcal Q_z^{\rm n}}\big\vert\mathscr C^{\bf m}_{\mathcal Q_z^{\rm n}}\Big)$ follows from the isomorphism
$\mathscr C^{\bf m}_{\mathcal Q_z}/\mathscr C_{\mathcal O_z}\!\cong\mathscr C^{\bf m}_{\mathcal Q_z^{\rm n}}$\,, under which
$\gamma_z(F_{\mathcal Q_z})$ corresponds to $F_{\mathcal Q_z^{\rm n}}$\,; this finishes the proof of (i).

\smallskip
The point (ii) follows from (i) and from Theorem \ref{garefuly} applied to the close invariant set $Y\!:=\mathcal Q_z^{\rm n}=\mathcal
Q_z\!\setminus\!\mathcal O_z\subset\mathcal Q_z$\,; see also Remark \ref{camsuy}.

\smallskip
Atkinson's Theorem states that an operator $H$ is Fredholm if and only if it is invertible modulo the compact operators, i.e.\! if and only
if $0$ does not belong to its essential spectrum.  In \eqref{jardiancess}, ${\sf s}\ne 0$ means precisely that $F_{\mathcal
  Q_z}\!\in\mathscr C^{\sf m}_{\mathcal Q_z}\!\setminus\!\mathscr C_{\mathcal Q_z}$.  Then (iii) follows from (ii) by taking into
account the following general simple observation: Let $\{\mathscr D_k\}_{k\in K}$ be a family of $C^*$-algebras and $a_k\in\mathscr D_k$
for every $k\in K$; then one has $0\notin{\bigcup_{k\in K}{\sf sp}\big(a_k\big)}$ if and only if each $a_k$ is invertible.
\end{proof}

\begin{Example}\label{camsui}
{\rm Let $z \in X$ be such that its $\Xi$-orbit $\mathcal O_z$ is closed and free. Then all the points of the orbit $\mathcal O_z$ are
  generic, and thus the unions in \eqref{jardiancess} are all void. In this case, we simply get ${\sf sp}_{\rm ess}(H_z)=\{{\sf s}\}$\,,
  which also follows directly from $\mathcal Q_z=\mathcal O_z$\,. If  $F\in\mathscr C$, then ${\sf sp}_{\rm ess}(H_z)=\{0\}$\,. This
  corresponds to the fact that, in this case, the operator $H_z$ is compact; see Example \ref{cuartet}.  }
\end{Example}

In particular, we are interested in {\it standard} twisted groupoids $(\Xi,\o)$ (cf. Definition \ref{lista2}) and, for an element $F=G+{\sf
  s}1\in{\sf C}^*(\Xi,\o)^{\bf m}\equiv\mathscr C^{\bf m}$, in its image through the vector representation, cf. Remark \ref{idem}. So
$H_0:=\Pi^{\bf m}_0(F)$ is a bounded operator in $\H_0:=L^2(M,\mu)$\,; here $M$ is the main orbit (open and dense). We state a consequence of
Theorem \ref{carefuly} on the essential spectrum; the corresponding Fredholm criteria can be easily formulated by the reader.

\begin{Corollary}\label{carefully}
Suppose that, in addition to being tractable, the twisted groupoid $\Xi \tto X$ is standard (see Definition \ref{lista2}), with
$X=M\sqcup X_\infty$\,, where $M=X^{\rm g}$ is the main orbit. Let $\,\{\mathcal Q_j\}_{j\in J}$ be a covering of $X_\infty$ with
quasi-orbits\,, and for each $j\in J$ let $x_j$ be a generic element of $\mathcal Q_j$\,.  Then
\begin{equation}\label{jardiances}
  {\sf sp}_{\rm ess}(H_0) = {\sf sp}\big(F_{X_\infty}\!\mid\!\mathscr C^{\bf m}_{X_\infty}\,\big) = \{{\sf s}\}\cup\bigcup_{j\in J}{\sf
    sp}\big(F_{\mathcal Q_j}\!\mid\!\mathscr C_{\mathcal Q_j}\big)=\{{\sf s}\}\cup\bigcup_{j\in J}{\sf
    sp}\big(H_{x_j}\big)=\{{\sf s}\}\cup\!\bigcup_{x\in X_\infty}\!{\sf sp}(H_x)\,.
\end{equation}
\end{Corollary}

\begin{proof}
This basically follows from Theorem \ref{carefuly}, since for standard groupoids, by definition, the entire unit space $X$ is a quasi-orbit
and the points $z$ belonging to the main orbit $M$ are regular. One must also replace $H_z:=\Pi^{\bf m}_z(F)$ of Theorem \ref{carefuly} by
the unitarily equivalent vector represented version $H_0\!:=\Pi^{\bf m}_0(F)$\,, as explained in Remark \ref{idem}.
\end{proof}

\subsection{Essential numerical ranges}\label{caloriforit}

We treat now various types of numerical ranges, in a version adapted to operator algebras \cite{BD,BD2,HRS,SW}. If $\mathscr D$ is a unital
$C^*$-algebra , then $\mathcal S(\mathscr D)$ denotes its {\em state space}, that is, the set of positive linear forms $\phi:\mathscr D\to\mathbb C$ such that $\phi(1_\mathscr D)=1$.

\begin{Definition}\label{statespate}
Let $\mathscr D$ be a unital $C^*$-algebra.  The (algebraic) {\em numerical range} of $F\in\mathscr D$ is
\begin{equation*}
  {\sf nr}(F) := \{\phi(F)\!\mid\phi\in\S(\mathscr D)\}\,.
\end{equation*}
\end{Definition}

The numerical range set ${\sf nr}(F)$ of $F\in\mathscr D$ is a compact, convex subset of $\mathbb C$ containing the convex hull of the spectrum:
\begin{equation*}
  {\rm co}[{\sf sp}(F)] \subset {\sf nr}(F) \subset \big\{\lambda \in \mathbb C\mid |\lambda| \le\, \|F\|_\mathscr D\!\big\}\,.
\end{equation*} 
For normal elements $F\in\mathscr D$ these are equalities, but in general the inclusions are strict. If $H\in\mathscr D=\mathbb B(\H)$
for some Hilbert space $\H$\,, then ${\sf nr}(H)$ is the closure of {\it the operator (spatial) numerical range} \cite{GR}
\begin{equation*}
{\sf nr}_0(H):=\big\{\<Hu,u\>_\H\,\big\vert\,u\in\H\,,\|u\|_\H\,=\!1\big\}\,.
\end{equation*}

For any subset $A$ of $\mathbb C$, we denote by $\overline{\rm co}(A)$ the smallest closed, convex set containing $A$\,.

\begin{Theorem}\label{garefly}
Let $\Xi$ be a tractable groupoid endowed with a continuous $2$-cocycle $\o$ and let $Y\subset X$ a closed invariant set of
units. Consider an element $F={\sf s} + G \in{\sf C}^*(\Xi,\o)^{\bf m}$, with ${\sf s}\in\mathbb C$ and $G\in{\sf C}^*(\Xi,\o)$\,, and
set $H_x\!:=\Pi_x^{\bf m}(F)={\sf s}+\Pi_x(G)\,$ for every unit $x$\,. Then
\begin{equation}\label{hardiangess}
  {\sf nr}(F_Y)=\overline{\rm co}\Big[\{{\sf s}\}\cup\bigcup_{x\in Y}{\sf nr}\big(H_x\big)\Big]\,.
\end{equation}
\end{Theorem}

\begin{proof}
It is convenient to use the abbreviations $\mathscr C_Y\!:={\sf C}^*\big(\Xi_Y,\o_Y\big)$ and $\mathscr B_x\!:=\mathbb B(\H_x)$\,.
To unify notations, let us set $F_\infty := s$ and $Y^{\bf m}:=Y\sqcup\{\infty\}$\,. The main equation, Equation
\eqref{hardiangess} may hence be written as ${\sf nr}(F_Y)=\overline{\rm co}\Big[\bigcup_{x\in Y^{\bf m}}\!{\sf nr}\big(H_x\big)\Big]$\,.  Let us also set $\mathscr
B_\infty\!:=\mathbb C$, and
\begin{equation*}
  \Pi^{\bf m}_{\infty}:\mathscr C_Y^{\bf m}\to\mathscr C^{\bf m}_Y/\mathscr C_Y\cong\mathbb C\,,\quad\Pi^{\bf
    m}_{Y,\infty}\big({\sf t}+G^0\big)\!:={\sf t}\,.
\end{equation*}
Note that ${\sf nr}\big[\Pi^{\bf m}_{\infty}(F_Y)\big]=\{{\sf s}\}$\,, so \eqref{hardiangess} may be rewritten ${\sf nr}(F_Y)=\overline{\rm co}\Big[\bigcup_{x\in Y^{\bf m}}\!{\sf nr}\big(H_x\big)\Big]$\,.

\smallskip
The inclusion ${\sf nr}(F_Y)\supset\overline{\rm co}\Big[\bigcup_{x\in Y^{\bf m}}\!{\sf nr}\big(H_x\big)\Big]$ follows from the fact that
the algebraical numerical range shrinks under $C^*$-morphisms, as shown in the proof of Theorem 4.4 of \cite{MN}. For a given $x\in
Y^{\bf m}$ this should be applied to $H_x=\Pi^{\bf m}_{Y,x}(F_Y)$\,. Then we use the fact that ${\sf nr}(F_Y)$ is known to be compact and convex.

\smallskip
For the opposite inclusion, we use the following criterion \cite[Th.2]{SW}, valid for any element $G$ of a unital $C^*$-algebra
$\mathscr C$\,: {\it Let $K\subset\mathbb C$ be closed and convex; then ${\sf nr}(G)\subset K$ if and only if }
\begin{equation*}
  \big\Vert(G-\lambda)^{-1}\big\Vert_\mathscr C\le{\rm dist}(\lambda,K)^{-1},\quad\forall\,\lambda\notin K.
\end{equation*}

We recall from the proof of Theorem \ref{garefuly} that the family of morphisms
\begin{equation*}
   \big\{\Pi^{\bf m}_{Y,x}\!:\mathscr C_Y^{\bf m}\!\to\mathscr B_x\!\mid\!x\in Y^{\bf m}\big\}
\end{equation*}
is strictly norming (this is a consequence of Corollary \ref{feminism}). Thus, for any $\lambda\in\mathbb C$ there exists $y(\lambda)\in Y^{\bf m}$ such that
\begin{equation}\label{faradecare}
  \big\Vert(F_Y-\lambda)^{-1} \big \Vert_{\mathscr C_Y^{\bf m}} = \big\Vert\big(H_{y(\lambda)} - \lambda\big)^{-1} \big\Vert_{\mathscr B_{y(\lambda)}} .
\end{equation}
Then, for an arbitrary convex compact set $K$, using \eqref{faradecare} and the criterion:
\begin{equation*}
\begin{aligned}
  \overline{\rm co} \Big[\bigcup_{x\in Y^{\bf m}} \! {\sf nr}(H_x)\Big]\subset K& \Longleftrightarrow \bigcup_{x\in Y^{\bf m}}\! {\sf nr}(H_x)\subset K\\
  & \Longleftrightarrow \big \Vert(H_x-\lambda)^{-1}\big\Vert_{\mathscr B_x} \le {\rm
    dist}(\lambda,K)^{-1},\ \forall\,\lambda\notin K\,, x\in Y^{\bf m}\\
  & \Longrightarrow \big\Vert(H_{y(\lambda)}-\lambda)^{-1}\big\Vert_{\mathscr B_{y(\lambda)}} \le {\rm  dist}(\lambda,K)^{-1},\ \forall\,\lambda\notin K\\
  & \Longleftrightarrow \big\Vert(F_Y-\lambda)^{-1}\big\Vert_{\mathscr C_Y^{\bf m}}\! \le {\rm dist}(\lambda,K)^{-1},\ \forall\,\lambda\notin K\\
  &\Longleftrightarrow {\sf nr}(F_Y)\subset K.
\end{aligned}
\end{equation*}
It follows that ${\sf nr}(F_Y)\subset\overline{\rm co}\Big[\bigcup_{x\in Y^{\bf m}}\!{\sf nr}\big(H_x\big)\Big]$\,.
\end{proof}

\begin{Remark}\label{inutila}
\normalfont{Simple examples show that the union in the right hand side of \eqref{hardiangess} is not convex in general. Take for instance
  $X=Y=\{\e_1,\e_2\}$ and $\Xi_1=\T=\Xi_2$\,, leading to a very simple minded group bundle.  Then
\begin{equation*}
{\sf C}^*(\Xi)\cong{\sf C}^*(\Z)\oplus{\sf C}^*(\Z)\cong C(\T)\oplus C(\T)\,,
\end{equation*} 
and it is already unital. If $g\in C(\T)$ is a real function, it is a self-adjoint element and its numerical range is a real segment $I=g(\T)$\,. Then consider the element $F:=(g,ig)$\,.  }
\end{Remark}

\begin{Definition}\label{esnumer}
Let $H$ be an element of a unital $C^*$-subalgebra $\mathscr D$ of $\,\mathbb B(\H)$ for some Hilbert space $\H$\,. Its ({\rm
  algebraical) essential numerical range} ${\sf nr}_{\rm ess}(H)$ is the (algebraic) numerical range of the canonical image of $H$ in the quotient $\mathscr D/\mathscr D\cap\mathbb K(\H)$\,.
\end{Definition}

The essential numerical range contains the essential spectrum (often strictly). We refer to \cite{BD2} for more information on this topic.

\begin{Theorem}\label{carefree}
Let $\Xi$ be a tractable groupoid endowed with a continuous $2$-cocycle $\o$ and consider an element $F={\sf s}+G\!\in{\sf C}^*(\Xi,\o)^{\bf m}$\,, with $G\in\mathscr C:={\sf C}^*(\Xi,\o)$
and ${\sf s}\in\mathbb C$\,. Let $z$ be a regular element of the unit space $X=\Xi^{(0)}$. One has
\begin{equation}\label{jardanche}
  {\sf nr}_{\rm ess}(H_z) = {\sf nr}\big(F_{\mathcal Q_z^{\rm n}}\big\vert\mathscr C^{\bf m}_{\mathcal Q_z^{\rm n}}\big)=\overline{\rm co}\Big[\{{\sf
      s}\}\cup\!\bigcup_{x\in\mathcal Q_z^{\rm n}}\!{\sf nr}\big(H_x\big)\Big]\,.
\end{equation}
\end{Theorem}

\begin{proof}
As in the proof of Theorem \ref{carefuly}, the first equality in \eqref{jardanche} follows from the way the essential numerical range has been defined and from the inclusion
\begin{equation}\label{bazilicac}
  \mathscr C^{\bf m}_{\mathcal Q_z^{\rm n}}\cong\mathscr C^{\bf m}_{\mathcal Q_z}/\mathscr C_{\mathcal O_z}\hookrightarrow\mathbb B(\H_z)/\mathbb K(\H_z)\,.
\end{equation}

The second equality in \eqref{jardanche} follows from Theorem \ref{garefly} applied to the close invariant set $Y\!:=\mathcal
Q_z^{\rm n}=\mathcal Q_z\!\setminus\!\mathcal O_z\subset\mathcal Q_z$\,.
\end{proof}

\begin{Remark}\label{cuorbite}
\normalfont{For simplicity, we skipped assertions involving restrictions to quasi-orbits and operators associated to selections
  of generic points of these quasi-orbits, as in Theorems \ref{garefuly} and \ref{carefuly}; they may be easily provided by the reader.  }
\end{Remark}

\subsection{Absence of the discrete spectrum}\label{gomitte}

We start from the following simple result:

\begin{Lemma}\label{laball}
Let $\mathscr D$ be a unital $C^*$-algebra of linear bounded operators in the Hilbert space $\H$\,. The following assertions are equivalent:
\begin{enumerate}
\item[(a)] 
One has $\mathscr D\cap\mathbb K(\H)=\{0\}$\,.
\item[(b)] All the elements of $\mathscr D$ have void discrete spectrum, i.e. ${\sf sp}_{\rm ess}(H)={\sf sp}(H)$\,, for every $H\in\mathscr D$.
\item[(c)] 
The operator $H\in\mathscr D$ is Fredholm if and only if it is invertible.
\item[(d)] One has ${\sf nr}_{\rm ess}(H)={\sf nr}(H)$ for every $H\in\mathscr D$.
\end{enumerate}
\end{Lemma}

\begin{proof}
$(b)\!\Rightarrow\!(a).$ Suppose that $\mathscr D$ contains a non-zero  compact operator $H$. This operator surely has some discrete spectrum, contradicting the hypothesis.

\smallskip
$(a)\!\Rightarrow\!(b).$ The essential spectrum of $H\in\mathscr D$ coincides with the usual spectrum of its canonical image in $\mathscr
D/\mathscr D\cap\mathbb K(\H)$\,, which now is just $\mathscr D$. Therefore ${\sf sp}(H)={\sf sp}_{\rm ess}(H)$\,.

\smallskip
$(b)\!\Leftrightarrow\!(c).$ This follows from the fact that $\lambda\notin{\sf sp}_{\rm ess}(L)$ if and only of $L-\lambda$ is
Fredholm and $\lambda\notin{\sf sp}(L)$ if and only if $L-\lambda$ is invertible.

\smallskip
$(a)\!\Leftrightarrow\!(d).$ This is similar to $(a)\!\Leftrightarrow\!(b);$ one has to use the definition of the
essential numerical range and the known fact that if ${\sf nr}_{\rm ess}(H)=\{0\}$ if and only if $H$ is compact \cite{BD2}.
\end{proof}

A representation of a $C^*$-algebra $\Pi:\mathscr C\to\mathbb B(\H)$ is called {\it essential} \cite[Page 19]{BO} if one has $\,\Pi(\mathscr C)\cap\mathbb K(\H)=\{0\}$\,.

\smallskip
We continue to assume that {\it $\o$ is a $2$-cocycle on a tractable groupoid $\Xi$} and we use all the previous notations.  Recalling
that $\Pi_{x}=\Pi_{\mathcal Q_x,x}\circ\rho_{\mathcal Q_x}$, one can state an obvious consequence of Lemma \ref{laball}.

\begin{Corollary}\label{oareofi?}
For $x\in X$ the following assertions are equivalent:
\begin{enumerate}
\item[($a_1$)] The representation $\Pi_x:{\sf C}^*(\Xi,\o)\to\mathbb B(\H_x)$ is essential.
\item[($a_2$)] The faithful representation $\Pi_{\mathcal Q_x,x}:{\sf C}^*\big(\Xi_{\mathcal Q_x},\o_{\mathcal Q_x}\big)\to\mathbb B(\H_x)$ is essential.
\item[(b)] The discrete spectrum of $\,\Pi_x(F)$ is void for every $F\in {\sf C}^*(\Xi,\o)$\,.
\item[(c)] For each $F\in{\sf C}^*(\Xi,\o)$ the operator $\Pi_x(F)$ is Fredholm if and only if it is invertible.
\item[(d)] One has ${\sf nr}_{\rm ess}\big[\Pi_x(F)\big]={\sf nr}\big[\Pi_x(F)\big]$ for every $F\in {\sf C}^*(\Xi,\o)$\,.
\end{enumerate}
\end{Corollary}

\begin{Remark}\label{sichichef}
\normalfont{For any $H\in\mathbb B(\H)$ it is known that ${\sf nr}_{\rm ess}(H)\!=\bigcap_{K\in\mathbb K(\H)}{\sf
    nr}(H+K)$\,. Hence the point (d) of the Corollary can be rewritten as
\begin{equation*}
  {\sf nr}\big[\Pi_x(F)+K\big]\supset{\sf nr}\big[\Pi_x(F)\big]\,,\quad\forall\,F\in {\sf C}^*(\Xi,\o)\,,\ K\in\mathbb K(\H_x)\,.
\end{equation*}}
\end{Remark}

A given $C^*$-algebra might only have {\it essential} faithful representations, and this can be applied to $\Pi_{\mathcal Q_x,x}$\,,
which is faithful. To arrive at such a situation, some more terminology is needed.

\smallskip
It will be convenient to assume that $X=\Xi^{(0)}$ is compact.  We say that the (closed) invariant subset $A$ of $X$ is {\it minimal} if all
the orbits contained in $A$ are dense.  Equivalently: $A$ and $\emptyset$ are the only closed invariant subsets of $A$\,. The point
$x\in X$ is said to be {\it almost periodic} if $\,{\sf V}_{x,U}\!:=\{\xi\in\Xi_x\!\mid\!\xi x\xi^{-1}\!\in U\}$ is syndetic
for every neighborhood $U$ of $x$ in $X$, i.e. ${\sf KV}_{x,U}=\Xi_x$ for some compact subset ${\sf K}$ of $\,\Xi$\,. In
\cite[Appendix]{MN}, assuming that $X$ is compact, it is shown that the unit $x$ is an almost periodic point if and only if its
quasi-orbit $\mathcal Q_x$ is minimal. A groupoid is called {\it topologically principal} if the units with trivial isotropy form a
dense subspace of the unit space. It is {\it \'etale} if ${\rm d}:\Xi\to X$ is a local homeomorphism (the same will hold for ${\rm r}:\Xi\to X$, of course).

\begin{Proposition}\label{simplicity}
Assume that $x$ is almost periodic, $\,\Xi_{\mathcal Q_x}$ is \'etale and topologically principal and the fiber $\Xi_x$ is infinite. Then
the representation $\Pi_x$ is essential and for every $F\in{\sf C}^*(\Xi,\o)$ the operator $\Pi_x(F)$ has purely essential spectrum
and its essential numerical range coincides with the usual one.
\end{Proposition}

\begin{proof}
By Corollary \ref{oareofi?}, it is enough to show that $\Pi_{\mathcal Q_x,x}$ is essential.

\smallskip
But {\it all the unital infinite-dimensional representations $\pi:\mathscr C\to\mathbb B(\H_\pi)$ of a simple unital $C^*$-algebra are (faithful and) essential}. To check this, suppose
that $\pi(\mathscr C)\cap\mathbb K(\H_\pi)\ne\{0\}$\,. It is an ideal of $\pi(\mathscr C)$\,, thus $\pi^{-1}\big[\pi(\mathscr C)\cap\mathbb
  K(\H_\pi)\big]\ne\{0\}$ is an ideal of $\mathscr C$. Since $\mathscr C$ is simple, one has $\pi^{-1}\big[\pi(\mathscr C)\cap\mathbb
  K(\H_\pi)\big]=\mathscr C$, implying that $\pi(\mathscr C)\cap\mathbb K(\H_\pi)=\pi(\mathscr C)$\,, i.e. $\pi(\mathscr
C)\subset\mathbb K(\H_\pi)$\,. Then $\pi(1_\mathscr C)={\rm id}_{\H_\pi}$ is a compact operator, which is impossible since $\H_\pi$ in infinitely dimensional.

\smallskip
So we have to prove all these properties for $\Pi_{\mathcal Q_x,x}:{\sf C}^*\big(\Xi_{\mathcal Q_x},\o_{\mathcal Q_x}\big)\to\mathbb B(\H_x)$\,.

\smallskip
Since $\mathcal Q_x$ is compact and $\Xi_{\mathcal Q_x}$ is \'etale, ${\sf C}^*\!\big(\Xi_{\mathcal Q_x},\o_{\mathcal Q_x}\big)$ is
unital. Actually in this case $\mathcal Q_x$ is an open compact subset of $\Xi_{\mathcal Q_x}$\,, one has $C(\mathcal Q_x)\hookrightarrow
C_0(\Xi_{\mathcal Q_x})\subset{\sf C}^*\big(\Xi_{\mathcal Q_x},\o_{\mathcal Q_x}\big)$ and the constant function $1:\mathcal
Q_x\to\mathbb C$ (reinterpreted as the characteristic function of $\mathcal Q_x$ defined on $\Xi_{\mathcal Q_x}$) becomes the unit of
the twisted groupoid algebra. It is sent by $\Pi_{\mathcal Q_x,x}$ into ${\rm id}_{\H_x}$\,. The fibre $\Xi_x$ is discrete and infinite, hence the space $\H_x=\ell^2(\Xi_x)$ is infinite dimensional.

\smallskip
As said above, $x$ being almost periodic, the quasi-orbit $\mathcal Q_x$ is minimal. Also using the fact that $\Xi_{\mathcal Q_x}$ is
\'etale and topologically transitive, by \cite[pag103]{Re}, the twisted $C^*$-algebra of the reduced groupoid $\Xi_{\mathcal Q_x}$ is simple.  The proof is finished.
\end{proof}

\begin{Remark}\label{sanusperi}
  \normalfont{In \cite{BCFS} it is proven that, for untwisted \'etale groupoids, the simplicity of the groupoid algebra implies that the groupoid is minimal and topologically principal.  }
\end{Remark}

Some results in \cite{BLLS1,BLLS2} (see also references therein) treat the absence of the discrete spectrum in a setup presenting some
similarities with ours. Using the theory of limit operators, they also treat the Banach space case, which is inaccessible to us. Restricted
to the Hilbert space setting, they basically work with minimal actions $\th$ of discrete countable groups $\G$ on compact Hausdorff spaces
$X$. This would correspond to the \'etale transformation groupoid $\Xi:=X\!\rtimes_\th\G$\,, on which there is no $2$-cocycle. But no
condition as topological principalness is in used. So there are some improvements but also some drawbacks in our Proposition \ref{simplicity}.

\smallskip
In \cite{KPS} a precise criterion of simplicity is given for twisted $C^*$-algebras of row-finite $k$-graphs with no sources. Via our
Proposition \ref{simplicity}, this would lead to the absence of the discrete spectrum for operators obtained from twisted Cuntz-Krieger
families. The constructions of \cite{KPS} are too involved to be reproduced here.

\section{Examples: magnetic operators on nilpotent Lie groups}\label{sec.four}

\subsection{Magnetic pseudo-differential operators}\label{gerfomoreni}

Let us fix a connected, simply connected nilpotent Lie group $\G$ with unit $\e$\,. If $Y \in \g$ and $\mathcal X \in \g^\sharp$ (the dual of
the Lie algebra $\g$)\,, we set $\<Y\!\!\mid\!\!\mathcal X\>:=\mathcal X(Y)$\,. The exponential map $\exp:\mathfrak g\to\G$ is a
diffeomorphism, with inverse $\log:\G\rightarrow\mathfrak g$\,, sending Haar measures $dY$ on $\g$ to a Haar measures $da$ on
$\G$\,. We proceed to construct magnetic pseudo-differential operators in this context.

\smallskip
Let $B$ be a magnetic field, i.e. a closed $2$-form on $\G$\,, seen as a smooth map associating to any $a\in\G$ the skew-symmetric bilinear
form $B(a):\g\times\g\rightarrow\mathbb R$\,. One more ingredient is available, making use of the de Rham cohomology. A $1$-form $A$ on
$\G$\,, also called {\it vector potential}, will be seen as a (smooth) map $A:\G\rightarrow\g^\sharp$ and it gives rise to a $1$-form
$A\circ\exp:\g\rightarrow\g^\sharp$. Being a closed $2$-form, the magnetic field can be written as $B=dA$ for some $1$-form $A$\,. Any
other vector potential $\tilde A$ satisfying $B=d\tilde A$ is related to the first by $\tilde A=A+d\psi$\,, where $\psi$ is a smooth function on $\G$\,; in physics this is called {\it gauge covariance}.

\smallskip
For $a,b\in\G$, one sets $[a,b]:\mathbb R\rightarrow\G$ by
\begin{equation*}\label{abb}
  [a,b]_s := \exp[(1-s)\log a+s\log b]=\exp[\log a+s(\log b-\log a)]\,.
\end{equation*}
The function $[a,b]$ is smooth and satisfies $[b,a]_s=[a,b]_{1-s}$\,, $[a,b](0)=a$, and $[a,b](1)=b$\,. In addition, $[\e,y]$ is the $1$-parameter subgroup passing through $y$\,.  {\it The segment in
  $\G$ connecting $a$ to} $b$ is $[[a,b]]:=\big\{\,[a,b]_s\mid s\in[0,1]\,\big\}$\,. {\it The circulation of the vector potential $A$ through the segment $[[a,b]]$} is the real number
\begin{equation*}\label{circ}
  \Gamma_A[[a,b]] \equiv \int_{[[a,b]]}\!A := \int_0^1\!\big\<\log b-\log a\,\big\vert\, A\big([a,b]_s\big)\big\>\,ds\,.
\end{equation*}

To conclude the construction, for every vector potential $A$ with $B=dA$\,, the outcome is a pseudo-differential correspondence
\cite[Sect.4]{BM}, assigning to suitable symbols $\si:\G\times\g^\sharp\to\mathbb C$ {\it magnetic pseudo-differential operators}
\begin{equation}\label{frame}
\big[\mathfrak{Op}_{A}(\si)v\big](a) := \int_{\G}\int_{\mathfrak g^\sharp} e^{i\<\log(ab^{-1})\mid\mathcal
  X\>}e^{i\int_{[[b,a]]}\!A}\,\si\big(b,\mathcal X\big)v(b)\,db\,d\mathcal X,
\end{equation}
acting in $\mathcal S(\G)$ (see below) or in $L^2(\G)$\,. Gauge covariance extends to these operators: if two vector potentials $A,A'$
define the same magnetic field, then $\mathfrak{Op}_{A}(\si)$ and $\mathfrak{Op}_{A'}(\si)$ are unitarily equivalent.

\begin{Remark}\label{dinozaur}
{\rm If $\G=\mathbb{R}^n$, the dual group can be identified with the vector space dual $\mathbb{R}^n$. It is also true that $\R^n$ is
  identified with its Lie algebra and (then) with its dual, so in this case the maps $\exp$ and $\log$ simply disappear from the formulas. The resulting expression
\begin{equation*}\label{framek}
  \big[\mathfrak{Op}_{A}(\si)v\big](a)=\int_{\R^n}\!\int_{\R^n}e^{i\<a-b\mid\mathcal X\>} e^{i\int_{[[b,a]]}\!A}\,\si\big(b,\mathcal X\big)v(b)\,db\,d\mathcal X,
\end{equation*}
a magnetic modification of the right quantization, is the starting point of a pseudo-differential theory studied in depth in \cite{IMP,MP,MPR2} and other articles ($\tau$-quantizations can also
be accomodated, including the Kohn-Nirenberg and the Weyl forms).  }
\end{Remark}

To model the behavior of $B$ at infinity, and to be more explicit on the class of symbols $\si$, let us consider a $\G$-equivariant
compactification $X = \G\sqcup X_\infty$ of $\G$\,. More precisely, we assume that $\G$ is a dense, open subset of the second countable,
Hausdorff, compact space $X$. Thus the action of $\G$ on itself by left translations extends to a continuous action $\th$ of $\G$ on
$X$. It is useful to regard $X$ as the Gelfand spectrum of the Abelian $C^*$-algebra $C(X)$\,. This one is isomorphic to the $C^*$-algebra
$\A$ formed by the restrictions of all the elements of $C(X)$ to $\G$\,. Clearly, $\A$ is a $C^*$-algebra composed of bounded and
uniformly continuous complex functions defined on $\G$\,, invariant under left translations and and containing $C_0(\G)$\,. The condition
on $B$ is simply that it extends continuously to $X$ or, equivalently, that its components in a base are elements of $\A$\,. {\it Note that
  such an assumption is not imposed on vector potentials $A$ generating the given magnetic field.} Actually this would be
impossible in most situations. Very often there is not even a bounded vector potential with $B=dA$\,; think of constant or periodic magnetic fields, for instance.

\smallskip
We denote by $\S(\g^\sharp)$ the Schwartz space on the vector space $\g^\sharp$ and by $\A \otimes_{\pi}\!\S(\g^\sharp)$ the projective
tensor product. A function $\si \in \A \otimes_{\pi}\!\S(\g^\sharp)$\,, a priori defined on $\G \times \g^\sharp$, has a
unique continuous extension to $X\times \g^\sharp$ that we shall use without any further comment, as in \eqref{defacut} below.  We then
have the following theorem that describes the essential spectrum of "order $-\infty$ magnetic pseudodifferential operators".

\begin{Theorem}\label{specmagnet}
Let $\si \in C(X) \otimes_{\pi} \S(\g^\sharp)$\,. For each $x\in X_\infty$ we set
\begin{equation*}\label{defacut}
  \si_x:\G\times\g^\sharp\to\mathbb C\,,\quad\si_x(a,\mathcal X):=\si\big(\th_a(x),\mathcal X\big)\,,
\end{equation*}
as well as $B_x(a):=B\big[\th_a(x)\big]$\,, and choose a vector potential $A_x$ such that $B_x=dA_x$\,. One has
\begin{equation}\label{firear}
  {\sf sp}_{\rm ess}\big[\mathfrak{Op}_{A}(\si)\big]=\bigcup_{x\in X}{\sf sp}\big[\mathfrak{Op}_{A_x}\!(\si_x)\big]
\end{equation}
and
\begin{equation}\label{firer}
  {\sf nr}_{\rm ess}\big[\mathfrak{Op}_{A}(\si)\big] = \overline{\rm co}\Big(\bigcup_{x\in X}{\sf nr}\big[\mathfrak{Op}_{A_x}\!(\si_x)\big]\Big)\,.
\end{equation}
\end{Theorem}

The remaining part of this subsection is basically dedicated to proving Theorem \ref{specmagnet}. To do this, we will reformulate the
magnetic pseudo-differential operators \eqref{frame} in order to fit a twisted groupoid framework and next we will apply the general spectral results obtained in the previous section.

\smallskip
The groupoid will be {\it the transformation groupoid} $\,\Xi\!:=\!X\!\rtimes_\th\!\G$\,. It coincides with $X\!\times\!\G$ as a topological space, and the algebraic structure is defined through
\begin{equation*}
  {\rm d}(x,a) := x\,,\quad {\rm r}(x,a):=\th_a(x)\,,\quad (x,a)^{-1}:=\big(\th_a(x),a^{-1}\big)\,,\quad
  \big(\th_a(x),b\big)(x,a):=(x,ba)\,.
\end{equation*}
It is easy to see that it is tractable; for amenability use \cite[Ch.II,\,Prop.3.9]{Re} and the amenability of $\G$\,. It is even
a standard groupoid, with main orbit $M=\G$\,. If $\m$ is a Haar measure on $\G$ (nilpotent groups are unimodular), 
then $\{\delta_x \otimes \m \mid x \in X\}$ is a Haar system on $X\!\rtimes_\th\!\G$\,.

\smallskip
We convert now the magnetic field into a groupoid cocycle. Let us set
\begin{equation*}
  \Delta := \big\{\,(t,s)\in[0,1]^2\,\big\vert\, s\le t\,\big\}\,.
\end{equation*}
For $a,b,c\in\G$ one defines the function $\<a,b,c\>:\R^2\rightarrow\G$
\begin{equation*}\label{function}
\<a,b,c\>_{t,s} := \exp\big[\log a+t(\log b-\log a)+s(\log c-\log b)\big]
\end{equation*}
and the set $\,\<\!\<a,b,c\>\!\>:=\<a,b,c\>_\Delta\subset\G$\,. Note that
\begin{equation*}\label{potop}
  \<a,b,c\>_{0,0} = a\,,\quad\<a,b,c\>_{1,0}=b\,,\quad\<a,b,c\>_{1,1}=c\,,
\end{equation*}
\begin{equation}\label{potrop}
  \<a,b,c\>_{t,0} = [a,b]_t\,,\quad\<a,b,c\>_{1,s}=[b,c]_s\,,\quad\<a,b,c\>_{t,t}=[a,c]_t\,,
\end{equation}
so the boundary of $\<\!\<a,b,c\>\!\>$ is composed of the three segments $[[a,b]]$\,, $[[b,c]]$ and $[[c,a]]$\,. One defines {\it the flux of $B$ through $\<\!\<a,b,c\>\!\>$ by}
\begin{equation*}\label{eguation}
  \Gamma_B\<\!\<a,b,c\>\!\> \equiv \int_{\<\!\<a,b,c\>\!\>}\!\!B\ := \int_0^1\!\int_0^t\!\,B\big(\<a,b,c\>_{t,s}\big)\big(\log a-\log b,\log a-\log c\big)dsdt\,.
\end{equation*}
If now $B=dA$ for some $1$-form $A$\,, Stokes' Theorem and \eqref{potrop} shows that for $q,a,b\in\G$ we have
\begin{equation}\label{circuflux}
\Gamma_B\<\!\<a,b,c\>\!\>=\Gamma_A[[a,b]]+\Gamma_A[[b,c]]+\Gamma_A[[c,a]]\,.
\end{equation}
It also follows from Stokes' Theorem and the closedness of $B$ that the formula
\begin{equation}\label{asta}
  \o_{B}\big((\th_a(x),b),(x,a)\big) := e^{i\Gamma_{\!B_x}\!\<\!\<\e,a,ba\>\!\>}
\end{equation}
defines a $2$-cocycle of the transformation groupoid. In \eqref{asta} we used the extension of $B$ to $X$ and the expression $B_x$ introduced in the statement of the Theorem.

\smallskip
As an output, one gets {\it the "magnetic" groupoid $C^*$-algebra} ${\sf C}^*\!\big(X\!\rtimes_\th\!\G,\o_B\big)$ constructed out of the
connected simply connected nilpotent group $\G$\,, a compactification $X$ and the magnetic field $B$ on $\G$ compatible with the
compactification. A Schwartz space on $\G$ is available; just push forward $\S(\g)$ through the exponential map, together with its
Fr\'echet space structure.  One can show the next chain of continuous dense inclusions:
\begin{equation*}\label{effort}
  C_{\rm c}(X\!\times\!\G) \subset C(X) \otimes_{\pi} \S(\G) \subset C(X) \otimes_{\pi} L^1(\G) \subset{\sf C}^*\!\big(X\!\rtimes_\th\!\G,\o_B\big)
\end{equation*}
involving projective tensor products. Recall the identification
$C(X)\cong\A$\,.

\smallskip
For further use, we write down the product in this case:
\begin{equation*}\label{zaproduct}
\begin{aligned}
  \big(f\star_{\o_B}\!g\big)(x,a)& = \int_\G f\big(\th_b(x),ab^{-1}\big)g(x,b)\,\o_B\big((\th_b(x),ab^{-1}),(x,b)\big)db\\
 & = \int_\G f\big(\th_b(x),ab^{-1}\big)g(x,b)\,e^{i\Gamma_{\!B_x}\!\<\!\<\e,b,a\>\!\>}db\,.
\end{aligned}
\end{equation*}

To make our way towards pseudo-differential operators with scalar-valued symbols, one takes advantage of the fact that $\G$ and
$\g$ are diffeomorphic through the exponential map; this allows defining a Fourier transform by
\begin{equation*}\label{clata}
  (\mathfrak F u)(\mathcal X):=\int_{\g}e^{-i\<X\mid\mathcal X\>}u(\exp X)\,dX =\int_\G\!\,e^{-i\<\log a\mid\mathcal X\>} u(a)da\,.
\end{equation*}
Then $\mathfrak F$ can be seen as a contraction $\,\mathfrak F:L^1(\G)\rightarrow L^\infty(\g^\sharp)$\,, as a unitary map
$\,\mathfrak F:L^2(\G)\rightarrow L^2(\g^\sharp)$ (for a convenient normalizations of the measures), or as a linear topological
isomorphism $\,\mathfrak F:\mathcal S(\G)\rightarrow\mathcal S(\g^\sharp)$ with inverse $(\mathfrak F^{-1}\mathfrak
u)(a):=\int_{\g^\sharp}\!e^{i\<\log a\mid\mathcal X\>} \mathfrak u(\mathcal X)\,d\mathcal X$. Then one uses the isomorphism
\begin{equation*}
  \mathbf F := {\rm id} \otimes_{\pi} \mathfrak F:C(X)\otimes_{\pi} \S(\G) \to C(X)\otimes\S(\g^\sharp)
\end{equation*} 
to transport the twisted groupoid algebra structure and then to generate twisted pseudo-differential operators.  By extension, this may be raised to a $C^*$-isomorphism
\begin{equation*}
  {\sf C}^*\!\big(X\!\rtimes_\th\!\G,\o_B\big)\cong{\sf C}^*(\G,X,\th,B)\supset C(X)\otimes\S(\g^\sharp)\,.
\end{equation*} 
Elements of $C(X)\otimes\S(\g^\sharp)$ will be identified to functions $\si:\G\times\g^\sharp\to\mathbb C$\,, to become symbols of the pseudo-differential operators \eqref{frame}.

\smallskip
We are not going to present details about the symbol $C^*$-algebra ${\sf C}^*(\G,X,\th,B)$ (basically an image via a partial Fourier
transform of the twisted groupoid algebra). It will be enough to identify the regular representations, composed with this partial Fourier transform.

\smallskip
We proceed now to the proof of Theorem \ref{specmagnet}.

\begin{proof}
The ${\rm d}$-fiber through the unit $x$ is $\Xi_x=\{x\}\times\G$\,. So
\begin{equation*}
  \iota_x:\H_x = L^2(\{x\}\!\times\!\G;\delta_x\! \otimes\!{\sf m})\to L^2(\G;{\sf m})\,, \quad\big[\iota_x(u)\big](a):=u(x,a)
\end{equation*}
is a Hilbert space isomorphism with inverse acting as $\big[\iota_x^{-1}(v)\big](x,a):=v(a)$\,. We denote by
\begin{equation*}
\mathcal I_x:\mathbb B(\H_x)\to\mathbb B\big[L^2(\G;{\sf m})\big]\,,\quad\mathcal I_x(T):=\iota_xT\iota_x^{-1}
\end{equation*}
the $C^*$-algebraic isomorphism induced by $\iota_x$\,.

\smallskip
The following diagram defines the mapping $\mathfrak{Op}_x$ in terms of the groupoid regular representation\,:
\begin{equation}\label{colifloch}
\begin{diagram}
\node{\quad\quad\quad\quad{\sf C}^*\!\big(X\!\rtimes_\th\!\G,\o_B\big)\supset C(X)\otimes\S(\G)}\arrow{s,r}{\mathbf
  F}\arrow{e,t}{\Pi_{x}}\node{\mathbb B(\H_x)}\arrow{s,r}{\mathcal I_x}\\ \node{\quad\quad\quad\quad{\sf C}^*(\G,X,\th,B)\supset
  C(X)\otimes\S(\g^\sharp)}\arrow{e,t}{\mathfrak{Op}_{x}}\node{\mathbb B[L^2(\G;{\sf m})]}
\end{diagram}
\end{equation}
We compute for $\si\in C(X)\otimes\S(\g^\sharp)$\,, $v\in L^2(\G;{\sf m})$ and $a\in\G$\,:
\begin{equation*}\label{pomelnic}
\begin{aligned}
  \big[\mathfrak{Op}_x(\si)v\big](a) & = \Big(\Pi_x\big[({\rm id}\otimes \mathfrak F^{-1}) \si\big]\big(\iota_x^{-1}v\big)\Big)(x,a)\\
  & = \Big(\big[({\rm id}\otimes\mathfrak F^{-1})\si\big]\star_{\o_B}\!\big(\iota_x^{-1}v\big)\Big)(x,a)\\
  & =\int_\G \big[({\rm id}\otimes\mathfrak F^{-1})\si\big]\big(\th_b(x), ab^{-1}\big)\big(\iota_x^{-1}v\big)(x,b)\, e^{i\Gamma_{\!B_x}\!\<\!\<\e,b,a\>\!\>}db\\
  & = \int_\G\int_{\g^\sharp} e^{i\<\log(ab^{-1})\mid\mathcal X\>}\,e^{i\Gamma_{\!B_x}\!\<\!\<\e,b,a\>\!\>} \si\big(\th_b(x), \mathcal X\big)v(b)db\\
  & = \int_\G\int_{\g^\sharp} e^{i\<\log(ab^{-1})\mid\mathcal X\>}\, e^{i\Gamma_{\!B_x}\!\<\!\<\e,b,a\>\!\>} \si\big(\th_b(x),\mathcal X\big)v(b)db\,.
\end{aligned}
\end{equation*}
Using \eqref{circuflux}\,, this may be written
\begin{equation}\label{pomelnoc}
  e^{i\Gamma_{\!A_x}\[\![\e,a\]\!]} \big[\mathfrak{Op}_x(\si)v\big](a) = \int_\G\int_{\g^\sharp} e^{i\<\log(ab^{-1})\mid\mathcal
    X\>}\,e^{i\Gamma_{\!A_x}\[\![b,a\]\!]}\si\big(\th_b(x),\mathcal X\big)e^{i\Gamma_{\!A_x}\[\![\e,b\]\!]}v(b)db\,.
\end{equation}
For any $x\in X$, the multiplication operator
\begin{equation*}
  U_{A_x}:L^2(\G)\to L^2(\G)\,,\quad\big(U_{A_x} w\big)(c):=e^{i\Gamma_{\!A_x}[\![\e,c]\!]}w(c)
\end{equation*}
is unitary and \eqref{pomelnoc} can be rewritten in the form $U_x\,\mathfrak{Op}_x(\si)=\mathfrak{Op}_{A_x}\!(\si_x)\,U_x$. Recalling diagram \eqref{colifloch}, one finally has
\begin{equation*}\label{fantazia}
  \mathfrak{Op}_{A_x}\!(\si_x) = U_{A_x}\,\mathfrak{Op}_x(\si)\, U_{A_x}^{-1} = \big(U_{A_x}\iota_x\big)\Pi_x\big[\mathbf F^{-1}(\si)\big]\big(U_{A_x}\iota_x\big)^{-1}
\end{equation*}
The conclusion is that {\it our magnetic pseudo-differential operator $\mathfrak{Op}_{A_x}\!(\si_x)$ is unitarily equivalent to }
\begin{equation*}
  H_x := \Pi_x(f)\,,\quad f := \mathbf F^{-1}(\si)\in C(X)\otimes\S(\g) \subset{\sf C}^*\!\big(X\!\rtimes_\th\!\G,\o_B\big)\,.
\end{equation*}
In particular, setting $x:=\e\in\G$\,, we see that $\mathfrak{Op}_{A}(\si)$ is unitarily equivalent to $H_\e:=\Pi_\e(f)$\,, for the same $f$. In its turn, $\Pi_\e(f)$ is
unitarily equivalent to $H_0\!:=\Pi_0(f)$\,, by Remark \ref{idem}\,, since $\G\equiv M$ is the main orbit of our groupoid.

\smallskip
Thus, the general spectral results from subsections \ref{ameny} and \ref{caloriforit} imply our Theorem \ref{specmagnet}. The point ${\sf
  s}=0$ need not be included explicitly in \eqref{fierar} and \eqref{firer}, since it automatically belongs to the right hand
sides. For simplicity, our elements were chosen in the twisted groupoid algebras, and not in the unitalizations.



\end{proof}

\subsection{Partial actions by restrictions and their twisted groupoid algebras}\label{gerfomoni}

We keep the setting and the notations of the preceding subsection. In particular, we are given a continuous action $\th$ extending the left
translations of the connected simply connected nilpotent group $\G$ to an equivariant compactification $X:=\G\sqcup X_\infty$ of $\G$\,, as
well as a magnetic field $B$ on $\G$ such that its components belong to the $C^*$-algebra $\A$\,, i.e. they extend continuously to $X$. Vector potentials $A$ such that $B=dA$ will also be used.

\begin{Assumption}\label{assumption}
We are going to fix a compact subset ${\sf K}\subset\G$ and denote by $\L$ the closure in $\G$ of the open set $\G\!\setminus\!{\sf
  K}$\,. Then $L\!:={\sf L}\sqcup X_\infty$ is a compact subset of $X$. We shall typically denote elements of $\,{\sf L}\subset\G\subset
X$ by the symbols $x,y,z\in X$, $a,b,c\in\G$ and $p,q,r\in{\sf L}$.
\end{Assumption}

Besides \eqref{frame}, we also consider {\it compressed magnetic pseudo-differential operators} acting in $L^2(\L)$. They are defined by the formula:
\begin{equation*}\label{frameme}
  \big[\mathfrak{Op}^{\sf L}_{A}(\si)w\big](p) := \int_{\sf L}\int_{\mathfrak g^\sharp} e^{i\<\log(pq^{-1})\mid\mathcal X\>}e^{i\int_{[[q,p]]}\!A}\,\si\big(s,\mathcal
  X\big)w(q)\,dq\,d\mathcal X,\quad p\in\L\,.
\end{equation*}
The only difference between this formula and Equation \eqref{frame} is that now we restrict $p$ and $q$ to the non-invariant subset $\L$\,.

\begin{Remark}\label{pasaj}
  \normalfont{One identifies $L^2({\sf L})\equiv L^2\big(\L;\m|_{\sf L}\big)$ with a closed subspace of the Hilbert space $L^2(\G)$\,. Setting $J_{\sf L}:L^2({\sf L})\to L^2({\sf G})$ for
    the canonical injection, the adjoint $J_{\sf L}^*\!:L^2({\sf G})\to L^2({\sf L})$ is then the restriction. It follows then directly from definitions that $\mathfrak{Op}^{\sf
      L}_{A}(\si)=J_{\sf L}^*\,\mathfrak{Op}_{A}(\si) J_{\sf L}$\,, hence $\mathfrak{Op}^{\sf L}_{A}(\si)$ is indeed the compression
    of the pseudo-differential operator $\mathfrak{Op}_{A}(\si)$ to $L^2({\sf L})$\,.  }
\end{Remark}

\begin{Theorem}\label{thesame}
The operators $\mathfrak{Op}^{\sf L}_{A}(\si)$ and $\mathfrak{Op}_{A}(\si)$ have the same essential spectra and the same essential numerical range. More explicitly, one has
\begin{equation}\label{fierar}
  {\sf sp}_{\rm ess}\big[\mathfrak{Op}^{\L}_{A}(\si)\big] = {\sf sp}_{\rm ess}\big[\mathfrak{Op}_{A}(\si)\big] = \bigcup_{x\in X}{\sf sp}\big[\mathfrak{Op}_{A_x}\!(\si_x)\big]\,,
\end{equation}
\begin{equation}\label{fifer}
  {\sf nr}_{\rm ess} \big[\mathfrak{Op}^{\L}_{A}(\si)\big] = {\sf nr}_{\rm ess} \big[\mathfrak{Op}_{A}(\si)\big] = \overline{\rm co}
  \Big(\bigcup_{x\in X}{\sf nr}\big[\mathfrak{Op}_{A_x}\!(\si_x)\big]\Big)\,,
\end{equation}
where $\si_x$ and $B_x=dA_x$ have the same meaning as in Theorem \ref{specmagnet}. In addition, $\mathfrak{Op}^{\sf L}_{A}(\si)+{\sf
  s}1$ and $\mathfrak{Op}_{A}(\si)+{\sf s}1$ are simultaneously Fredholm.
\end{Theorem}

\begin{Remark}\label{difersse}
\normalfont{Note that the two operators $\mathfrak{Op}^{\sf L}_{A}(\si)$ and $\mathfrak{Op}_{A}(\si)$ act in two different
  Hilbert spaces. It is not immediately clear how to prove the first equalities in \eqref{fierar} and \eqref{fifer} by some relative
  compactness criterion. Our proof will be to relate the operators $\mathfrak{Op}^{\sf L}_{A}(\si)$ to a groupoid that is a
  non-invariant reduction of the transformation groupoid of the preceding subsection, a result that could be useful elsewhere.}
\end{Remark}

The starting observation is that $\th$ restricts to a {\it partial action} $\th^L\!$ of $\G$ on $L={\sf L}\sqcup X_\infty$
\cite{Ex1,Ex2}. The abstract definition, using the present notations, is that for each $a\in \G$ one has a homeomorphism
\begin{equation*}\label{tuberoza}
  \th^L_{a}:L_{a^{-1}}\!\to L_{a}
\end{equation*}
between two open subsets of $L$ such that $\th^L_{\rm e}={\rm id}_{L}$ and such that $\th^L_{a}\!\circ\th^L_{b}$, defined on the maximal domain
\begin{equation*}\label{trandaflor}
  (\th_{b}^L)^{-1}\big(L_{b}\cap L_{a^{-1}}\big)=\Big\{z\in  L_{b^{-1}}\Big\vert\,\th^L_{b}(z)\in L_{a^{-1}}\Big\}\,,
\end{equation*}
is a restriction of $\th_{ab}^L$\,. It is easy to check that $\th_{a^{-1}}^L\!=\big(\th_{a}^L\big)^{-1}$, where $\beta^{-1}:V\to U$ denotes here the inverse of the partial homeomorphism $\beta:U\to V$.

\smallskip
Specifically, in our case, one sets $\,\th_{a}^L$ to be the restriction of $\th_{a}$ to
\begin{equation*}\label{crinulet}
  L_{a^{-1}} := L\cap\th_{a}^{-1}(L) = \big\{z\in L\big\vert\,\th_{a}(z)\in L\big\} = \big({\sf L}\cap a^{-1}{\sf
    L}\big)\sqcup X_\infty\,,
\end{equation*}
and the axioms are easy to check. Unless $\L=\G$\,, the subset $L$ is not $\th$-invariant, so $\th^L$ {\it will not be a global group action}.

\smallskip
In \cite{Ab}, {\it to any continuous partial action corresponds a locally compact groupoid}. We describe the construction for our
concrete partial action $\th^L$, making certain notational simplifications permitted by the context. One sets
\begin{equation}\label{albastrizta}
  \Xi(L)\equiv L\!\times_{\!(\th)}\!\G := \big\{(z,a)\,\big\vert\,a\in\G\,,\, z\in L\,, \th_{a}(z)\in L\big\}
\end{equation}
(it can also be seen as the disjoint union over $a\in\G$ of the domains $L_{a^{-1}}$). The topology is the restriction of the product topology of $L\times\G$\,. The inversion is defined by
\begin{equation*}\label{panseluta}
(z,a)^{-1}\!:=\big(\th_{a}(z),a^{-1}\big)
\end{equation*}
and the multiplication (only) by
\begin{equation*}\label{violetta}
  \big(\th_{a}(z),b\big)\big(z,a\big):=\big(z,ba\big)\,.
\end{equation*}
Then the source and the range map are given by
\begin{equation*}\label{tufaneaka}
  {\rm d}_L(z,a)=(z,{\sf e})\equiv z\,,\quad {\rm r}_L(z,a)=\big(\th_{a}(z),{\sf e}\big)\equiv\th_{a}(z)\,,
\end{equation*}
so we identify the unit space of the groupoid with $L$\,. To $\L=\G$ corresponds $\Xi\equiv\Xi(X)=X\!\rtimes_\th\!\G$\,, the transformation groupoid defined by the initial global action $\th$\,.

\begin{Proposition}\label{summarize}
Let $(X,\th,\G)$ be a dynamical system, where $X=\G\sqcup X_\infty$ is a Hausdorff second-countable compactification of the connected simply
connected nilpotent group $\,\G$ and the restriction of the action $\th$ of $\,\G$ on itself consists of left translations. Let
$\,\L\subset\G$ be a subset satisfying Assumption \ref{assumption} and $\,L\!:={\sf L}\sqcup X_\infty$\,. The partial transformation groupoid
$\,\Xi(L)$ is a standard groupoid with unit space $L$ and main orbit $\L$\,. It is the (non-invariant) reduction of the transformation groupoid $\,\Xi\equiv\Xi(X)$ to $L$\,.
\end{Proposition}

\begin{proof}
Clearly the space $\Xi(L)$ is locally compact, Hausdorff and second countable.

\smallskip
From \eqref{albastrizta} and the form of the maps ${\rm d,r}:\Xi\to X$ it follows that
\begin{equation*}\label{garoafa}
  \Xi(L) = \big\{\xi\in\Xi\,\big\vert\, {\rm d}(\xi)\,,{\rm r}(\xi)\in L\big\} = {\rm d}^{-1}\big(L\big)\cap {\rm r}^{-1}\big(L\big)\,,
\end{equation*} 
so $\Xi(L)$ is the reduction of $\Xi\equiv\Xi(X)$ to the closed (non-invariant) subset $L$\,. It is an closed subgroupoid of $\Xi$\,,
thus a locally compact groupoid in itself. Its amenability follows from \cite[Ch.II,\,Prop.3.7,\,3.9]{Re}.

\smallskip
The problem of (fibrewise) restricting Haar systems to non-invariant closed subsets is not trivial. In \cite[Sect.1]{Nic2} it has been
solved for the case of transformation groupoids. We write down the output using the present notations.

For any $z\in L$\,, denote by $\lambda_z$ the restriction to $\Xi(L)$ of the product $\delta_z\!\otimes\!\m$\,, as well as {\it the set of $L$-admissible translations}
\begin{equation*}\label{mac}
\G^L(z):=\big\{a\in\G\,\big\vert\,\th_{a}(z)\in L\big\}\,.
\end{equation*}
A necessary and sufficient condition for $\{\lambda_z\!\mid\!z\in L\}$ to be a Haar system for the restricted groupoid is precisely the pair of conditions:
\begin{enumerate}
\item[($\alpha$)] the restriction of the Haar(=Lebesgue) $\m$ to any subset $\G_L(z)$ has full support,
\item[($\beta$)] the mapping $L\ni z\to\!\int_{\G^L\!(z)}\psi(a)d\m(a)\in\mathbb C$ is continuous for each $\psi\in C_{\rm c}(\G)$\,.
\end{enumerate}
But in our case one has 
\begin{equation}\label{according}
\G^L(z) = \G\ \,{\rm if}\ z\in X_\infty\quad \mbox{and} \quad \G^L(b)={\sf L}b^{-1}\ {\rm if}\ z=b\in\L\,.
\end{equation}
Condition ($\alpha$) surely holds for $z\in X_\infty$\,, and it also holds for $z=b\in\L$\,, since $\m$ is invariant under (right)
translations and $L$ was defined to be the closure of the open set $\G\!\setminus\!{\sf K}$ (the Haar measure of any non-void open set is
strictly positive). Condition ($\beta$) amounts to show for every $\psi\in C_{\rm c}(\G)$ the continuity of the function
\begin{equation*}
  \Psi:L\to\mathbb C\,,\quad\Psi(z) :=\int_{\G}\psi(a)da\ \ \mbox{if}\ \ z\in X_\infty\,,\quad\Psi(b)\!:=\!\int_{{\sf L}b^{-1}}\!\psi(a)da\ \ \mbox{if}\ \ z=b\in\L\,.
\end{equation*}
The continuity at points belonging to $\L$ is standard and easy (translations are continuous in $L^1(\G)$)\,. Let us set ${\sf
  D}\!:=\supp(\psi)$\,; it is a compact subset of $\G$\,. To show continuity at a point $z_\infty\in X_\infty$\,, it is enough to find,
for each $\epsilon>0$\,, a neighborhood $V_\epsilon$ of $z_\infty$ such that
\begin{equation}\label{danssa}
  |\Psi(b)-\Psi(z_\infty)|=\Big\vert\int_{{\L}b^{-1}\cap{\sf D}}\psi(a)da-\!\int_{\sf
    D}\psi(a)da\Big\vert\le\epsilon\,,\quad\forall\,b\in V_\epsilon\cap\G\,.
\end{equation}
Assume that $b^{-1}$ does not belong to the compact subset ${\sf K}^{-1}{\sf D}$\,; then ${\sf D}\cap{\sf K}b^{-1}=\emptyset$\,. It follows that
\begin{equation*}
   {\sf D}\subset\big({\sf K}b^{-1}\big)^c={\sf K}^cb^{-1}\subset\overline{{\sf K}^c}b^{-1}={\sf L}b^{-1},
\end{equation*}
so ${\sf L}b^{-1}\!\cap{\sf D}={\sf D}$ and the difference of integrals in \eqref{danssa} is in fact zero. But it is clear that,
given such a compact set ${\sf K}^{-1}{\sf D}$\,, there is a neighborhood $V_\epsilon$ of $z_\infty$ that does not meet it. Thus the restriction problem is solved.

\smallskip
We deal now with the remaining issues. The ${\rm d}_L$-fiber of $z\in L$ is
\begin{equation*}\label{crizantema}
  [\Xi(L)]_z=\big\{(z,a)\,\big\vert\,\th_{a}(z)\in L\big\}=\{z\}\times\G^L(z)\,.
\end{equation*}
The isotropy group is
\begin{equation*}\label{krizantema}
  [\Xi(L)](z)=\big\{(z,a)\,\big\vert\,\th_{a}(z)=z\big\}
\end{equation*} and can be identified with the isotropy group of $z$ with respect to the initial global action. In particular, if $z\equiv p\in {\sf L}$\,,
it is trivial. Two points $z_1$ and $z_2$ are orbit-equivalent if, and only if, both being elements of $L$\,, are on the same $\th$-orbit. In
particular ${\sf L}$ is itself an (open and dense) orbit under the action of the groupoid. The other equivalence classes are the orbits contained in $X_\infty$\,. The groupoid $\Xi(L)$ is standard.
\end{proof}

{\it Proof of Theorem \ref{thesame}.} The proof is based on Proposition \ref{summarize}. Note the decompositions
\begin{equation}\label{writedown}
 \Xi(L) = \Xi({\sf L}) \sqcup \Xi(X_\infty)\,, \quad \Xi(X) = \Xi({\sf
   G}) \sqcup \Xi(X_\infty)\,.
\end{equation}
The second one has been used in the previous subsection to prove Theorem \ref{specmagnet}, in conjunction with the $2$-cocycle \eqref{asta}. The first one and the restricted $2$-cocycle
\begin{equation*}\label{astak}
  \o_{B}^{\sf L} \big((\th^{\sf L}_q(x),p),(x,q)\big) := e^{i\Gamma_{\!B_x}\!\<\!\<\e,q,pq\>\!\>},\quad q,p\in{\sf L}\,, x\in X_\infty
\end{equation*}
could be used similarly to prove the first terms in \eqref{fierar} and \eqref{fifer} are equal to the last ones, respectively. The partial
Fourier transform serves in the same way to relate the regular representations of the twisted groupoid $C^*$-algebra associated to
$\big(\Xi(L),\o_{B}^{\sf L}\big)$ to the compressed magnetic pseudo-differential operators. Recall that the ${\rm d}_{L}$-fibers
$[\Xi(L)]_z=\{z\}\times\G^L(z)$ of $\Xi(L)$ are now of two distinct types, according to \eqref{according}. This explains why, while the
operator $\mathfrak{Op}^{\sf L}_{A}(\si)$ acts in $L^2({\sf L})$\,, "the asymptotic operators" $\mathfrak{Op}_{A_x}\!(\si_x)$ act in $L^2(\G)$\,.

\smallskip
A more direct proof is to note that the invariant restrictions to $X_\infty$ of the two groupoids in \eqref{writedown} coincide, as well as the two cocycles. Then the first equalities in \eqref{fierar} and
\eqref{fifer} follow, respectively, from the first equalities of \eqref{jardiances} and \eqref{jardanche}\,; both involve restrictions of the symbols to the common part $\Xi(X_\infty)$\,, and these
restrictions are the same.

\smallskip
For the Fredholm properties one uses Atkinson's Theorem.

\subsection{Twisted Wiener-Hopf operators on the Heisenberg group}\label{grafoni}

In this subsection $\G$ will be the $3$-dimensional Heisenberg group; it is a connected simply connected $2$-step nilpotent group. As a set $\G:=\R^3$, the multiplication being
\begin{equation*}\label{viitura}
  (a_1,a_2,a_3)(b_1,b_2,b_3) := (a_1+b_1,a_2+b_2,a_3+b_3+a_1b_2)
\end{equation*} 
and the inversion $(a_1,a_2,a_3)^{-1} := (-a_1,-a_2,-a_3+a_1a_2)$\,. The Lebesgue measure $d\m(a)\equiv da$ is a Haar measure\,. Its reduction to the closed submonoid
\begin{equation*}\label{positiv}
  {\sf H} := \R_+^3 = \{h=(h_1,h_2,h_3)\in\G\! \mid\!h_1,h_2,h_3 \ge 0\}
\end{equation*} 
has full support; $L^2({\sf H})$ will be identified to a closed  subspace of the Hilbert space $L^2(\G)$\,. The adjoint $J_{\sf H}^*$ of the canonical inclusion $J_{\sf H}:L^2({\sf H})\to L^2({\sf G})$ is
the restriction map from $L^2({\sf G})$ to $L^2({\sf H})$\,. Here and below we use the fact that ${\sf H}$ is {\it solid}, i.e.\,it coincides with the closure of its interior.

\smallskip
As in the preceding subsection, we fix a magnetic field $B$ and a corresponding vector potential $A$ with $B=dA$\,. For every
$\varphi\in L^1({\sf G})$ one has {\it the twisted (magnetic) convolution operator} in $L^2({\sf G})$ given by
\begin{equation*}\label{convol}
  \big[C_A(\varphi)u\big](a) := \int_\G e^{i\int_{[\![b,a]\!]}A}\varphi(ab^{-1})u(b)db\,.
\end{equation*}


\begin{Definition}\label{musetel}
The compression $W_{\!A}(\varphi):=J_{\sf H}^*\,C_A(\varphi)J_{\sf H}$ to $L^2({\sf H})$ of $C_A(\varphi)$ is called {\rm the magnetic
  Wiener-Hopf operator of symbol $\varphi\in L^1(\G)$} and it has the explicit form
\begin{equation*}\label{WH}
  \big[W_{\!A}(\varphi)w\big](h):=\int_{\sf H}e^{i\int_{[\![k,h]\!]}A}\varphi(hk^{-1})w(k)dk\,,\quad\forall\,w\in L^2({\sf H})\,,\,h\in{\sf H}\,.
\end{equation*}
The $C^*$-subalgebra of $\,\mathbb B\big[L^2({\sf H})\big]$ generated by all these operators is {\rm the magnetic Wiener-Hopf $C^*$-algebra}, denoted by ${\sf W}_{\!A}(\G,{\sf H})$\,.
\end{Definition}

\begin{Remark}\label{gagecov}
\normalfont{If $B=dA=dA'$ is given by two vector potentials, the connection is $A'=A+d\nu$\,, where $\nu$ is a smooth real function
  on $\G$\,. It follows immediately that $W_{\!A'}(\varphi)=e^{-i\nu(\cdot)}W_{\!A}(\varphi)e^{i\nu(\cdot)}.$
  The two Wiener-Hopf operators associated to the same symbol but to two equivalent vector potentials are unitarily equivalent (this is
  the gauge covariance in this setting). So their spectral properties only depend on the magnetic field. It also follows that the two
  $C^*$-algebras ${\sf W}_{\!A}(\G,{\sf H})$ and ${\sf W}_{\!A'}(\G,{\sf H})$ are isomorphic, the isomorphism being unitarily implemented.  }
\end{Remark}

Following \cite{Nic2}, we indicate a compactification of $\sf H$ suited to study these Wiener-Hopf operators. One starts with
\begin{equation*}
  \mathbf X := \{\rho\in L^\infty(\G)\! \mid\, \|\rho\|_\infty\, \le1\}\,.
\end{equation*} 
By the Alaoglu theorem, it is a compact space with respect of the $w^*$-topology ($L^\infty(\G)$ is the dual of the Banach space
$L^1(\G)$)\,. In terms of characteristic functions, one introduces the map
\begin{equation*}
   \chi:{\sf H}\to\mathbf X, \quad\chi(h) := \chi_{h{\sf H}^{-1}}.
\end{equation*}
It is injective, since ${\sf H}$ is solid. The closure of its range $X\!:=\overline{\chi({\sf H})}\subset\mathbf X$ will be the unit space
of a future groupoid.  It is shown in \cite[Prop.\,2.1]{Nic2} that all the elements of $X$ are characteristic functions $\chi_{\sf E}$\,,
where $\sf E$ belongs to some family $\mathscr X$ of solid subsets of $\G$\,. Through the bijection $\mathscr X\ni{\sf E}\to\chi_{\sf E}\in
X$ one transfers to $\mathscr X$ the $w^*$-topology. The (homeomorphic) spaces $X$ and $\mathscr X$ may be seen as
compactifications of $\sf H$\,. Actually, cf. \cite[Sect.\,3]{Nic2}, the compactification is regular, meaning that $\chi(\sf H)$ is open in
its closure $X$. It is also true that, if $a\in\G$ and $\sf E\in\mathscr X$, {\it one has $a\sf E\in\mathscr X$ if and only if  $a\in\sf E^{-1}$}.

\begin{Remark}\label{comeback}
\normalfont{Let us assume that the restriction of the magnetic field $B$ to $\sf H$ admits a continuous extension $\tilde B$ to the
  compactification $\mathscr X$. Thus $\tilde B(h{\sf H}^{-1})=B(h)$ if $h\in\sf H$\,. For any ${\sf E}\in\mathscr X$ we set
\begin{equation*}\label{lainfinit}
  B_{{\sf E}}:{\sf E}^{-1}\to\wedge^2(\G)\,,\quad B_{\sf E}(a) :=\tilde B(a{\sf E})\,,\quad\forall\,a\in\sf E^{-1}.
\end{equation*}
In particular, if $h\in\sf H$ and $a\in(h{\sf H}^{-1})^{-1}={\sf H}h^{-1}$ (meaning that $ah\in\sf H$)\,, then
\begin{equation*}
  B_{h{\sf H}^{-1}}(a)=\tilde B(ah{\sf H}^{-1})=B(ah)\,.
\end{equation*}
So, for every $h\in\sf H$\,, $B_{h{\sf H}^{-1}}$ is the restriction to ${\sf H}h^{-1}\!\subset\G$ of a right translation of $B$\,. Taking
$h=\e$ we see that $B_{\sf H^{-1}}$ is the restriction to $\sf H$ of the magnetic field $B$\,.}
\end{Remark}

\smallskip
We are going to need two particular subsets of $\G\equiv\R^3$\,:
\begin{equation}\label{ahaceia}
  {\sf U} = \R_-\!\times\R_-\!\times\R\quad{\rm and}\quad{\sf V}=\{b\in\R^3\mid b_2\le 0,b_3-b_1b_2\le 0\}\,,
\end{equation}
with group inverses 
\begin{equation*}
  {\sf U}^{-1} = \R_+\!\times\R_+\!\times\R\quad{\rm and}\quad {\sf V}^{-1}=\R\times\R_+\!\times\R_+\,.
\end{equation*}
One also define the operators in $L^2({\sf U})$ and $L^2(\sf V)$\,, respectively, by
\begin{equation}\label{doardansul}
  \big[\pi_{\sf U}(\varphi)w\big](a) := \int_{{\sf U}^{-1}}\!\!e^{i\Gamma_{\!B_{\sf U}}\!\<\!\<\e,b,a\>\!\>}\varphi(ab^{-1})\,w(b)db\,,
\end{equation}
\begin{equation}\label{doardansa}
\big[\pi_{\sf V}(\varphi)w\big](a) := \int_{{\sf V}^{-1}}\!\!e^{i\Gamma_{\!B_{\sf V}}\!\<\!\<\e,b,a\>\!\>}\varphi(ab^{-1})\,w(b)db\,.
\end{equation}

\begin{Theorem}\label{plictisit}
Given $\varphi\in C_{\rm c}(\G)$ and $B=dA$ a magnetic field on $\G$ whose restriction to $\sf H$ can be extended continuously to $\mathscr X$, one has
\begin{equation*}\label{spesentialul}
  {\sf sp}_{\rm ess}\big[W_A(\varphi)\big] = {\sf sp}\big[\pi_{\sf U}(\varphi)\big]\cup{\sf sp}\big[\pi_{\sf V}(\varphi)\big]
\end{equation*}
and
\begin{equation*}\label{nresentialul}
  {\sf nr}_{\rm ess} \big[W_A(\varphi)\big] = \overline{\sf co}\Big({\sf nr} \big[\pi_{\sf U}(\varphi)\big] \cup{\sf nr}\big[\pi_{\sf V}(\varphi)\big]\Big)\,.
\end{equation*}
The magnetic Wiener-Hopf operator $W_A(\varphi)$ is Fredholm if and only if $\pi_{\sf U}(\varphi)$ and $\pi_{\sf V}(\varphi)$ are invertible.
\end{Theorem}

To prove this result, we use the groupoid
\begin{equation*}\label{vicnis}
  \Xi := \big\{({\sf E},a)\in\mathscr X\!\times\!\G\!\mid\!a\in\sf E^{-1}\big\}
\end{equation*}
with the restriction of the product topology and the structure maps
\begin{equation*}\label{structar}
  (a{\sf E},b)({\sf E},a) := ({\sf E},ba)\,,\quad({\sf E},a)^{-1}:=\big(a{\sf E},a^{-1}\big)\,,
\end{equation*}
\begin{equation*}\label{structur}
  {\rm d}({\sf E},a) := {\sf E}\,,\quad{\rm r}({\sf E},a):=a{\sf E}\,.
\end{equation*}
Note that 
\begin{equation}\label{fibrika}
  \Xi_{\sf E} := {\rm d}^{-1}(\{{\sf E}\})=\big\{({\sf E},a)\!\mid\!a\in\sf E^{-1}\big\}\cong{\sf E}^{-1}.
\end{equation}
The properties of this groupoid have been investigated in \cite{Nic2} (see also \cite{MuR,Nic1}). Restricting the Haar measure $\m$ to the
solid subset $\sf E^{-1}\!\subset\G$ and then transferring it to the fiber $\Xi_{\sf E}$\,, one gets a right Haar system. In our terminology, {\it the groupoid $\Xi$ is standard, with unit space
  $X\cong\mathscr X$ and main orbit} ${\sf H}\equiv\chi({\sf H})=\big\{\chi_{h{\sf H}^{-1}}\!\mid\!h\in{\sf H}\big\}$\,.
Actually, the map ${\sf H}\times{\sf H}\ni(h,k)\to\big(k{\sf H}^{-1}\!,hk^{-1}\big)\in\Xi|_{\sf H}$ is an isomorphism between the
pair groupoid and the (invariant) restriction of $\Xi$ to its main orbit.


\smallskip
The construction of the groupoid is quite general. Following \cite[Sect.4]{Nic2}, we indicate now the unit space and its
quasi-orbit structure for the case of the $3$-dimensional Heisenberg group. First of all, if $h:=(h_1,h_2,h_3)\in\R^3\equiv\G$\,, then
\begin{equation*}
  h{\sf H}^{-1}=\big\{a=(a_1,a_2,a_3)\,\big\vert\, a_1\le h_1,a_2\le h_2,a_3+(h_2-a_2)a_1\le h_3\big\}\,.
\end{equation*}
We recall that $\{h{\sf H}^{-1}\!\mid h\in{\sf H}\equiv\R_+^3\}$ has been identified to $\sf H$\,. To describe the other sets $\sf
E\in\mathscr X$ that are units of the groupoid, for any $h=(h_1,h_2,h_3)\in\mathbb R^3_+$ we use the following notations:
\begin{equation*}
   {\sf S}_{h_1,h_2,\cdot} := \{a\in\R^3\mid a_1\le h_1,a_2\le h_2\}\,,
\end{equation*}
\begin{equation*}
   {\sf S}_{\cdot,h_2,h_3} := \{a\in\R^3\mid a_2\le h_2,a_3+(h_2-a_2)a_1\le h_3\}\,,
\end{equation*}
\begin{equation*}
  {\sf S}_{h_1,\cdot,\cdot} := \{a\in\R^3\mid a_1\le h_1\}\,,\quad{\sf S}_{\cdot,h_2,\cdot} := \{a\in\R^3\mid a_2\le h_2\}\,.
\end{equation*}
Then 
\begin{equation*}
   \mathscr X = \mathscr X_{1,2,3}\sqcup\mathscr X_{1,2}\sqcup\mathscr X_{2,3}\sqcup\mathscr X_{1}\sqcup\mathscr X_{2}\sqcup\mathscr X_{0}
\end{equation*}
is the disjoint union of six orbits, given explicitly by
\begin{equation*}
   \mathscr X_{1,2,3} := \{h{\sf H}^{-1}\!\mid h\in\R_+^3\}\,,\quad\mathscr X_{0}:=\{\R^3\}\,,
\end{equation*}
\begin{equation*}
  \mathscr X_{1,2} := \{{\sf S}_{h_1,h_2,\cdot}\!\mid h_1\ge 0,h_2\ge 0\}\,,\quad \mathscr X_{2,3} := \{{\sf S}_{\cdot,h_2,h_3}\!\mid h_2\ge 0,h_3\ge 0\}\,,
\end{equation*}
\begin{equation*}
 \mathscr X_{1} := \{{\sf S}_{h_1,\cdot,\cdot}\!\mid h_1\ge 0\}\,,\quad \mathscr X_{2}:=\{{\sf S}_{\cdot,h_2,\cdot}\!\mid h_2\ge 0\}\,.
\end{equation*}
The quasi-orbits are given by closures in the weak$^*$-topology (transported from $X\subset L^\infty(\G)$ to $\mathscr X$)\,:
\begin{equation*}
\overline{\mathscr X_{1,2,3}} = \mathscr X,\quad\overline{\mathscr X_{0}}=\mathscr X_{0}\,,
\end{equation*}
\begin{equation*}
  \overline{\mathscr X_{1,2}} = \mathscr X_{1,2}\sqcup\mathscr X_{1}\sqcup\mathscr X_{2}\sqcup\mathscr X_{0}\,,\quad\overline{\mathscr X_{2,3}}=\mathscr
  X_{2,3}\sqcup\mathscr X_{1}\sqcup\mathscr X_{2}\sqcup\mathscr X_{0}\,,
\end{equation*}
\begin{equation*}
  \overline{\mathscr X_{1}} = \mathscr X_{1}\sqcup\mathscr X_{0}\,,\quad\overline{\mathscr X_{2}}=\mathscr X_{2}\sqcup\mathscr X_{0}\,.
\end{equation*}
Note that 
\begin{equation}\label{covering}
 \mathscr X\!\setminus\!{\sf H}=\overline{\mathscr X_{1,2}}\cup\overline{\mathscr X_{2,3}}\,.
\end{equation}


\begin{Lemma}\label{cirnat}
The formula
\begin{equation*}\label{salam}
  \o_{B}^{\sf H}\big((a{\sf E},b),({\sf E},a)\big) :=e^{i\Gamma_{\!B_{\sf E}}\<\!\<\e,a,ba\>\!\>},\quad {\sf E}\in\mathscr X,\,a,ba\in{\sf E}^{-1}
\end{equation*}
defines a $2$-cocycle of the groupoid $\Xi$\,.
\end{Lemma}

\begin{proof}
To check the $2$-cocycle identity, one considers three elements of the groupoid
\begin{equation*}
  \xi := (ba{\sf E},c)\,,\quad\eta := (a{\sf E},b)\,,\quad\zeta:=({\sf E},a)\,,\quad{\rm with}\quad a,ba,cba\in\sf E^{-1}.
\end{equation*}
Then
\begin{equation*}
\begin{aligned}
  \o_{B}^{\sf H}(\xi,\eta\big)\o_{B}^{\sf H}(\xi\eta,\zeta) &=\o_{B}^{\sf H}\big((ba{\sf E},c),(a{\sf E},b)\big)\o_{B}^{\sf
    H}\big((a{\sf E},cb),({\sf E},a)\big)\\ &=e^{i\Gamma_{\!B_{a\sf E}}\<\!\<\e,b,cb\>\!\>}e^{i\Gamma_{\!B_{\sf E}}\<\!\<\e,a,cba\>\!\>},
\end{aligned}
\end{equation*}
while
\begin{equation*}
\begin{aligned}
  \o_{B}^{\sf H}(\xi,\eta\zeta)\o_{B}^{\sf H}(\eta,\zeta\big) &=\o_{B}^{\sf H}\big((ba{\sf E},c),({\sf E},ba)\big)\o_{B}^{\sf
    H}\big((a{\sf E},b),({\sf E},a)\big)\\
  &=e^{i\Gamma_{\!B_{\sf E}}\<\!\<\e,ba,cba\>\!\>}e^{i\Gamma_{\!B_{\sf E}}\<\!\<\e,a,ba\>\!\>}.
\end{aligned}
\end{equation*}
The relation
\begin{equation*}
  e^{i\Gamma_{\!B_{a\sf E}}\<\!\<\e,b,cb\>\!\>} = e^{i\Gamma_{\!B_{\sf E}}\<\!\<a,ba,cba\>\!\>}
\end{equation*}
and Stokes Theorem for the closed $2$-form $B_{\sf E}$ show that the two expressions are equal.

\smallskip
Normalization is easy: the units of the groupoid are of the form $\sf E\in\mathscr X$, $\chi_{\sf E}\in X$ or $({\sf E},\e)\in\Xi^{(0)}$, depending on the interpretation. One has
\begin{equation*}
  \o_{B}^{\sf H}\big(({\sf E},b),({\sf E},\e)\big) = e^{i\Gamma_{\!B_{\sf E}}\<\!\<\e,\e,b\>\!\>} = 1\,,\quad{\rm if}\quad b\in\sf E^{-1},
\end{equation*}
\begin{equation*}
\o_{B}^{\sf H}\big((a{\sf E},\e),({\sf E},a)\big) = e^{i\Gamma_{\!B_{\sf E}}\<\!\<\e,a,a\>\!\>}=1\,,\quad{\rm if}\quad a\in\sf E^{-1}.
\end{equation*}

The continuity of the $2$-cocycle follows straightforwardly from the continuity of $\tilde B$ on $\mathscr X$.
\end{proof}

Now we compute an adapted version of the regular representations. Let ${\sf E}\in\mathscr X$. By \eqref{fibrika} one has the unitary map
\begin{equation*}\label{fibruka}
  \iota_{\sf E}:\H_{\sf E} := L^2(\Xi_{\sf E})\to L^2({\sf E}^{-1})\,,\quad\big[\iota_{\sf E}(v)\big](a) := v({\sf
    E},a)\,,\quad\forall\,a\in{\sf E}^{-1},
\end{equation*}
inducing a unitary equivalence
\begin{equation*}\label{unitec}
  \mathcal I_{\sf E} : \mathbb B(\H_{\sf E})\to\mathbb B\big[L^2({\sf  E}^{-1})\big]\,,\quad \mathcal I_{\sf E}(S) := \iota_{\sf  E}S\iota_{\sf E}^{-1}.
\end{equation*}
Denoting by $\Pi_{\sf E}$ the regular representation attached to the unit $\sf E$\,, we are interested in
\begin{equation*}\label{interes}
  \tilde\Pi_{\sf E} := \mathcal I_{\sf E}\circ\Pi_{\sf E} : {\sf C}^*\big(\Xi,\o_B^{\sf H}\big) \to\mathbb B\big[L^2({\sf E}^{-1})\big]\,.
\end{equation*}
One computes for $f\in C_{\rm c}(\Xi)$\,, $w\in L^2({\sf E}^{-1})$ and $a\in{\sf E}^{-1}$
\begin{equation*}
\begin{aligned}
 \big[\tilde\Pi_{\sf E}(f)w\big](a)& =\big[\iota_{\sf E}\Pi_{\sf  E}(f)\iota_{\sf E}^{-1}w\big](a) = \big[\Pi_{\sf E}(f)\iota_{\sf E}^{-1}w\big]({\sf E},a)\\
 &= \Big[f\star_{\o_B^{\sf H}}\!(\iota_{\sf E}^{-1}w)\Big]({\sf E},a)\\
  & = \int_{\Xi_{\sf E}}\!f(b{\sf E},ab^{-1})\,(\iota_{\sf E}^{-1}w)({\sf E},b)\,\o_B^{\sf H}\big((b{\sf E},ab^{-1}),({\sf E},b)\big)d\lambda_{\Xi_{\sf E}}({\sf E},b)\\
 & = \int_{\sf E^{-1}}\!e^{i\Gamma_{\!B_{\sf E}}\<\!\<\e,b,a\>\!\>}f(b{\sf E},ab^{-1})\,w(b)db\,.
\end{aligned}
\end{equation*}
It is possible, but not necessary, to write all these operators in terms of vector potentials, up to unitary equivalence.

\begin{Remark}\label{faceceva}
{\rm There is an obvious injection 
\begin{equation*}\label{lainjection}
C_{\rm c}(\G)\ni\varphi\to f_\varphi\in C_{\rm c}(\Xi)\,,\quad f_\varphi({\sf E},a):=\varphi(a)\,.
\end{equation*}
We are only going to consider the operators acting in $L^2({\sf E}^{-1})$ and given by
\begin{equation}\label{doardansii}
\big[\pi_{\sf E}(\varphi)w\big](a) := \big[\tilde\Pi_{\sf E}(f_\varphi)w\big](a)=\int_{\sf E^{-1}}\!e^{i\Gamma_{\!B_{\sf E}}\<\!\<\e,b,a\>\!\>}\varphi(ab^{-1})\,w(b)db\,.
\end{equation}}
\end{Remark}

We are mainly interested in three particular cases. First of all, notice that the subsets introduced in \eqref{ahaceia} may be
identified as ${\sf U}={\sf S}_{0,0,\cdot}$ and ${\sf V}={\sf S}_{\cdot,0,0}$\,. For these two cases, the operators $\pi_{\sf
  E}(\varphi)$ are precisely those given in \eqref{doardansul} and \eqref{doardansa}.  Let us also set $\sf E\!:=\sf H^{-1}$ in
\eqref{doardansii}. By Remark \ref{comeback}, for $\varphi\in C_{\rm  c}(\G)$\,, $w\in L^2(\sf H)$ and $h\in\sf H$ we have
\begin{equation}\label{lamurit}
\big[\pi_{\sf H^{-1}}(\varphi)w\big](h) = \int_{\sf H}\!e^{i\Gamma_{\!B}\<\!\<\e,k,h\>\!\>}\varphi(hk^{-1})\,w(k)dk\,.
\end{equation}
Since $B=dA$ one has, by Stokes' Theorem,
\begin{equation*}
\Gamma_{\!B}\<\!\<\e,k,h\>\!\>=\Gamma_{\!A}[\![\e,k]\!]+\Gamma_{\!A}[\![k,h]\!]+\Gamma_{\!A}[\![h,\e]\!]
\end{equation*}
and \eqref{lamurit} becomes
\begin{equation*}\label{lamurita}
e^{i\Gamma_{\!A}[\![\e,h]\!]}\big[\pi_{\sf  H^{-1}}(\varphi)w\big](h)=\int_{\sf H}\!e^{i\Gamma_{\!A}[\![k,h]\!]}\varphi(hk^{-1})\,e^{i\Gamma_{\!A}[\![\e,k]\!]}w(k)dk\,.
\end{equation*}
Recalling the definition \ref{musetel} of the Wiener-Hopf operators and since the operator $w\to e^{i\Gamma_{\!A}[\![\e,\cdot]\!]}w$ is
unitary in $L^2(\sf H)$\,, {\it we get the unitary equivalence of the operators $\pi_{\sf H^{-1}}(\varphi)$ and $W_{\!A}(\varphi)$}\,.

\smallskip
By all these preparations, by the covering \eqref{covering} of the boundary of the compactification by two quasi-orbits and the fact that
${\sf U}\in\mathscr X_{1,2}\subset\overline{\mathscr X_{1,2}}$ and ${\sf V}\in\mathscr X_{2,3}\subset\overline{\mathscr X_{2,3}}$\,, our
Theorem \ref{plictisit} follows from the results of subsections \ref{ameny} and \ref{caloriforit}.

\smallskip
The fact that the quasi-orbit structure of the unit space of the groupoid is known and simple, allowed us to write the results in terms
just of a pair of "asymptotic operators", instead of using the entire family $\{\pi_{\sf E}(\varphi)\mid {\sf E}\in\mathscr X\!\setminus\!{\sf H}\}$\,.

\section{Localization and non-propagation properties}\label{sec.five}

\subsection{The abstract result}\label{homeny}

Assuming that $\Xi$ is a standard groupoid with main orbit $M$ and $2$-cocycle $\o$\,, we pick a normal element $F\in\,{\sf
  C}^*(\Xi,\o)^{\bf M}$. Its image $H_0:=\Pi_0^{\bf M}(F)$ in the vector representation is an operator in $\H_0:= L^2(M,\mu)$\,. By
$\textbf{1}_V$ we denote the charcteristic function of the set $V\subset M$, as well as the corresponding multiplication operator in $\H_0$\,.

\begin{Theorem}\label{main}
Let $F\in\,{\sf C}^*(\Xi,\o)^{\bf M}$ be normal and $\mathcal Q\subset X_\infty:=X\!\setminus\!M$ a quasi-orbit. Let $\kappa\in C_0(\mathbb
R)$ be a real function with support disjoint from the spectrum of the restriction $F_\mathcal Q\in{\sf C}^*(\Xi_\mathcal Q,\o_\mathcal Q)^{\bf M}$\,.

\begin{enumerate}
\item[(i)] For every $\epsilon>0$ there is a neighborhood $W$ of $\mathcal Q$ in $X$ such that, setting $W_0:=W\cap M$, one has
\begin{equation*}\label{traznaie}
\|{\bf 1}_{W_0}\kappa(H_0)\|_{\mathbb B(\H_0)}\,\le\epsilon\,.
\end{equation*}
\item[(ii)] Suppose that $F$ is self-adjoint. Then for all $t\in\R$ and $u\in\H_0$, one also has
\begin{equation}\label{traznea}
\|{\bf 1}_{W_0}e^{itH_0}\kappa(H_0)u\|_{\H_0}\,\le\epsilon\|u\|_{\H_0}.
\end{equation}
\end{enumerate}
\end{Theorem}

\begin{Remark}\label{tosee}
\normalfont{To see the generic usefulness of \eqref{traznea}, suppose that one is interested in the unitary group
  $\big\{e^{itH_0}\!\mid\!t\in\R\big\}$\,, perhaps describing the evolution of a quantum system. The estimate holds for every "time"
  $t\in\R$ and any "state" $u\in\H_0$\,. We used terminology from Quantum Mechanics, having in mind the quantum interpretation of
  quantities as $\|{\bf 1}_{W_0}v_t\|^2_{\H_0}$ (a localization probability). If $H_0$ is the Hamiltonian of a quantum system, a
  condition of the form $\kappa(H_0)u=u$ means roughly that the vector (quantum state) $u$ have energies belonging to the support of the
  function $\kappa$\,, interpreted as a "localization in energy". In a typical case, $\kappa$ could be a good approximation of the characteristic function of a real interval $I$. So, finally, we
  expect \eqref{traznea} to hold because of some correlation between the energy interval $I$ and the region $W_0\subset M$ in which "propagation is improbable for time-dependent states
  $u_t=e^{itH_0}u$ with $u$ having energies only belonging to $I$".}
\end{Remark} 

Of course (ii) follows from (i), because $e^{itH_0}$ is a unitary operator and it commutes with $\kappa(H_0)$\,. We put \eqref{traznea}
into evidence because of its dynamical interpretation. So we need to show (i).

\smallskip
We start with three preliminary results; by combining them, Theorem \ref{main} will follow easily. The first one is abstract, it deals
with elements belonging to a $C^*$-algebra endowed with a distinguished ideal, and it reproduces \cite[Lemma 1]{AMP}.

\begin{Lemma}\label{lemma1}
Let $G$ be a normal element of a unital $C^*$-algebra $\mathscr A$ and let $\mathscr K$ be a closed self-adjoint two-sided ideal of
$\,\mathscr A$. Denote by ${\sf sp}_\mathscr K(G)$ the spectrum of the canonical image of $G$ in the quotient $\mathscr A/\mathscr K$. If
$\kappa\in C_0(\R)_\R$ and $\supp(\kappa)\cap{\sf sp}_\mathscr K(G)=\emptyset$\,, then $\kappa(G)\in \mathscr K$.
\end{Lemma}

The second preliminary result concerns a relationship between multiplication operators and elements of the twisted groupoid $C^*$-algebra.

\begin{Lemma}\label{lemma2}
\begin{enumerate}
\item[(i)] The (Abelian) $C^*$-algebra $C(X)$ of complex continuous functions on $X$ acts by double centralizers
\begin{equation*}\label{socoteala}
C(X)\times C_{\rm c}(\Xi)\ni(\psi,f)\to\psi\cdot f:=(\psi\circ\r)\,f\in C_{\rm c}(\Xi)\,,
\end{equation*}
\begin{equation*}\label{zocoteala}
C_{\rm c}(\Xi)\times C(X)\ni(f,\psi)\to f\cdot\psi:=(\psi\circ\d)\,f\in C_{\rm c}(\Xi)\,.
\end{equation*}
This action extends to the twisted groupoid $C^*$-algebra and leads to an embedding in the multiplier $C^*$-algebra: $C(X)\hookrightarrow
{\sf C}^*(\Xi,\o)^{\bf M}$\,.
\item[(ii)] The canonical extension $\Pi_0^{\bf M}$ of the vector representation to the multiplier algebra ${\sf C}^*(\Xi,\o)^{\bf M}$ acts on $C(X)$ by multiplication operators:
\begin{equation}\label{titiriseala}
\Pi_0^{\bf M}(\psi)u=\psi|_M\,u\,,\quad\forall \,\psi\in C(X)\,,\,u\in L^2(M)\,.
\end{equation}
\end{enumerate}
\end{Lemma}

\begin{proof}
The point (i) is straightforward; actually it is a particular case of \cite{Re}[II,Prop.\,2.4(ii)] (take there $\H:=X$ to be the closed subgroupoid of $\mathcal G=\Xi$)\,.

\smallskip
The point (ii) is also straightforward. One has to check, for instance, the following identity in $\mathbb B\big[L^2(M)\big]$\,:
\begin{equation}\label{desguise}
\Pi_0(\psi\cdot f)=\psi|_M\,\Pi_0(f)\,,\quad\forall\,\psi\in C(X)\,,\,f\in C_{\rm c}(\Xi)\,.
\end{equation}
Recall that $\r_z$ is the restriction of range map $\r$ to the $\d$-fiber $\Xi_z$, and it is a bijection $\Xi_z\to M$ if $z\in M$. We
can write for any $z\in M$, $\xi\in\Xi_z$ and $u\in L^2(M,\mu)\equiv L^2(M,\r_z(\lambda_z))$
\begin{equation*}
\begin{aligned}
\big[ (\psi\cdot f) \star_\o\! (u\circ \r_z) \big] (\xi)& = \int_\Xi\psi[\r(\eta)]f(\eta)\,u[\r(\eta^{-1}\xi)]\, \o \big( \eta,\eta^{-1} \xi \big)\, d\lambda^{\r(\xi)}(\eta)\\
& = \psi[\r(\xi)]\int_\Xi f(\eta)\, u[\r(\eta^{-1}\xi)]\, \o\big(\eta,\eta^{-1}\xi\big)\, d\lambda^{\r(\xi)}(\eta)\\
& = \big[ (\psi \circ \r_z) \big(f \star_\o \!(u \circ \r_z) \big)\big](\xi)\,,
\end{aligned}
\end{equation*}
which is \eqref{desguise} in desguise, by \eqref{ibidem}. We leave the remaining details to the reader.
\end{proof}

The third result essentially speaks of a bounded approximate unit.

\begin{Lemma}\label{lemma3}
Let $A\subset X$ be a closed invariant subset and $f\in{\sf C}^*\big(\Xi_{X\setminus A},\o_{X\setminus A}\big)\subset{\sf
  C}^*(\Xi,\o)$\,. For every $\epsilon>0$ there exists $\psi\in C(X)$ with $\psi(X)=[0,2]$\,, $\psi|_A=2$ and
\begin{equation}\label{estimam}
  \|\psi\cdot f\|_{{\sf C}^*(\Xi,\o)}+\|f\cdot\psi\|_{{\sf C}^*(\Xi,\o)}\,\le\epsilon\,.
\end{equation}
\end{Lemma}

\begin{proof}
Let us set $B:=X\!\setminus\!A$\,. By density, there exists $f_0\in C_{\rm c}\big(\Xi_B\big)$ such that $\|f-f_0\|_{{\sf C}^*(\Xi,\o)}\,\le\epsilon/4$\,. Set
\begin{equation*}
S_0:=\d[\supp(f_0)]\cup \r[\supp(f_0)]\,;
\end{equation*} 
since $\Xi_B=\Xi^B\!=\Xi^B_B$\,, the subset $S_0$ is compact and disjoint from $A$\,. So there is a continuous function
$\psi:X\to[0,2]$ with $\psi|_A=2$ and $\psi|_{S_0}=0$\,. In particular, $\psi\cdot f_0$ and $f_0\cdot\psi$ both vanish. Then
\begin{equation*}
\begin{aligned}
\|\psi\cdot f\|_{{\sf C}^*(\Xi,\o)}\!+\|f\cdot\psi\|_{{\sf C}^*(\Xi)}&=\,\|\psi\cdot(f-f_0)\|_{{\sf
    C}^*(\Xi,\o)}\!+\|(f-f_0)\cdot\psi\|_{{\sf C}^*(\Xi,\o)}\\ &\le 2\|\psi\|_\infty\|f-f_0\|_{{\sf C}^*(\Xi,\o)}\,\le\epsilon\,,
\end{aligned}
\end{equation*}
since $C(X)$ has been embedded isometrically in the multiplier algebra of ${\sf C}^*(\Xi,\o)$\,, by Lemma \ref{lemma2}.
\end{proof}

We can prove now Theorem \ref{main}, (i)\,.

\begin{proof}
We use the identification
\begin{equation*}
{\sf C}^*\big(\Xi_{\mathcal Q},\o_\mathcal Q\big)^{\bf M}\cong{\sf C}^*(\Xi,\o)^{\bf M}/{\sf C}^*\big(\Xi_{X\setminus\mathcal Q},\o_{X\setminus\mathcal Q}\big)\,,
\end{equation*}
implying the fact that the spectrum ${\sf sp}_{{\sf C}^*(\Xi_{X\setminus\mathcal Q})}(F)$ of the image of $F$ in the
quotient $C^*$-algebra coincides with ${\sf sp}\big(F_\mathcal Q\!\mid\!{\sf C}^*(\Xi_\mathcal Q,\o_\mathcal Q)^{\bf
  M}\big)$\,. Then, since we assumed that $\,\supp(\kappa)\cap{\sf sp}\big(F_\mathcal Q\big)=\emptyset$\,, one has $\kappa(F)\in{\sf
  C}^*\big(\Xi_{X\setminus\mathcal Q},\o_{X\setminus\mathcal Q}\big)$; this follows from Lemma \ref{lemma1}, with $\mathscr A\!:={\sf
  C}^*(\Xi,\o)^{\bf M}$, $\mathscr K\!:={\sf C}^*\big(\Xi_{X\setminus\mathcal Q},\o_{X\setminus\mathcal Q}\big)$\,.

\smallskip
Using Lemma \ref{lemma2}, the fact that morphisms commute with the functional calculus and the injectivity of $\Pi^{\bf M}_0$, for any $\psi\in C(X)$ one has
\begin{equation*}
\begin{aligned}
\|\psi|_M\,\kappa(H_0)\|_{\mathbb B(\H_0)}\,&=\,\|\Pi^{\bf M}_0(\psi)\,\kappa\big[\Pi^{\bf M}_0(F)\big]\|_{\mathbb B(\H_0)}\\ &=\,\|\Pi^{\bf
  M}_0\big[\psi\cdot\kappa(F)\big]\|_{\mathbb B(\H_0)}\\ &=\,\|\psi\cdot\kappa(F)\|_{{\sf C}^*(\Xi,\o)}.
\end{aligned}
\end{equation*}

Let us fix $\epsilon>0$\,. By Lemma \ref{lemma3} with $A=\mathcal Q$ and $f=\kappa(F)$\,, there is a continuous function $\psi:X\to[0,2]$ with $\psi|_\mathcal Q=2$ and
\begin{equation*}
\|\psi|_M\,\kappa(H_0)\|_{\mathbb B(\H_0)}\,=\,\|\psi\cdot\kappa(F)\|_{{\sf C}^*(\Xi,\o)}\,\le\epsilon\,.
\end{equation*}
Let us set $W:=\psi^{-1}(1,\infty)$\,; it is an open neighborhood of $\mathcal Q$ on which ${\bf 1}_W\le\psi$\,. In particular, ${\bf 1}_{W_0}={\bf 1}_W|_M\le\psi|_M$\,. Then
\begin{equation*}
\begin{aligned}
\|{\bf 1}_{W_0}\,\kappa(H_0)\|_{\mathbb B(\H_0)}&=\,\|\kappa(H_0){\bf 1}_{W_0}\kappa(H_0)\|^{1/2}_{\mathbb
  B(\H_0)}\\ &\le\,\|\kappa(H_0)(\psi|_M)^2\kappa(H_0)\|^{1/2}_{\mathbb B(\H_0)}\\ &=\,\|\psi|_M\,\kappa(H_0)\|_{\mathbb B(\H_0)}\,\le\epsilon
\end{aligned}
\end{equation*}
and the proof is over.
\end{proof}

\begin{Remark}\label{oscilez}
{\rm When applying Theorem \ref{main} one might want to take Remark \ref{doomsday} into consideration, in order to have a rather large class of elements to which the result applies and the restriction operation is explicit.  On the other hand, to $F\in{\sf C}^*(\Xi)^{\bf M}$ one can add a "potential" $V\in C(X)\subset{\sf
    C}^*(\Xi)^{\bf M}$, for which $\rho_\mathcal Q^{\bf M}(V)=V|_\mathcal Q$\,.}
\end{Remark}

\begin{Remark}\label{kinshasa}
\normalfont{We already know from Corollary \ref{carefully} that ${\sf sp}\big(F_\mathcal Q\big)$ is included (very often strictly) in
  the essential spectrum of the operator $H_0$\,. Thus, in Theorem  \ref{main}, the interesting case is
\begin{equation*}
\supp(\kappa)\subset{\rm sp}_{\rm ess}(H_0)\!\setminus\!{\sf sp}\big(F_\mathcal Q\big)\subset{\rm sp}(H_0)\!\setminus\!{\sf sp}\big(F_\mathcal Q\big)\,.
\end{equation*}
}
\end{Remark}

\subsection{Standard groupoids with totally intransitive groupoids at the boundary}\label{gorofarat}

We recall that {\it a totally intransitive groupoid} is a groupoid $\Xi \tto X$ for which the source and the range maps coincide
\cite[Definition 1.5.9]{Mac}. Then the groupoid can be written as the disjoint union of its isotropy groups.

\smallskip
Assume again that $\Xi$ is a standard groupoid over the unit space $X=M\sqcup X_\infty$, with open dense orbit $M$ having trivial isotropy. Also assume that "the restriction at infinity"
\begin{equation*}
  \Xi_{X_\infty} =: \Si=\bigsqcup_{n\in X_\infty}\!\Si_{n}
\end{equation*} 
is a totally transitive groupoid, where $\Si_{n}\!:=\Xi^n_n$ is the isotropy group of $n\in X_\infty$\,. We set
\begin{equation*}
  {\rm q}\!:={\rm d}_\infty={\rm r}_\infty:\Si\to X_\infty
\end{equation*} 
for the bundle map. Since $\Xi$ is standard, each $\Si_{n}$ is an amenable, second countable, Hausdorrff locally compact group.  Since,
by assumption, $\Xi$ has a (chosen) right Haar measure, this is also true for the closed restriction $\Si$\,; for every $n\in X_\infty$ the
fiber measure $\lambda_n$ is a right Haar measure on $\Si_{n}$\,.

\smallskip
The orbit structure is very simple: There is the main (open, dense) orbit $M$, and then all the points $n\in X_\infty$ form singleton
orbits by themselves: $\mathcal O_n=\mathcal Q_n=\{n\}$\,. The Hilbert space $\H_n$ is the $L^2$-space of the group $\Si_n$ with respect to
the Haar measure $\lambda_n$\,. We also take into consideration a $2$-cocycle $\o$ on $\Xi$\,. Its restriction to $\Xi_M$ is a
$2$-coboundary, by Remark \ref{idem}. The restrictions $\o_n\!:\Sigma_n^{(2)}\equiv\Sigma^2\to\mathbb T$ are usual group
$2$-cocycles, for all $n\in X_\infty$\,. The twisted groupoid $C^*$-algebra ${\sf C}^*\big(\Si_{n},\o_n\big)$ corresponding to the
quasi-orbit $\{n\}$ coincides now with the twisted group $C^*$-algebra of $\Si_n$\,.

\smallskip
Let uf fix a function $F$ belonging to $L^{\infty,1}_{\rm cont}(\Xi)\subset {\sf C}^*(\Xi,\o)$\,, cf. Remark \ref{doomsday}. One gets restrictions
\begin{equation*}
  F_\infty := \rho_{X_\infty}(F) = F|_{\Si}\in L^{\infty,1}_{\rm cont}(\Si)\,,
\end{equation*} 
\begin{equation*}
  F_{\{n\}} := \rho_{n}(F) = F|_{\Si_n}\in L^{\infty,1}_{\rm cont}(\Si_n) = L^1(\Si_n)\cap C(\Si_n)\,.
\end{equation*}
It is easy to see that the operators $H_n\!:=\Pi_n(F)$ given by regular representations are just twisted convolution operators:
\begin{equation}\label{hashen}
  [H_n(u)](a) = \int_{\Si_n}\!\!\o_n(ab^{-1}\!,b) F_{n}\big(ab^{-1}\big)u(b)d\lambda_n(b)\,,\quad u\in
  L^2(\Si_n;\lambda_n)\,,\ a\in \Si_n\,.
\end{equation}

Then one can easily write down {\it formulae for the essential spectrum and the essential numerical range of the operator
  $H_0=\Pi_0(F)\in\mathbb B\big[L^2(M)\big]$} in the vector representation, using the results from subsections \ref{ameny} and
\ref{caloriforit}. In addition, $H_0$ {\it is Fredholm if and only if all the operators $H_n$ are invertible}. We only include here only a
particular case, in which more can be said about each individual $H_n$\,. This is based on the following remark.

\begin{Remark}\label{mantorc}
\normalfont{Assume that $\o_n$ is trivial and the isotropy group $\Si_n$ is Abelian, with Pontryagin dual group
  $\widehat\Si_n$\,. The Fourier transform implements an isomorphism between the (Abelian) group $C^*$-algebra ${\sf C}^*(\Si_n)$ and the
  function algebra $C_0\big(\widehat\Si_n\big)$ of all complex continuous functions on $\widehat\Si_n$ decaying at infinity. The
  operator $H_n$ of convolution by $F|_{\Si_n}$ is unitarily equivalent to the operator of multiplication by the Fourier transform $\widehat{F|_{\Si_n}}$ acting in
  $L^2\big(\widehat\Si_n;\widehat\lambda_n\big)$\,, where $\widehat\lambda_n$ is a Haar measure on the dual group,
  conveniently normalized. The spectrum of this operator is simply the closure of the range of this function ${\sf R}_n\!:=\widehat{F|_{\Si_n}}\big(\widehat\Si_n\big)$\,.}
\end{Remark}

\begin{Corollary}\label{sugochev}
Suppose that the groupoid $\Xi$ is standard, with main orbit $M$ and with abelian totally intransitive groupoid at infinity. Then, using the notations above, one has
\begin{equation}\label{vaviloni}
  {\sf sp}_{\rm ess}(H_0)=\bigcup_{n\in X_\infty}\!\overline{{\sf R}_n}\,,
\end{equation}
\begin{equation}\label{vavilonie}
  {\sf nr}_{\rm ess}(H_0)=\sf co\Big(\bigcup_{n\in X_\infty}\!\!\overline{{\sf R}_n}\Big)\,.
\end{equation}
\end{Corollary}

\begin{proof}
One has to apply Corollary \ref{carefully}, Theorem \ref{carefree} (with obvious adaptations) and Remark \ref{mantorc}. The number ${\sf
  s}$ is zero in this situation, and it is already contained in the right hand sides. Since ${\sf C}^*(\Si_n)\cong
C_0\big(\widehat\Si_n\big)$ is Abelian, the operators $H_n$ are normal, so ${\sf nr}(H_n)={\sf co}[{\sf sp}(H_n)]={\sf co}(\overline{{\sf R}_n})$\,, leading to
\begin{equation}\label{vavilony}
  {\sf nr}_{\rm ess}(H_0)=\overline{\sf co}\Big(\bigcup_{n\in X_\infty}\!\!{\sf co}\big(\overline{{\sf R}_n}\big)\Big)\,.
\end{equation}
Then it is easy to see that the right hand side of \eqref{vavilonie} and \eqref{vavilony} coincide.

\smallskip
Actually, there is a direct way to deduce \eqref{vavilonie} from \eqref{vaviloni}. The quotient of $\Pi_0\big[{\sf C}^*(\Xi)\big]$
through the ideal of compact operators is Abelian, isomorphic to ${\sf C}^*\big(\Xi_{X_\infty}\big)$\,, so $H_0$ is essentially normal and
its essential numerical range is then the convex hull of its essential spectrum.
\end{proof}

One may write down the localization result of Section \ref{sec.five} for standard twisted groupoids with totally intransitive groupoids at
the boundary, but without further assumptions it is impossible to make explicit the neighborhoods $W$ of the singleton quasi-orbits
$\{n\}$\,. To get a transparent, still general situation, we are going to construct $X$ as a compactification of $M$ under the following

\begin{Hypothesis}\label{ipotenuza}
\normalfont{Let $M=M^{\rm in}\sqcup M^{\rm out}$ be a second countable Hausdorff locally compact space with topology $\mathcal T(M)$\,,
  decomposed as the disjoint union between a compact component $M^{\rm in}$ and a non-compact one $M^{\rm out}$. Let ${\sf p}:M^{\rm
    out}\!\to X_\infty$ be a continuous surjection to a second countable Hausdorff compact space $X_\infty$\,, with topology
  $\mathcal T(X_\infty)$\,. We are also going to suppose that no fibre $M^{\rm out}_n\!:={\sf p}^{-1}(\{n\})$ is compact. }
\end{Hypothesis}

Let us also denote by $\mathcal K(M)$ the family of all the compact subsets of $M$. We need a convention about complements: if $S\subset M^{\rm out}$, we are going to write
\begin{equation*}
\tilde S\!:=M^{\rm out}\!\setminus\!S\,,\quad S^c\!:=M\!\setminus\!S=M^{\rm in}\sqcup\tilde S\,.
\end{equation*}
For $E\in\mathcal T(X_\infty)$ and $K\in\mathcal K(M)$ one sets
\begin{equation*}
  A^M_{E,K} := {\sf p}^{-1}(E)\cap K^c\subset M^{\rm out}\subset M,
\end{equation*}
\begin{equation*}
  A_{E,K} := A^M_{E,K}\sqcup E\subset X\!:=M\sqcup X_\infty\,.
\end{equation*}
One has 
\begin{equation*}
  A^M_{E,K} = A_{E,K}\cap M,\quad E=A_{E,K}\cap X_\infty\,,
\end{equation*} 
as well as
\begin{equation*}
  A_{E_1,K_1}\cap A_{E_2,K_2}=A_{E_1\cap E_2,K_1\cup K_2}\,.
\end{equation*}
Let us set 
\begin{equation*}
  \mathcal A(X) := \big\{A_{E,K}\,\big\vert\,E\in\mathcal T(X_\infty)\,,\,K\in\mathcal K(M)\big\}\,,
\end{equation*} 
\begin{equation*}
\mathcal B(X)\!:=\mathcal T(M)\cup\mathcal A(X)\,.
\end{equation*} 
It is easy to check that $\mathcal B(X)$ is the base of a topology on $X$ (different from the disjoint union topology on $M\sqcup
X_\infty$)\,, that we denote by $\mathcal T(X)$\,. It also follows easily that $M$ embeds homeomorphically as an open subset of $X$ and
the topology $\mathcal T(X)$\,, restricted to $X_\infty$\,, coincides with $\mathcal T(X_\infty)$\,.

\begin{Lemma}\label{ctificaishn}
The topological space $(X,\mathcal T(X))$ is a compactification of $(M,\mathcal T(M))$\,.
\end{Lemma}

\begin{proof}
To show that $X$ is compact, it is enough to extract a finite subcover from any of its open cover of the form
\begin{equation*}
  \big \{O_\gamma \mid \gamma \in \Gamma \big \} \cup \big\{A_{E_\delta,K_\delta} \mid \delta \in \Delta\big\}\,,
\end{equation*}
where $O_\gamma\in \mathcal T(M)$\,, $E_\delta\in\mathcal T(X_\infty)$ and $K_\delta\in\mathcal K(M)$\,. Since $X_\infty$ is compact, one has
$X_\infty=\bigcup_{\delta\in\Delta_0}\!E_\delta$ for some finite subset $\Delta_0$ of $\Delta$\,. We show now that the complement of
the (open) set $\bigcup_{\delta\in\Delta_0}\!A^M_{E_\delta,K_\delta}$ in $M$ is compact. One has
\begin{equation*}
\begin{aligned}
& \Big[\bigcup_{\delta\in\Delta_0}\!A^M_{E_\delta,K_\delta}\Big]^c =\bigcap_{\delta\in\Delta_0}\!\big[{\sf p}^{-1}(E_\delta)\cap
    K_\delta^{c}\big]^c=\bigcap_{\delta\in\Delta_0}\!\big[{\sf p}^{-1}(E_\delta)^c\cup K_\delta\big]\,.
\end{aligned}
\end{equation*} 
A set of the form $\,\bigcap_{j=1}^m(R_j\cup S_j)$ can be written as $\,\bigcup\big(T_1\cap T_2\cap\dots\cap T_m\big)$\,, over all possible
$T_k\in\{R_k,S_k\}$\,, $\forall\,k=1,\dots,m$\,. In our case, each time at least one of the sets in an intersection is some $K_\delta$\,,
the intersection is already compact. There is also a contribution containing no set $K_\delta$\,:
\begin{equation*}
\begin{aligned}
\bigcap_{\delta\in\Delta_0}\!{\sf p}^{-1}(E_\delta)^c&= \Big[\bigcup_{\delta\in\Delta_0}\!{\sf p}^{-1}(E_\delta)\Big]^c=\Big[{\sf
 p}^{-1}\Big(\bigcup_{\delta\in\Delta_0}\!E_\delta\Big)\Big]^c\\ &=\big[{\sf p}^{-1}(X_\infty)\big]^c=(M^{\rm out})^c=M^{\rm in}.
\end{aligned}
\end{equation*}
So the complement of $\,\bigcup_{\delta\in\Delta_0}\!A^M_{E_\delta,K_\delta}$ in $M$ can be covered by a finite number of sets $O_\gamma$ and we are done.

\smallskip
The fact that $M$ is dense in $X$ is obvious, since any neighborhood of a point belonging to $X_\infty$ contains a set $A_{E,K}$\,, with
$E\ne\emptyset$\,, that meets $M$, because our fibers $M_n^{\rm out}$ are not compact.
\end{proof}

We could call $(X,\mathcal T(X))$ {\it the ${\sf p}$-compactification of} $(M,\mathcal T(M))$\,. It is meant to generalize the radial
compactification of a vector space $M$, for which $M^{\rm int}=\{0\}$ and $X_\infty$ is a sphere.

\begin{Remark}\label{descripshn}
\normalfont{A function $\phi:X\to\mathbb C$ is continuous with respect to the topology $\mathcal T(X)$ if and only if
\begin{itemize}
\item 
the restrictions $\phi|_M$ and $\phi|_{X_\infty}$ are continuous,
\item 
for every $n\in X_\infty$ and for every $\epsilon>0$\,, there exists $E\in\mathcal T(X_\infty)$ with $n\in E$ and $K\in\mathcal K(M)$ such
that $|\phi(m)-\phi(n)|\le\epsilon$ if ${\sf p}(m)\in E$ and $m\notin K$.
\end{itemize}
Let us set
\begin{equation*}\label{intind}
\begin{aligned}
  & C_{\sf p}(M)\!:=\\
  & \big \{\varphi\in C(M)\, \big\vert\, \exists\, M^{\rm in} \subset K \in \mathcal K(M)\ \,{\rm
    s.\,t.}\,\ \varphi(m)=\varphi(m')\ \,{\rm if}\,\,m,m'\notin K,\ {\sf p}(m) = {\sf p}(m')\big\}\,.
\end{aligned}
\end{equation*}
It is a unital $^*$-algebra consisting of bounded continuous functions that are asymptotically constant along all the fibers $M_n^{\rm
  out}$. It is not closed. Denoting by $C_{\mathcal T(X)}(M)$ the $C^*$-algebra formed of restrictions to the dense subset $M$ of all the elements of $C(X)$\,, one has
\begin{equation*}\label{nimic}
C(X)\cong C_{\mathcal T(X)}(M)\supset C_0(M)+C_{\sf p}(M)\,.
\end{equation*}}
\end{Remark}

Assume now, in the same context, that $\Xi$ is a standard groupoid over
\begin{equation*}
X=M\sqcup X_\infty=M^{\rm in}\sqcup M^{\rm out}\sqcup X_\infty
\end{equation*} 
such that the restriction $\Xi_{X_\infty} =: \Si=\bigsqcup_{n\in X_\infty}\!\Si_{n}$ is a totally intransitive groupoid, as
above. Now the setting is rich enough to allow a transparent form of Theorem \ref{main}.

\begin{Proposition}\label{amia}
Let $F\in L^{\infty,1}_{\rm cont}(\Xi)\subset{\sf C}^*(\Xi,\o)$ be a normal element and $n\in X_\infty$\,. Let $\kappa\in C_0(\mathbb R)$
be a real function with support that does not intersect the spectrum of the twisted convolution operator $H_n$ given in \eqref{hashen}.

\begin{enumerate}
\item[(i)] For every $\epsilon>0$ there is a neighborhood $E$ of $n$ in $X_\infty$ and a compact subset $K$ of $M$ such that
\begin{equation*}\label{trasnaie}
\big\Vert\,{\bf 1}_{\{m\notin K\mid {\sf p}(m)\in E\}}\,\kappa(H_0)\,\big\Vert_{\mathbb B(\H_0)}\le\epsilon\,.
\end{equation*}
\item[(ii)]
Suppose that $F$ is self-adjoint. One also has
\begin{equation*}\label{trasnea}
\big\Vert\,{\bf 1}_{\{m\notin K\mid {\sf p}(m)\in E\}}\,e^{itH_0}\kappa(H_0)u\,\big\Vert_{\H_0}\le\epsilon\|u\|_{\H_0}
\end{equation*}
uniformly in $t\in\R$ and $u\in\H_0$\,.
\end{enumerate}
\end{Proposition}

For the given point $n$\,, if Remark \ref{mantorc} applies, the condition on $\kappa$ reads $\supp \kappa\cap\overline{\sf R}_n=\emptyset$\,.

\end{document}